\documentclass[11pt]{amsart}
\usepackage{amssymb, adjustbox, enumerate, amsbsy, stmaryrd}
\usepackage{amsthm}
\usepackage{geometry}
\usepackage[all]{xy}
\usepackage{ulem}
\usepackage{todonotes}

\usepackage{amsfonts, amssymb, amscd}
\numberwithin{equation}{section}

\usepackage{bm}
\usepackage{verbatim}
\usepackage{mathrsfs}
\usepackage{graphicx}
\usepackage{tikz-cd}
\usepackage{subcaption}
\usepackage{listings}
\usepackage{subfiles}
\usepackage[toc,page]{appendix}
\usepackage{mathtools}
\usepackage{comment}
\usepackage{enumerate}
\usepackage{enumitem}
\usepackage[all]{xy}

\usepackage{graphicx}
\graphicspath{{images/}}

\usepackage{appendix}
\usepackage{hyperref}
\hypersetup{
    colorlinks=true,
    citecolor=red,
    linkcolor=blue,
    filecolor=magenta,      
    urlcolor=red,
}
\def\ideal#1.{I_{#1}}
\def\ring#1.{\mathcal {O}_{#1}}
\def\fring#1.{\hat{\mathcal {O}}_{#1}}
\def\proj#1.{\mathbb P(#1)}
\def\pr #1.{\mathbb P^{#1}}
\def\af #1.{\mathbb A^{#1}}
\def\Hz #1.{\mathbb F_{#1}}
\def\Hbz #1.{\overline{\mathbb F}_{#1}}
\def\pic#1.{\operatorname {Pic}\,(#1)}
\def\pico#1.{\operatorname{Pic}^0(#1)}
\def\picg#1.{\operatorname {Pic}^G(#1)}
\def\ner#1.{NS (#1)}
\def\rdown#1.{\llcorner#1\lrcorner}
\def\rup#1.{\ulcorner#1\urcorner}
\def\cone#1.{\operatorname {NE}(#1)}
\def\Ker{\operatorname {Ker}}
\def\ccone#1.{\overline{\operatorname {NE}}(#1)}
\def\coef#1.{\frac{(#1-1)}{#1}}
\def\vit#1.{D_{\langle #1 \rangle}}
\def\mm#1.{\overline {M}_{0,#1}}
\def\H1#1.{H^1(#1,{\ring #1.})}
\def\ac#1.{\overline {\mathbb F}_{#1}}

\def\adj#1.{\frac {#1-1}{#1}}
\def\spn#1.{\overline{#1}}
\def\ses#1.#2.#3.{0\to #1\to #2\to #3 \to 0}
\def\pek#1.#2.{\Cal P^{#1}(#2)}
\def\plk#1.#2.{\Cal P^{\leq #1}(#2)}
\def\ev#1.{\operatorname{ev_{#1}}}
\def\bminv#1.{(\nu_1,s_1;\nu_2,s_2;\dots ;\nu_{#1},s_{#1};\nu_{r+1})}
\def\zinv#1.{(\nu_1,s_1;\nu_2,s_2;\dots ;\nu_{#1},s_{#1};0)}
\def\iinv#1.{(\nu_1,s_1;\nu_2,s_2;\dots ;\nu_{#1},s_{#1};\infty)}
\def\map#1.#2.{#1 \longrightarrow #2}
\def\rmap#1.#2.{#1 \dasharrow #2}
\def\emb#1.#2.{#1 \hookrightarrow #2}

\def\Proj{\operatorname{Proj}}
\def\Supp{\operatorname{Supp}}

\def\dim{\operatorname{dim}}

\def\Pic{\operatorname{Pic}}

\def\Div{\operatorname{Div}}

\def\mult{\operatorname{mult}}

\def\bbeta{\boldsymbol{\beta}}
\def\ggamma{\boldsymbol{\gamma}}

\def\ddelta{\boldsymbol{\delta}}
\def\DDelta{\boldsymbol{\Delta}}

\def\C{\mathbb C}

\def\N{\mathbb N}

\def\Z{\mathbb Z}

\def\e{\Cal E}
\def\ZZ{\mathbb Z}

\def\e1{E_1}
\def\e2{E_2}

\def\OO{\mathcal O}

\newcommand{\po}{\ar@{}[dr]|{\text{\pigpenfont R}}}
\newcommand{\pb}{\ar@{}[dr]|{\text{\pigpenfont J}}}

\newcommand\Q{{\mathbb{Q}}}
\newcommand\R{{\mathbb{R}}}

\newtheorem{thm}{Theorem}[section]

\newtheorem{prop}[thm]{Proposition}

\newtheorem{claim}[thm]{Claim}

\theoremstyle{definition}

\theoremstyle{definition}

\newtheorem{theorem}{Theorem}[section]
\newtheorem{lemma}[theorem]{Lemma}
\newtheorem{proposition}[theorem]{Proposition}
\newtheorem{corollary}[theorem]{Corollary}
\newtheorem*{notation}{Notation ($\star$)}

\theoremstyle{definition}
\newtheorem{definition}[theorem]{Definition}

\newtheorem{remark}[theorem]{Remark}
\newtheorem{conjecture}[theorem]{Conjecture}

\begin{document}
\author{Christopher Hacon} \address{Department of Mathematics\\ University of Utah\\ 155 S 1400 E\\ Salt Lake City, Utah 84112} \email{hacon@math.utah.edu}
\thanks{Christopher Hacon was partially supported by the NSF research grants no:  DMS-2301374, DMS-2601582 and by a grant from the
Simons Foundation, SFI-MPS-MOV-00006719-0. The authors would like to thank M. Paun, V. Tosatti, W. Ou, and O. Das for many useful suggestions. In particular several arguments in this paper are inspired by joint work of O. Das and the first author, especially \cite{DH26}.}
\author{Lingyao Xie} 
\address{School of Mathematical Sciences, Zhejiang University, Hangzhou, China}
\email{xielingyao@zju.edu.cn}
\address{Department of Mathematics, University of California San Diego, 9500 Gilman Drive
 0112, La Jolla, CA 92093-0112, USA} \email{l6xie@ucsd.edu}

\date{\today}
\title[On the K\"ahler MMP]{On the K\"ahler MMP and the transcendental base-point-free theorem}

\begin{abstract}
    In this article,   we establish the minimal model program for big gklt K\"ahler pairs, and in particular we prove Tosatti's transcendental base-point-free conjecture. \end{abstract}
    \maketitle
\tableofcontents
\section{Introduction}    
The purpose of this paper is to extend the celebrated results of \cite{BCHM10} to the context of K\"ahler varieties, that is, we  prove the following.
\begin{theorem}\label{t-KahlerBCHM}
Let $(X,B+\bbeta)$ be a compact K\"ahler gklt pair.
\begin{enumerate}
\item If $K_X+B+\bbeta _X$ is not pseudo-effective, then there is a $(X,B+\bbeta)$ Mori fiber space.
\item If $K_X+B+\bbeta _X$ is pseudo-effective and $B+\bbeta _X$ is big (or if $K_X+B+\bbeta _X$ is big), then $(X,B+\bbeta )$ has a good log terminal model.\end{enumerate}
\end{theorem}

Note that the above result is proved in \cite{BCHM10} when $X$ is projective and $\bbeta _X=0$. The case when $X$ is projective and $\bbeta $ is a nef b-divisor (i.e. the case of traditional generalized klt pairs) also follows easily from the results of \cite{BCHM10}.
Theorem \ref{t-KahlerBCHM} is proved in full generality in the projective case in \cite{DH24}.
The standard approach to proving Theorem \ref{t-KahlerBCHM} is to prove the cone theorem, the transcendental base-point-free theorem, the existence of flips and the termination of flips with scaling. Two of these results have been previously established.
\begin{theorem}[Existence of flips {\cite[Theorem 5.12]{DH23}}]\label{t-flip}
Let $(X,B+\bbeta )$ be a $\Q$-factorial, compact, K\"ahler, gklt pair, and $f:X\to Z$  a flipping contraction, then the flip $f^+:X^+\to Z$ exists.
\end{theorem}
\begin{theorem}[Cone theorem {\cite[Theorem 3.1, Corollary 3.2]{HP24}, \cite[Corollary 3.6]{HLX26}}]\label{t-cone}
Let $(X,B+\bbeta )$ be a compact, K\"ahler, gklt pair,  then there are countably many rational
curves $\{\Gamma _i\}_{i\in I}$ such that $0<-(K_X +B +\bbeta _X)\cdot \Gamma _i\leq  2\dim X$ for all $i\in I$ and \[\overline{\rm NA}(X)=\overline{\rm NA}(X)_{(K_X +B +\bbeta _X)\geq 0}+\sum \R^+[\Gamma _i].\]
If $B+\bbeta _X$ is big then $I$ is finite.
\end{theorem}

The main contribution of this paper is to prove the following version of the transcendental base-point-free conjecture (see \cite[Conjecture 1.2]{FT18}, \cite{Tos24}) and of the termination of flips with scaling.
\begin{theorem}
 [Transcendental base-point-free theorem]\label{t-bpf}
    Let $(X,B+\bbeta)$ be a compact, $n$-dimensional, K\"ahler, gklt pair such that  $\alpha=K_X+B+\bbeta _X$ is nef, and $B+\bbeta_X$ is modified big. Then $\alpha$ is semiample, more precisely, there exists a (strongly) Moishezon contraction $f: X\to Y$ such that
    \begin{enumerate}
    \item $Y$ is a compact K\"ahler space, and
    \item $\alpha\equiv f^*\gamma$ for some K\"ahler class $\gamma$ on $Y$.
    \end{enumerate}
\end{theorem}
In particular, if $\alpha=K_X+\omega$ where $\omega$ is K\"ahler, then $-K_X$ is actually $f$-ample and hence $f$ projective by Lemma \ref{lem: Moishezon to strong Moishezon} below.

\begin{theorem}[MMP with scaling]\label{t-MMPscaling} Let $(X,B+\bbeta )$ be a compact, $n$-dimensional, K\"ahler, strongly $\Q$-factorial, gklt pair such that $B+\bbeta _X$ is big, and $\omega$ a modified K\"ahler form such that $K_X+B+\bbeta_X +\omega$ is nef. Then we can run the $K_X+B+\bbeta _X$ minimal model program with scaling of $\omega$ which terminates with either a good log terminal model or a Mori fiber space.
\end{theorem}

\subsection{The structure of the proof}
In this section we set up the inductive proof of our main theorems. Note that we will say that Theorem 1.X$_{n}$ holds to mean that Theorem 1.X holds in dimension $\le n$.
\begin{theorem}
[Contraction theorem]\label{t-nkltcontraction}
Let $(X,B+\bbeta)$ be a $n$-dimensional compact K\"ahler generalized pair such that $B+\bbeta_X$ is modified big and $\alpha=K_X+B+\bbeta _X$ is nef and NQC. Assume that the contraction defined by $[\alpha]|_N$ exists and is Moishezon, where $N={\rm Nklt}(X,B+\bbeta)$ is the closed subspace defined by the multiplier ideal $\mathcal{J}(X,B+\bbeta)$. 

Then the contraction defined by $[\alpha]$ exists and is Moishezon.
\end{theorem}

Note that Theorems \ref{t-bpf}$_1$, \ref{t-MMPscaling}$_1$, \ref{t-nkltcontraction}$_1$ are trivially seen to hold (i.e. these results hold in dimension 1).
In this paper we will prove these theorems in arbitrary dimension via the following induction. We will need to distinguish two separate cases of Theorem \ref{t-nkltcontraction}$_{n}$ depending on whether $\alpha$ is big or not. Therefore we will indicate the $\alpha$ is big case of this theorem by  Theorem \ref{t-nkltcontraction}$_{n,{\rm big}}$. 
\begin{itemize}
\item Theorem \ref{t-nkltcontraction}$_{n-1}$
implies Theorem \ref{t-nkltcontraction}$_{n,{\rm big}}$  (see Section \ref{s-contr-big}).

\item Theorems \ref{t-bpf}$_{n-1}$, \ref{t-MMPscaling}$_{n-1}$, \ref{t-nkltcontraction}$_{n,{\rm big}}$
imply Theorem \ref{t-bpf}$_{n}$ (see Section \ref{s-semiample}).

\item Theorems \ref{t-bpf}$_{n}$, \ref{t-MMPscaling}$_{n-1}$ 
imply Theorem \ref{t-MMPscaling}$_{n}$ (see Section \ref{s-MMPS}).

\item Theorems \ref{t-bpf}$_{n}$, \ref{t-MMPscaling}$_{n}$, \ref{t-nkltcontraction}$_{n-1}$ 
imply 
Theorem \ref{t-nkltcontraction}$_{n}$ (see Section \ref{s-contr-nonbig}).

\end{itemize}
In order to show the third implication, namely that Theorems \ref{t-bpf}$_{n}$, \ref{t-MMPscaling}$_{n-1}$  
imply Theorem \ref{t-MMPscaling}$_{n}$, it is necessary to show that the minimal model program with scaling terminates. Similarly to \cite{BCHM10}, this is achieved by proving a theorem on the boundedness of weak log canonical models as the 'boundary'  $B+\bbeta$ varies in a compact set (see Theorem \ref{t-MMPgeography}).

We note that in order to run the minimal model program with scaling for K\"ahler gklt pairs, it is necessary to check that the K\"ahler condition is preserved by steps of the minimal model program (see \cite{HLX26}).
\begin{theorem} \label{t-bigKahler} 
Let $X$ be a compact Fujiki class $\mathcal C$ variety of dimension $n$,  $(X,B+\bbeta)$ a gklt pair, where $\alpha:=[K_X+B+\bbeta _X]\in H^{1,1}_{\rm BC}(X)$ is nef and big. If $\alpha \cdot C>0$ for all rational curves $C\subset X$, then $\alpha$ is K\"ahler. 
\end{theorem}
\begin{corollary}[See Proposition \ref{prop: MMP can be run}]\label{c-contractions}
Let $(X,B+\bbeta )$ be a strongly $\Q$-factorial, compact, K\"ahler, gklt pair, where $B+\bbeta_X$ is big. If $f:X\to Y$ is a flipping or divisorial contraction then $Y$ is K\"ahler.
\end{corollary}
\subsection{Sketch of the proof}
Throughout the proof we will study gklt pairs $(X,B+\bbeta )$ where $B+\bbeta _X$ is modified big and $X$ is compact K\"ahler. Our arguments will deal separately with the following two alternatives:
\begin{enumerate}
\item $K_X+B+\bbeta_X$ is not big, and
\item $K_X+B+\bbeta_X$ is big.
\end{enumerate}
The main crux of our inductive argument is to prove the transcendental base-point-free theorem (Theorem \ref{t-bpf}) by induction on the dimension. For technical reasons that we explain below, it is necessary to prove a non-gklt version of this result (see Theorem \ref{t-nkltcontraction}).

The key to understanding case (1), is the following fundamental result of Wenhao Ou.
\begin{theorem}[\cite{Ou25}]\label{t-ou} Let $X$ be a compact K\"ahler manifold. Then $X$ is uniruled if and only if the
canonical line bundle is not pseudoeffective.
\end{theorem} 
Since $B+\bbeta _X$ is modified big, then (1) implies that $K_X$ is not pseudoeffective and so $X$ is uniruled. Therefore, we may consider the maximally rationally connected fibration (MRC) $X\dasharrow Y$. After resolving singularities, we may assume that $X\to Y$ is a projective morphism with rationally connected fibers. Running a relative log terminal model, we obtain a Mori fiber space. This means that after a finite sequence of flips and divisorial contractions over $Y$ (which we suppress for convenience), we obtain a $K_X+B+\bbeta _X$ Mori fiber space $f:X\to Z$ over $Y$. By the canonical bundle formula (Theorem \ref{t-gcbf}) we have $K_X+B+\bbeta _X\equiv f^*(K_Z+\Delta+\ddelta _Z)$ where $(Z,\Delta +\ddelta )$ is a gklt pair, $\Delta +\ddelta _Z$ is modified big and $Z$ is in Fujiki's class $\mathcal C$ (and hence birational to a K\"ahler variety, but for simplicity we suppose here that $Z$ is K\"ahler). If $K_X+B+\bbeta _X$ is nef, then so is $K_Z+\Delta +\ddelta _Z$. Since   $\dim Z<\dim X$, we can now conclude by induction on the dimension that $K_Z+\Delta +\ddelta _Z$ is semiample and hence so is $K_X+B+\bbeta _X$.

The key to understanding case (2) is the following result of \cite{HP24} and \cite{CT15}.
\begin{theorem}[Theorem 4.22 \cite{HP24}]\label{thm:null-non-kahler}
       Let $X$ be a compact  analytic variety in Fujiki's class $\mathcal C$, and $\alpha$ a smooth (1,1)-form, which is locally $\partial\bar\partial$-exact and such that the corresponding class is nef and big. Then $E_{nK}^{as}(\alpha)= {\rm Null}(\alpha)$. In particular,
the set ${\rm Null}(\alpha)$ is analytic.   
\end{theorem}
Therefore, if $\alpha =K_X+B+\bbeta _X$ is nef and big where $(X,B+\bbeta _X)$ is a gklt pair, we consider a K\"ahler current $\psi \in[K_X+B+\bbeta _X]$ with singularities concentrated along ${\rm Null}(\alpha)$.
The key fact here is that we may assume that $\psi$ has weak analytic
singularities (see \cite[Definition 4.11]{HP24}) and that the singular set $E_+(\psi)$ is a closed, analytic subset
of $X$ (\cite[Lemma 4.16]{HP24}) containing 
${\rm Null}(\alpha)$.
Since $\psi$ has weak analytic
singularities, replacing $\psi$ by another locally exact current with $[\psi]=\alpha$, we may assume that there is a resolution $\nu :X'\to X$ such that $\nu ^*\psi \equiv D'+\eta '$ where $D'$ is an effective $\R$-divisor and $\eta'$ is a K\"ahler form.
If $E_+(\psi)=\emptyset$,
then $\alpha $ is K\"ahler and hence $K_X+B+\bbeta _X$ is semiample. Otherwise we consider the log canonical threshold $\lambda$ for $(X,B+\bbeta _X)$ with respect to $\psi$. 
Therefore $(X,B+\bbeta +\lambda \psi)$ has a minimal log canonical center say $Z$ and by adjunction we have $(K_{X}+B+\bbeta +\lambda \psi)|_Z=K_Z+\Delta +\ddelta _Z$ where $(Z, \Delta +\ddelta _Z)$ is a gklt pair. It is easy to arrange that $\Delta +\ddelta _Z$ is modified big and hence by induction on the dimension $K_Z+\Delta +\ddelta _Z$ is semiample. 
We now repeat this procedure by considering successive jumping numbers for the multiplier ideals $\mathcal J(X,B+\bbeta +t\psi)$ so that we obtain a sequence of numbers $\lambda _1=\lambda <\lambda _2<\lambda _3<\ldots $
and multiplier ideals \[\mathcal I_Z=\mathcal J(X,B+\bbeta +\lambda _1\psi)\supset \mathcal I_{Z_2}=\mathcal J(X,B+\bbeta +\lambda_2 \psi)\supset \mathcal I_{Z_3}=\mathcal J(X,B+\bbeta +\lambda_3 \psi)\supset \ldots \]
Proceeding by induction, we assume that $(K_{X}+B+\bbeta +\lambda \psi)|_{Z _i}$ is semiample (note however that $Z_i$ is typically not normal and even not reduced). Suppose that $Z_{i+1}=Z'\cup Z_i$, then we expect that \[(1+\lambda _{i+1})\alpha|_{Z'}=(K_{X}+B+\bbeta +\lambda_{i+1} \psi)|_{Z '}=K_{Z'}+\Delta _{Z'}+\ddelta _{Z'}\] where $(Z',\Delta _{Z'}+\ddelta _{Z'})$ is a generalized pair which is gklt on the complement of $Z'\cap Z_i$. Since $\alpha |_{Z_i}$ is semiample, then so is $\alpha|_{Z'\cap Z_i}$. Therefore, by an induction on the dimension, using the non-gklt version of the transcendental base-point-free theorem (see Theorem \ref{t-nkltcontraction}), we see that $\alpha |_{Z'}$ is semiample. It remains to glue the corresponding morphisms for $Z',Z_i$ along $Z'\cap Z_i$ to obtain a morphism defined on $Z_{i+1}$ that shows that $\alpha |_{Z_{i+1}}$ is semiample.

Repeating this argument we see that $\alpha |_{\mathcal N}$ is semiample, where $\mathcal N$ denotes the formal neighborhood of $E_+(\psi)$ in $X$. Let $f:\mathcal N\to \mathcal W$ be the corresponding morphism. Using ideas of \cite{Ar70} and \cite{Fuj75}, we then extend this morphism to $F:X\to Y$ where $F|_{\mathcal N}=f$ and $F|_{X\setminus \mathcal N}={\rm id}_{X\setminus \mathcal N}$.

We also note that an important ingredient in our arguments is the canonical
bundle formula, see \S \ref{s-canbundle}. 
\begin{remark} Note that there is a close connection between running the minimal model program with scaling and the K\"ahler-Ricci flow. We refer the reader to \cite{ST17} for the details. We also refer the reader to \cite[Theorem C]{EGZ09} and \cite{BBEGZ19} for results about the existence of K\"ahler-Einstein metrics.
\end{remark}
\begin{theorem}Conjectures 4.12 and 4.13 of \cite{Tos18} hold.
\end{theorem} 
\begin{proof} Conjecture 4.13 of \cite{Tos18} is an immediate consequence of Theorem \ref{t-bpf}. Conjecture 4.12 of \cite{Tos18} then follows by Proposition 4.14 of \cite{Tos18}.
\end{proof}
\begin{theorem}Conjectures 1.1, 1.2, and 1.3 of \cite{TZ18} hold. 
\end{theorem} 
\begin{proof} Conjecture 1.2 of \cite{TZ18} is an immediate consequence of Theorem \ref{t-bpf}. Conjecture 1.1 of \cite{TZ18} then follows by Theorem 2.3, of \cite{TZ18}. Conjecture 1.3 of \cite{TZ18}  now follows from \cite[Theorem 3.2]{TZ18}. Note that Conjecture 1.3 of \cite{TZ18} is first proven in \cite{Zha25}.
\end{proof}
\subsection{Open problems}
We expect that the main results of this paper should also hold for gklt (or even glc) pairs where $B+\bbeta_X$ is not assumed to be modified big. In particular we have the following.
\begin{conjecture}(Existence of log terminal models)\label{c-min} Let $(X,B+\bbeta )$ be a compact K\"ahler gklt pair such that $K_X+B+\bbeta _X$ is pseudo-effective, then $(X,B+\bbeta )$ has a log terminal model.
\end{conjecture}

We note that Proposition \ref{prop: MMP can be run} implies that one can run the MMP for any compact K\"ahler dlt pair (with scaling), therefore Conjecture \ref{c-min} would follow from the following termination conjecture.
 
\begin{conjecture}(Termination of flips)\label{c-term} Let $(X,B+\bbeta )$ be a compact K\"ahler gklt pair, then any sequence of $K_X+B+\bbeta _X$ flips is finite.
\end{conjecture}

Next we recall our famous Abundance conjecture in the K\"ahler setting. 

\begin{conjecture}(Abundance)\label{c-ab} Let $(X,B)$ be a compact K\"ahler klt pair such that $K_X+B$ is nef, i.e. being non-negative on $\overline{\rm NA}(X)$, or equivalently, a limit of K\"ahler classes.

Then $K_X+B$ is semiample in the following senses:
\begin{enumerate}
\item(as an $\R$-line bundle) There is a proper contraction $f: X\to Z$ such that $K_X+B=f^*A$ for some ample divisor $A$ on $Z$.
\item(as a class in $H^{1,1}_{\rm BC}(X)$) There is a proper contraction $f: X\to Z$ such that $[K_X+B]\equiv f^*\gamma$ for some K\"ahler class $\gamma$ on $Z$.
\end{enumerate}
\end{conjecture}

We remark that Conjecture \ref{c-ab}(1) is known in dimension $\leq 3$ (see \cite{DO24}, \cite{DO26} and references therein), however Conjecture \ref{c-ab}(2) is known to fail for generalized pairs (see \cite[Example 6.2]{LP20}).

\section{Preliminaries}
We will follow the usual conventions of the minimal model program. In
particular, we refer the reader to \cite[Chapter 2]{KM98} for the definition of
pairs and their singularities (klt, lc, plt etc.), \cite[Subsection 2.1]{DHY23} for
the definitions of generalized pairs and related singularities (glc, gklt, gdlt,
etc.), \cite[Chapters 2, 3, 11]{Fuj22} for many key concepts of the minimal model
program for analytic varieties, to \cite[Definition 2.8]{DHP24} for the definitions of
nef, minimal and log terminal models, \cite{HP16} for a discussion of Bott-Chern
cohomology $H^{1,1}_
{\rm BC}(X)$, the K\"ahler and nef cones $\mathcal K\subset \overline{\mathcal K}$, the Mori cone and
the cone of positive closed currents $ {\rm NA}(X) \subset  N^1(X)$.  We say that $\alpha \in H^{1,1}_{\rm BC}(X)$ is {\it semiample} if there is a holomorphic map $f:X\to Z$ of compact analytic varieties and a K\"ahler class $\gamma $ on $Z$ such that $\alpha \equiv f^*\gamma$. In this paper, when we say $\alpha\equiv\alpha'$, it means that they equal as classes in $H^{1,1}_{\rm BC}(X)$, and we say 
\[
\alpha\equiv_f\alpha' \text{ or } \alpha\equiv_Y\alpha'
\]
if they equal up to a pullback $f^*\beta$ for some $\beta\in H^{1,1}_{\rm BC}(Y)$, where $f: X\to Y$ is usually a proper morphism between complex spaces. We say $\alpha$ is nef (resp. big) over $Y$ if $\alpha\equiv_Y\alpha'$ for some nef (resp. big) class $\alpha'\in H^{1,1}_{\rm BC}(X)$. If $f$ is projective and $\alpha$ is the class of a line bundle, then (at least locally) this matches the usual definition of relatively nef (resp. big).

We refer the reader to \cite[Definition 2.11]{DHY23} for the definition of log minimal models, log terminal models, and (weak) log canonical models. Note that if $(X,B+\bbeta)$ is gklt then every log minimal model is also a log terminal model.
If $(X,B+\bbeta)$ is a log terminal model and $K_X+B+\bbeta _X$ is semiample, then we say that $(X,B+\bbeta)$ is a good log terminal model.
We say that $X$ is strongly $\Q$-factorial if every divisorial sheaf $\mathcal F$ (i.e. reflexive sheaf of rank 1) is $\Q$-Cartier (i.e. $\mathcal F^{[m]}$ is a line bundle for some integer $m>0$) see \cite[Defn.2.2(ix)]{DH23}. 
Given a divisor $D$ we will write $D=D^{\geq 0}-D^{\leq 0}$ where $D^{\geq 0},D^{\leq 0}$ are effective with no common components. If $\nu:X'\to X$ is a birational morphism we will often write ${\rm Ex}(\nu )$ to denote the sum of the exceptional divisors of $\nu$ (with coefficient 1). We refer the reader to \cite{Bou04} for the definition of the Boucksom-Zariski (or BZ) decomposition $\alpha=N(\alpha)+P(\alpha)$ of a pseudo-effective class $\alpha \in H^{1,1}_{\rm BC}(X)$. 

Throughout this paper we will interchangeably use the terms bimeromorphic map and birational map, meromorphic map and rational map, holomorphic map and morphism.
\subsection{Multiplier ideals} We recall here the definition of multiplier ideal sheaf for a generalized pair. Since we are in the singular setting, this is somewhat non-standard.
\begin{definition} Let $(X,B+\bbeta )$ be a generalized pair and $\nu:X'\to X$ a log resolution such that $\bbeta$ descends on $X'$ and $\nu ^{-1}_*B+{\rm Ex}(\nu)$ is a divisor with simple normal crossings. We write $K_{X'}+B_{X'}+\bbeta _{X'}=\nu ^*(K_X+B+\bbeta _X)$ and we define 
\[\mathcal J(X,B+\bbeta)=\nu _*\OO _{X'}(-\lfloor B_{X'}\rfloor).\]It is easy to see that $(X,B+\bbeta )$ is gklt iff $\mathcal J(X,B+\bbeta)=\OO _X$.

\end{definition}
One reason why multiplier ideals are especially useful is Nadel vanishing.

\begin{theorem}[Nadel vanishing]\label{t-nadel} Let $(X,B+\bbeta)$ be a generalized pair, $g:X\to T$ a projective morphism  of
analytic varieties, and $D$ a Cartier divisor such that $D-(K_X+B+\bbeta _X)$ is nef and big over $T$. Then $R^ig_*(\OO_X(D)\otimes \mathcal J(X,B+\bbeta ))=0$ for any $i>0$.
\end{theorem}
\begin{proof} Let $\nu :X'\to X$ be a log resolution such that $\bbeta$ descends and $\nu ^{-1}_*B+{\rm Ex}(\nu)$ is a divisor with simple normal crossings. We write $K_{X'}+B_{X'}+\bbeta _{X'}=\nu ^*(K_X+B+\bbeta _X)$.
Note that \[\nu ^*D-\lfloor B_{X'}\rfloor \equiv_T K_{X'}+\{B_{X'}\}+\bbeta _{X'} +\nu ^*(D-K_X-B-\bbeta _X).\]
Therefore $\nu ^*D-\lfloor B_{X'}\rfloor-(K_{X'}+\{B_{X'}\})$ is nef and big over $T$, hence also over $X$.
By \cite[Theorem 3.7]{Nak87}, we have
\[
R^j\nu _* \OO _{X'}(\nu ^*D-\lfloor B_{X'}\rfloor)=0,~R^j(g\circ\nu)_* \OO _{X'}(\nu ^*D-\lfloor B_{X'}\rfloor)=0,~j>0
\]
and hence by the projection formula and Leray spectral sequence, we have \[R^ig_* \mathcal (\OO _X(D)\otimes \mathcal J(X,B+\bbeta))=R^i(g\circ \nu) _* \OO _{X'}(\nu ^*D-\lfloor B_{X'}\rfloor)=0\]
for all $i>0$.
\end{proof}

It follows that the relative Kawamata-Viehweg
vanishing theorem holds for generalized pairs (cf. \cite[Theorem 2.21]{DH25}).
\begin{theorem}[Relative Kawamata-Viehweg vanishing]\label{t-kvv} Let 
$(X, B +\bbeta)$ be a gklt pair, $g : X \to T$ a projective morphism
of analytic varieties, and $D$ a $\Q$-Cartier $\mathbb Z$-divisor such that $D-(K_X +B+\bbeta _X )$
is nef and big over $T$. Then $R^ig_*\OO_X (D) = 0$ for all $i > 0$.\end{theorem}
\begin{proof}
When $D$ is Cartier this is a special case of Theorem \ref{t-nadel}. In general, this follows similarly as the above resolution arguments by noticing that $\nu_*\OO_{X'}(\lfloor \nu^*D\rfloor)=\OO_X(D)$.
\end{proof}

The following lemma will be used several times in Section \ref{s-contr-big}.

\begin{lemma}\label{l-ses} Let $g:X\to T$ be a projective morphism from a complex manifold to a compact analytic variety, $E,Z$ effective divisors with no common components, $\Delta$ an $\R$-divisor such that $\lfloor \Delta \rfloor =0$ and ${\rm Supp}(E+Z+\Delta)$ has simple normal crossings. Suppose that \[E-Z\equiv_{g} K_X+\Delta +\gamma \] where $\gamma \in H^{1,1}_{\rm BC}(X)$ is nef and big over $T$, and $g_*\OO_X(E)=\OO_T$. 

Then $g_*\OO _X(-Z)\to g_*\OO _X(E-Z)$ and $g_*\OO _Z\to g_* \OO _Z(E|_Z)$ are isomorphisms and we have a short exact sequence \[ 0\to g_*\OO _X(-Z)\to \OO _T\to g_*\OO _Z\to 0.\]

\end{lemma}
\begin{proof} Since $g_*\OO _X(E-Z)\subset g_*\OO _X(E)=g_*\OO_X=\OO_T$, then $g_*\OO _X(E-Z)=g_*\OO _X(-Z)$.
Consider the short exact sequence \[0\to \OO _X(E-Z)\to \OO_X(E) \to \OO _Z(E) \to 0.\]
By Theorem \ref{t-kvv}, $R^1g_* \OO _X(E-Z)=0$ and so we have a short exact sequence 
\[0\to g_* \OO _X(E-Z)\to \OO_T \to g_* \OO _Z(E) \to 0.\] 
In particular the following composition $\OO_T\to g_*\OO _Z\to g_* \OO _Z(E) $ is surjective.
Since $E\wedge Z=0$, $\OO _Z\to \OO _Z(E)$ is injective and hence so is $g_*\OO _Z\to g_*\OO _Z(E)$.
\end{proof}

\subsection{The negativity lemma}
For the convenience of the reader, we will now recall the following useful lemmas that will be used throughout the paper.
\begin{lemma} [Negativity Lemma]\label{l-negbir} Let $f:X\to Y$ be a proper bimeromorphic morphism of normal complex varieties. Let $E$ be a $\R$-Cartier divisor on $X$ such that $-E$ is $f$-nef. Then $E\geq 0$
if and only if $f_*E\geq 0$.
\end{lemma}\begin{proof} See \cite[Lemma 1.3]{Wan21}.
\end{proof}
\begin{definition} Let $f :X \to Y$  be a proper surjective morphism of normal
varieties and $D$ a $\R$-divisor. Then
\begin{enumerate}\item $D$ is {\it $f$-exceptional} if ${\rm codim}(f({\rm Supp}(D)))\geq 2$.
\item $D$ is {\it $f$-degenerate} if for every prime divisor $Q\subset f({\rm Supp}(D))$ there exists a divisor $P$ not contained in ${\rm Supp}(D)$ such that $f(P)=Q$. \end{enumerate}
\end{definition}
\begin{lemma}\label{l-neg} Let $f :X \to Y$ be a proper morphism of normal analytic varieties, where $X$ is a K\"ahler space, $f_*\OO_X = \OO_Y$ and $Y$ is relatively compact. Let
$E = \sum a_iE_i \leq 0$ be an $f$-degenerate $\R$-Cartier divisor such that $-E$ is $f$-nef.
Then $E = 0$.
\end{lemma}
\begin{proof} See \cite[Lemma 2.8]{DH24}. 
\end{proof}

\subsection{Big classes}
\begin{definition}If $X$ is a compact normal K\"ahler variety, then we say that $\alpha\in H^{1,1}_{\rm BC}(X)$ is {\it big} if it admits a K\"ahler current $T$, i.e. $[T]\in \alpha$ and $[T]\ge\omega$ for some K\"ahler form $\omega$. If $f:X\to Y$ is a birational morphism of normal varieties, then we say that $f_*\alpha $ is {\it modified big}. Note that $f_*\alpha $ may not be locally exact, however if $f_*\alpha\in H^{1,1}_{\rm BC}(Y)$ and $Y$ is K\"ahler, then $f_*\alpha$ is big by pushing forward $[T]-\omega+(\omega-f^*\omega_Y)+f^*\omega_Y$, where $\omega_Y$ is some small K\"ahler form on $Y$. We can also define bigness when $X$ is not K\"ahler by replacing $\omega$ with a positive hermitian form, but we will not use that in this paper.
\end{definition}
\begin{remark} Recall that $\alpha\in H^{1,1}_{\rm BC}(X)$ is pseudo-effective if it admits a positive current, therefore $\alpha +\beta$ is big if $\alpha $ is big and $\beta$ is pseudo-effective. Note that if $\nu :X'\to X$ is (projective) birational then $\nu ^*\alpha$ is big iff $\alpha$ is. To see this, consider $\mu:X''\to X'$ a resolution such that $f:=\nu \circ \mu$ is given by a sequence of blow ups along smooth centers. Hence for any K\"ahler form $\omega$ on $X$, we may assume that there exists an effective exceptional divisor $F$ and a K\"ahler form $\omega''$ on $X''$ such that $[\omega'']=[f^*\omega]-[F]$. Then $f^*\alpha-[\omega'']$ is pseudo-effective if $\alpha-[\omega]$ is and the rest follows as $\mu _* f^*\alpha=\nu ^*\alpha\in H^{1,1}_{\rm BC}(X')$. 
\end{remark}
\begin{lemma}\label{l-modbig} 

Let $(X,B+\bbeta)$ be a generalized pair such that $B+\bbeta_X$ is modified big, and $\nu :X'\to X$ a birational morphism where $X'$ is strongly $\Q$-factorial. If $K_{X'}+B_{X'}+\bbeta _{X'}=\nu ^*(K_X+B+\bbeta _{X})$ then $B_{X'}^{\geq 0}+\bbeta_{X'}+E$ is big for any effective divisor $E$ whose support contains all the $\nu$-exceptional divisors.
\end{lemma}
\begin{proof} 
Assume that $\nu :X'\to X$ is a log resolution of $(X,B+\bbeta)$ then $X'$ is smooth, $\bbeta _{X'}$ is nef, $\bbeta =\overline{\bbeta _{X'}}$, and $B_{X'}$ has simple normal crossings where $K_{X'}+B_{X'}+\bbeta _{X'}=\nu ^*(K_X+B+\bbeta _X)$.
By definition, there is a resolution $f:X''\to X$ and a big class $\gamma''\in H^{1,1}_{\rm BC}(X'')$ such that $f _*\gamma ''=B+\bbeta _X$. We may assume that $f$ factors through $\mu:X''\to X'$. It follows that $f_*(\gamma ''-\mu ^*(B'+\bbeta _{X'}))=0$. Thus $\gamma ''-\mu ^*(B'+\bbeta _{X'})\equiv F$ where $F$ is an $f$-exceptional $\R$-divisor. But then $B'+\bbeta _{X'}+\mu _*F=\mu _* \gamma'' $ is big and hence so is $(B')^{\geq 0}+\bbeta _{X'}+(\mu _*F)^{\geq 0}$.
Since $(B')^{\geq 0}+\bbeta _{X'}$ is pseudo-effective, then $(B')^{\geq 0}+\bbeta _{X'}+\epsilon (\mu _*F)^{\geq 0}$ is big for any $\epsilon >0$. Since $(\mu _*F)^{\geq 0}$ is $\nu$-exceptional, we may assume that $\epsilon (\mu _*F)^{\geq 0}\leq E$ and hence $(B')^{\geq 0}+\bbeta _{X'}+E$ is big.

We now address the general case. Let $\mu :X''\to X'$ be a log resolution. By what we have seen above, $B^{\geq 0}_{X''}+\bbeta _{X''}+\delta {\rm Ex}(\nu\circ \mu)$ is big for any $\delta >0$. Therefore
\[B^{\geq 0}_{X'}+\bbeta _{X'}+\delta {\rm Ex}(\nu)=\mu _*(B^{\geq 0}_{X''}+\bbeta _{X''}+\delta {\rm Ex}(\nu\circ \mu))\] is big for any $\delta >0$.
Since $B^{\geq 0}_{X'}+\bbeta _{X'}+E\geq B^{\geq 0}_{X'}+\bbeta _{X'}+\delta {\rm Ex}(\nu)$ for any $0<\delta \ll 1$, the claim follows.
\end{proof}

\subsection{Cone theorems}
In this paper there are two possible cone theorems in different spaces to consider, depending on whether we are running the projective or the transcendental minimal model program (see Theorems \ref{t-cone} and \ref{t-projcone}). We summarize and clarify the different definitions here for the reader's convenience.

\begin{definition}[Algebraic]
Let $X$ be a compact analytic variety, we will denote $$N^1(X)=({\rm Pic}(X)\otimes _\Z\R)/\equiv^{alg}$$ as the space of $\R$-line bundles modulo numerical equivalence by curves, i.e. $L\equiv^{alg} 0$ iff $L\cdot C=0$ for every curve $C\subset X$. The dual, denoted by $N_1(X)$, is the space of curves modulo numerical equivalence when paired with line bundles. 
We let 
\[
{\rm NE}(X)\subset\overline{\rm NE}(X)\subset N_1(X)
\]
be the cone generated by curves and its closure. When $X$ is algebraic (or even projective), then these are the usual definitions in the MMP theory. 

In general, algebraic cycles are rare and these spaces and corresponding cones might be too small to be useful. However, when there is a projective morphism $f: X\to Y$ between compact spaces, we often consider the relative cones. We let $N_1(X/Y)\subset N_1(X)$ be the subspace generated by the $f$-vertical curves, i.e. $C\subset X$ such that $f(C)={\rm pt}$, and $N^1(X/Y)$ be the dual. The relative Mori cone ${\rm NE}(X/Y)$ is the cone generated by the $f$-vertical curves and we let $\overline{\rm NE}(X/Y)$ be the closure. We define $\rho(X),\rho(X/Y)$ to be the rank of $N^1(X)$ and $N^1(X/Y)$ as $\R$-linear spaces.

If $L$ is an $\R$-line bundle on $X$, then according to the relative Kleiman projectivity criterion we have: {\it $L$ is $f$-ample iff $[L]\in N^1(X/Y)$ is positive on $\overline{\rm NE}(X/Y)\backslash\{0\}$.} We remark that $f$ is assumed to be projective in this criterion.
\end{definition}

\begin{definition}[Analytic]\label{d-an}
Let $X$ be a compact analytic variety, for simplicity, we will further assume that $X$ is normal and in Fujiki's class $\mathcal{C}$. We denote
\[
N^1_{an}(X):= H^{1,1}_{\rm BC}(X),
\]
and $N_{1,an}(X)$  the space of real closed currents of bidimension
$(1, 1)$ modulo the numerical equivalence given by pairing with $H^{1,1}_{\rm BC}(X)$, where $an$ stands for ``analytic". We denote
\[
\overline{\rm NA}(X)\subset N_{1,an}(X)
\]
the closed cone generated by the classes of positive closed currents. We let 
\[
{\rm NE}_{an}(X)\subset\overline{\rm NE}_{an}(X)\subset\overline{\rm NA}(X)
\]
be the cone (and its closure) generated by the currents $[C]$ of integration along an irreducible curve $C\subset X$. Let $V_{alg}\subset N_{1,an}(X)$ be the subspace generated by curve currents. Then the natural map $\Pic(X)\to H^{1,1}_{\rm BC}(X)$ induces natural surjective maps
\[
V_{alg}\to N_1(X),~ {\rm NE}_{an}(X)\to {\rm NE}(X),~ \overline{\rm NE}_{an}(X)\to \overline{\rm NE}(X),
\]
which are typically not injective. If $X$ has only rational singularities, then we have $N^1_{an}(X)\to H^2(X,\R)$ is injective and the natural map $N^1_{an}(X)\to N_{1,an}(X)^\vee$ is an isomorphism (see \cite[Chapter 3]{HP16} for details).  

Now assume $X$ is K\"ahler and $\alpha$ is a closed $(1,1)$-form with local potentials, then $[\alpha]$ is K\"ahler (resp. nef) if and only if $[\alpha]\in N^1_{an}(X)$ is positive (non-negative) on $\overline{\rm NA}(X)\backslash\{0\}$ (resp. $\overline{\rm NA}(X)$).
Let $f: X\to Y$ be a surjective proper morphism between compact normal K\"ahler varieties, then there are natural maps
$$f^*: N^1_{an}(Y)\to N^1_{an}(X),~f_*:N_{1,an}(X)\to N_{1,an}(Y),$$
such that $f_*(\overline{\rm NA}(X))\subseteq \overline{\rm NA}(Y)$.
By \cite[Lemma 2.38]{DHP24} we have 
\[
f^*[\beta]\text{ is nef if and only if } [\beta] \text{ is nef, } [\beta]\in N^1_{an}(Y). 
\]
Then it follows easily that $f_*(\overline{\rm NA}(X))=\overline{\rm NA}(Y)$ and $f_*$ is surjective. By duality,  $f^*$ is injective, assuming that both $X,Y$ have rational singularities. 

If $f$ is further assumed to be projective, then we define ${\rm NE}_{an}(X/Y)\subset V_{alg}(X/Y)\subset N_{1,an}(X)$ to be the cone and subspace generated by $f$-vertical curves and $\overline{\rm NE}_{an}(X/Y)$ the closure of the cone. As above, there is a natural surjective map $\overline{\rm NE}_{an}(X/Y)\to\overline{\rm NE}(X/Y)$. We define $\rho_{\rm BC}(X/Y)$ to be $\dim_\R (H^{1,1}_{\rm BC}(X)/f^*H^{1,1}_{\rm BC}(Y))$.
\end{definition}

\begin{remark}
We can also introduce $N_{1,{\rm an}}(X/Y)\subset N_{1,{\rm an}}(X)$ as the subspace generated by closed currents supported on fibers of $f$  and $N^1_{\rm an}(X/Y)=N_{1,{\rm an}}(X/Y)^{\vee}$. Similarly, the cone $\overline{\rm NA}(X/Y)$ is generated by the positive classes supported on fibers. However, we will not  use these notations in this paper in an essential way. The following lemma will help to understand our definition of $\rho_{\rm BC}(X/Y)$.
\end{remark}

\begin{lemma} Let $f: X\to Y$ be a proper contraction between compact complex spaces with rational singularities (in particular $f_*\OO _X=\OO_Y$), $X$ is in Fujiki's class $\mathcal{C}$, and $R^if_*\OO _X=0$ for $i>0$, then $H^{1,1}_{\rm BC}(X)/H^{1,1}_{\rm BC}(Y)\cong N^1_{\rm an}(X/Y)$ and $N_{1,\rm an}(X/Y)$ is spanned by $f$-vertical curves. 
\end{lemma}
\begin{proof} 
By Lemma \ref{l-rccfibres} below, the dual of the short exact sequence
\[
0\to H^{1,1}_{\rm BC}(Y)\stackrel{f^*}{\longrightarrow} H^{1,1}_{\rm BC}(X)\to H^{1,1}_{\rm BC}(X)/H^{1,1}_{\rm BC}(Y)\to 0
\]
is exactly
\[
0\to V_{alg}(X/Y)\to N_{1,an}(X)\stackrel{f_*}{\longrightarrow} N_{1,an}(Y)\to 0
\]
Since $V_{alg}(X/Y)\subseteq N_{1,an}(X/Y)$ and $N_{1,an}(X/Y)\subset\Ker f_*$ from the definitions, we must have $V_{alg}(X/Y)=N_{1,an}(X/Y)$ and the isomorphism
\[
H^{1,1}_{\rm BC}(X)/H^{1,1}_{\rm BC}(Y)\cong N^1_{\rm an}(X/Y)=(N_{1,\rm an}(X/Y))^{\vee}.
\]
\end{proof}

\subsection{Generalities of the transcendental MMP}
We recall the following definition from \cite[Definition 2.11]{DHY23}.
\begin{definition}\label{d-models}
If $(X/S,B+\boldsymbol{\beta})$ is a generalized dlt pair over $S$, then we say that a bimeromorphic map $\phi:X\dasharrow X^{\rm m}$ (proper over $S$) is a {\it log minimal model over $S$} (resp. a {\it log terminal model} over $S$) if (1-3) below hold (resp. (1-4) below hold).
\begin{enumerate}
    \item $(X^{\rm m},B^{\rm m}+\boldsymbol{\beta} )$ is a strongly $\Q$-factorial generalized dlt pair, where $B^{\rm m}=\phi _*B+E$, and $E$ is the reduced sum of all $\phi ^{-1}$-exceptional divisors,
    \item $K_{X^{\rm m}}+B^{\rm m}+\boldsymbol{\beta}_{X^{\rm m}}$ is nef over $S$,
    \item $a(P,X,B+\boldsymbol{\beta} )<a(P,X^{\rm m},B^{\rm m}+ \boldsymbol{\beta})$ for every $\phi$-exceptional divisor $P$, and
    \item $\phi$ is a birational contraction (i.e. there are no $\phi^{-1}$-exceptional divisors, so that $E=0$).
\end{enumerate}
We say that $\phi:X\dasharrow X^{\rm m}$ (proper over $S$) is a {\it good log minimal model over $S$} (resp. a {\it good log terminal model} over $S$) if (1-3) above hold (resp. (1-4) above hold) and there exists a morphism $g:X^{\rm m}\to Z$ over $S$, and a K\"ahler form $\omega _Z$ on $Z$ such that $K_{X^{\rm m}}+B^{\rm m}+\boldsymbol{\beta}_{X^{\rm m}}\equiv g^*\omega _Z$.

If $(X/S,B+\boldsymbol{\beta})$ is a generalized dlt pair over $S$, then we say that a meromorphic map $\phi:X\dasharrow Z$  (proper over $S$) is a {\it K\"ahler model over $S$} if there is a K\"ahler class  $\omega _Z$ on $Z$ and a resolution $\nu :X'\to X$ such that the induced meromorphic map is a morphism $f:X'\to Z$, $\nu ^*(K_X+B+\bbeta_X)\equiv f^*\omega _Z+E$ where $0\leq E\leq N(\nu^*(K_X+B+\bbeta_X))$ (here $N(...)$ denotes the negative part of the Boucksom-Zariski decomposition see \cite[Definition 3.7]{Bou04}). We will often refer to the K\"ahler model as the log canonical model.

If $(X/S,B+\boldsymbol{\beta})$ is a generalized dlt pair over $S$, then we say that a bimeromorphic map $\phi:X\dasharrow X^{\rm m}$  (proper over $S$) is a {\it weak log canonical model over $S$} 
if (1-3) below hold. (Note that if (1-4) below hold, then $\phi$ is a birational log canonical  model over $S$).
\begin{enumerate}
    \item $(X^{\rm m},B^{\rm m}+\boldsymbol{\beta})$ is generalized lc pair, where $B^{\rm m}:=\phi _*B+E$, and $E$ is the reduced sum of all $\phi ^{-1}$-exceptional divisors,
   
   \item $K_{X^{\rm m}}+B^{\rm m}+\boldsymbol{\beta}_{X^{\rm m}}$ is nef over $S$, 
    \item $a(P,X,B+\boldsymbol{\beta} )\leq a(P,X^{\rm m},B^{\rm m}+\boldsymbol{\beta})$ for every $\phi$-exceptional divisor $P$, and
    \item $[K_{X^{\rm m}}+B^{\rm m}+\boldsymbol{\beta}_{X^{\rm m}}]\in H^{1,1}_{\rm BC}(X^{\rm m})$ is K\"ahler over $S$.
\end{enumerate}
If $X$ is proper and $S$ is a point, then we drop ``over $S$'' and simply say that we have a log minimal model, log terminal model etc.
\end{definition}

\begin{definition}Let $\phi:X\dasharrow X'$ be a proper bimeromorphic contraction of normal analytic varieties, $\alpha \in H^{1,1}_{\rm BC}(X)$ and $\alpha ' \in H^{1,1}_{\rm BC}(X')$ such that $\alpha '=\phi _*\alpha$ (if $p:W\to X$ and $q:W\to X'$ is a common resolution, then $\phi _*\alpha:=q_*(p^*\alpha)$). Then $\phi $ is $\alpha$-non-positive (resp. $\alpha$-negative) if for some common resolution $p:W\to X$, $q:W\to X'$ we have $p^*\alpha =q^*\alpha '+E$ where $E\geq 0$ is $q$-exceptional (resp. $E\geq 0$ is $q$-exceptional and its support contains all $\phi$-exceptional divisors). \end{definition}

\begin{definition}Let $R\in \overline{\rm NA}(X)$ be an extremal ray on a compact normal variety $X$ and $\alpha\in H^{1,1}_{\rm BC}(X)$ a nef supporting class of $R$. Then we say that a contraction morphism $f:X\to Z$ is the contraction associated to $R$ if there is a K\"ahler class $\omega \in H^{1,1}_{\rm BC}(Z)$ on $Z$ such that $f^*\omega =\alpha$. 

If $(X,B+\bbeta )$ is a generalized pair such that $(K_X+B+\bbeta _X)\cdot R<0$, then we say that \begin{enumerate} \item if $\dim Z<\dim X$, then $f$ is a $K_X+B+\bbeta _X$ Mori fiber space,
\item if $\dim Z=\dim X$ and $\dim({\rm Ex}(f))=\dim X-1$, then $f$ is a $K_X+B+\bbeta _X$ divisorial contraction, and
\item if $\dim Z=\dim X$ and $\dim({\rm Ex}(f))\leq \dim X-2$, then $f$ is a $K_X+B+\bbeta _X$ flipping contraction.
\end{enumerate}

If $f:X\to Z$ is a flipping contraction and $f^+:X^+\to Z$ is a small proper birational contraction such that $(X^+,B^++\bbeta)$ is a generalized pair and $K_{X^+}+B^++\bbeta _{X^+}$ is $f^+$-K\"ahler, then we say that $f^+$ is the flip of $f$. 
\end{definition}
\begin{remark}\label{r-div} If $f$ is a divisorial contraction and $X$ is $\Q$-factorial, then there is a unique $f$-exceptional divisor $E$. 
\end{remark}

\begin{proof} Suppose $E_1,E_2$ are two different $f$-exceptional irreducible divisors, then by \cite[Lemma A.3]{DHY23} we know $E_i$ is covered by an analytic family of $f$-vertical curves $\{C_{t_i}\}_{t_i\in T_i}$ such that $E_i\cdot C_{t_i}<0$. Then we see $[C_{t_i}]$ is a non-zero element contained in the extremal ray $R$. However, $E_1\cdot C_{t_2}\ge 0$ at least for some $t_2\in T_2$, we get a contradiction.

\end{proof}

\begin{lemma}\label{l-np} Let $\phi:X\dasharrow X'$ be a birational contraction of compact K\"ahler varieties. If $(X,B+\bbeta )$ and $(X',B'+\bbeta )$ are generalized dlt pairs where $B'=\phi _* B$, $K_{X'}+B'+\bbeta _{X'}$ is nef, and $a(P,X,B+\boldsymbol{\beta} )\leq a(P,X',B'+\boldsymbol{\beta})$ for every $\phi$-exceptional divisor $P$, then $\phi$ is $K_X+B+\bbeta_X$ non-positive.
\end{lemma}
\begin{proof} Let $p:W\to X$ and $q:W\to X'$ be a common resolution and write
\[p^*(K_X+B+\bbeta_X)+F=q^*(K_{X'}+B'+\bbeta_{X'})+E\] where $E,F\geq 0$ and $E\wedge F=0$.
Note that $F$ is $p$-exceptional and $E$ is $q$-exceptional.
Since $K_{X'}+B'+\bbeta _{X'}$ is nef, $ N(q^*(K_{X'}+B'+\bbeta_{X'})+E)=E$. 
Since $F$ is $p-$exceptional, $F\leq N(p^*(K_X+B+\bbeta_X)+F)$. Since $F\wedge E=0$, $F=0$ and so $\phi$ is $K_X+B+\bbeta_X$ non-positive.
\end{proof}
\begin{lemma}\label{l-wmm} Let $\phi:X\dasharrow X'$ be a birational contraction, of $d$-dimensional compact K\"ahler varieties, $(X,B+\bbeta )$ and $(X',B'+\bbeta )$  gklt pairs where $B'=\phi _* B$, $\phi$ is $K_X+B+\bbeta_X$ non-positive, and $K_{X'}+B'+\bbeta _{X'}$ is nef.
Then $(X,B+\bbeta )$ has a log terminal model.
\end{lemma}
\begin{proof} With the above notation, write
\[p^*(K_X+B+\bbeta_X)=q^*(K_{X'}+B'+\bbeta_{X'})+E\] where $E\geq 0$. 
Let $\mathcal F$ be the set of $\phi$-exceptional divisors not contained in the support of $E$. 
For any $F\in \mathcal F$ we have $0\geq a_F(X,B+\bbeta )=a_F(X',B'+\bbeta )$ where $a_F(\ldots)$ denotes the discrepancy along $F$.
It follows that there is a birational morphism $\mu:X''\to X'$, where $X'' $ is strongly $\Q$-factorial, whose exceptional divisors coincide with the set $\mathcal F$ (see Corollary \ref{c-ext}). We write $K_{X''}+B''+\bbeta_{X''}=\mu ^*(K_{X'}+B'+\bbeta _{X'})$ and $\psi:X\dasharrow X''$ be the induced birational map. Then $B''\geq 0$, $\psi $ is $K_X+B+\bbeta_X$ negative, and $K_{X''}+B''+\bbeta _{X''}$ is nef. Thus $\psi$ is a log terminimal model for $(X,B+\bbeta )$.
\end{proof}
\begin{lemma}\label{l-mms} Let $(X,B+\bbeta )$ be a compact gdlt pair and $\nu:X'\to X$ a proper birational morphism. Suppose that $(X',B'+\bbeta)$ is a gdlt pair such that \[K_{X'}+B'+\bbeta _{X'}=\nu ^*(K_X+B+\bbeta_X)+F\] where $F\geq 0$ and ${\rm Supp}(F)={\rm Ex}(\nu)$, then $(X,B+\bbeta)$ has a log terminal model (resp. a Mori-fiber space) iff so does $(X',B'+\bbeta )$. If $(X,B+\bbeta )$ is gklt, then we may assume that $B+\bbeta _X$ is modified big, then we may assume that $B'+\bbeta _{X'}$ is big.
\end{lemma}
\begin{proof} The first part is standard (see e.g. \cite[Lemma 3.6.10]{BCHM10}). For the second part see Lemma \ref{l-modbig}.
\end{proof}
\begin{lemma}\label{l-dltmultiples} Let $(X,B+\bbeta )$ be a compact gdlt pair. Suppose that $(X,C+\ggamma)$ is a gdlt pair such that \[K_{X}+C+\ggamma _{X}\equiv \lambda (K_X+B+\bbeta_X)\] where $\lambda > 0$. Then $(X,B+\bbeta)$ has a log terminal model iff so does $(X,C+\ggamma )$.
\end{lemma}
\begin{proof} This is a standard result that follows easily from the definitions.
\end{proof}

\begin{lemma}\label{lem: g-kltperturb} 

Let $X$ be a compact Fujiki class $\mathcal C$ variety, and $(X,B+\bbeta)$ a generalized pair such that $B+\bbeta _X$ is modified big. Then there is a generalized pair $(X,G+\ggamma)$, a resolution $X'\to X$, and a K\"ahler form $\omega'$ on $X'$ such that $K_{X}+B+\bbeta _{X}\equiv K_{X}+G+\ggamma _{X}$, $\ggamma$ descends to $X'$, $\ggamma _{X'}-\omega'$ is nef, and
\[
{\rm NKLT}(X,B+\bbeta)\supset {\rm NKLT}(X, G+\ggamma ) \text{ i.e. }\mathcal J(X,B+\bbeta)\subset \mathcal J(X, G+\ggamma ).
\]
In particular, if we set $\ddelta=\ggamma -\bar \omega'$, then $(X,G+\ddelta )$ is a generalized pair. Note that if $(X,B+\bbeta)$ is gklt, then so is $(X, G+\ggamma )$.
\end{lemma}

\begin{proof}
This follows easily from the proofs of \cite[Lemma 2.14]{HLX26} and  \cite[Lemma 3.1]{DH24}. 
\end{proof}

\subsection{The projective MMP}
Let $f:X\to Y$ be a projective morphism of analytic varieties such that either $Y$ is compact or $f$ satisfies property P (see \cite[Section 1]{Fuj22}). If $(X,B)$ is a klt pair, then the projective MMP over $Y$ is established in \cite{Nak87}, \cite{Fuj22}, and \cite{DHP24}. With some extra effort the theory can actually be generalized to klt g-pairs (cf. \cite[Theorem 2.10]{CH24} for the pair version):

\begin{theorem}\label{t-projcone} Let $f:X\to Y$ be a projective morphism of compact analytic varieties, $(X,B+\bbeta)$ a strongly $\Q$-factorial gklt pair such that $Y$ a K\"ahler space and $\bbeta \equiv _Y \mathbf N$, where $\mathbf N$ is an $\R$-Cartier b-divisor.
\begin{enumerate}\item If $K_X+B+\bbeta_X$ is not $f$-nef, then there exists a countable collection of $f$-vertical rational curves $C_i$ such that 
\[\overline{\rm NE}(X/Y)=\overline{\rm NE}(X/Y)_{(K_X+B+\mathbf N_X)\geq 0}+\sum _{i\in I}\R ^+[C_i]\] where $0<-(K_X+B+\mathbf N_X)\cdot C_i\leq 2\dim X$, for every extremal ray $\R ^+[C_i]$ there exists a line bundle $L_i$ such that $\R ^+[C_i]=\overline{\rm NE}(X/Y)\cap L_i^\perp$, and if $B+\bbeta$ is big then $I$ is finite.
\item For every extremal ray $\R ^+[C_i]$ there is a contraction morphism $\varphi :X\to Z$ onto a normal compact K\"ahler space $Z$ and a projective morphism $g:Z\to Y$ such that $-(K_X + B+\mathbf N_X)$ is $\varphi$-ample
and for any curve $C \subset  X$ such that $f (C)$ is a point we have that
\[\varphi (C)={\rm pt.}\qquad {\rm iff}\qquad [C]\in \R^+[C_i].\]  
\item We can run the projective $(K_X+B+\mathbf N_X)$-MMP over $Y$ with scaling of an $f$-ample divisor.
\end{enumerate}
\end{theorem}
\begin{proof}
Let $\omega_Y$ be a K\"ahler form on $Y$ such that $\omega_Y\cdot C>2\dim X$ for any curve $C\subset Y$, then (1) follows from applying  the global cone theorem (Theorem \ref{t-cone}) to $(X,B+\bbeta+\overline{f^*\omega_Y})$ and the fact that
\begin{enumerate}
\item[(i)] $V_{alg}(X/Y)\cap \overline{\rm NA}(X)=\overline{\rm NE}(X/Y)$,
\item[(ii)] the natural map $\overline{\rm NE}_{an}(X/Y)\to\overline{\rm NE}(X/Y)$ is surjective.
\end{enumerate}

Now for every extremal ray $\R ^+[C_i]$, there exists an $f$-ample divisor $A$ on $X$ such that $K_X+B+\mathbf N_X+A$ is $f$-nef and supports the extremal ray $\R ^+[C_i]\subset\overline{\rm NE}(X/Y)$, then (2) follows from the usual (relative) base-point-free theorem for generalized pairs since the existence of $\varphi$ can be checked locally over the base using the results of \cite{Fuj22}. We refer the reader to \cite{HLX26} for the details.

For (3), Let $H$ be an $f$-ample divisor such that $K_X+B+\mathbf N_X+H$ is $f$-nef, and $[H]+af^*\omega_Y+(\mathbf N_X-\bbeta_X)$ is a K\"ahler class for some $a\gg 0$. Then $(X,B+\lambda (\overline{H+af^*\omega_Y})+\bbeta)$ is a gklt pair for any $\lambda\ge 0$ and we can consider the relative cone theorem for $(X,B+\mathbf N+\lambda\overline{H})$ as in (1). Since we have the contraction theorem (2) and  flips are given by relative log canonical models whose  existence can be checked locally over $Y$ using the results of \cite{Fuj22}, by standard arguments we can run the MMP with scaling of an $f$-ample divisor.
\end{proof}

With some bigness assumption we can further obtain:
\begin{proposition}\label{p-relMMP} Let $f:X\to Y$ be a projective morphism of normal compact K\"ahler varieties. Assume that $(X,B+\bbeta )$ is a gklt pair such that $B+\bbeta_X$ is modified big over $Y$ and $\bbeta \equiv _Y \mathbf N$ where $\mathbf N$ is an $\R$-Cartier b-divisor.
\begin{enumerate}\item If $K_X+B+\bbeta_X$ is pseudo-effective over $Y$, then $(X,B+\bbeta)$ has a good minimal model over $Y$.
\item  If $K_X+B+\bbeta_X$ is not pseudo-effective over $Y$, then $(X,B+\bbeta)$ is birational to a Fano fibration  over $Y$.
\item If $X$ is strongly $\Q$-factorial, then the good minimal model or Fano fibration can be obtained as the output of an MMP with scaling.
\end{enumerate}
\end{proposition}
\begin{proof} (1) and (2) would follow from Lemma \ref{l-mms} and (3) by taking a resolution. By Theorem \ref{t-projcone} we may run a (projective) minimal model program over $Y$ with scaling of an $f$-ample divisor on $X$. Since termination is local over $Y$ this process terminates as a consequence of \cite[Theorem E]{Fuj22}. The goodness can also be checked locally, which is guaranteed by the usual base-point-free theorem.
\end{proof}

\subsection{Generalized dlt and $\Q$-factorial models}
\begin{lemma}\label{L-gklt-rational} Let $(X,B+\bbeta )$ be a gklt pair, then $X$ has rational singularities.
\end{lemma}
\begin{proof} See e.g. \cite[Theorem 1.5]{DH24}.
\end{proof}
\begin{theorem}[DLT models {\cite[Theorem 3.1]{HLX26}}]
\label{t-dltmodel} Let $(X,B+\bbeta )$ be a generalized pair, where $X$ is compact K\"ahler. Then there exists a projective birational morphism $f^{\rm m}:X^{\rm m}\to X$, such that all $f^{\rm m}$-exceptional divisors $P$ have discrepancy $a(X,B+\bbeta ,P)\leq -1$, and
 $(X^{\rm m}, B^{\rm m})$ is dlt where $B^{\rm m}={(f^{\rm m})}^{-1}_*(B\wedge {\rm Supp}(B))+{\rm Ex}(f^{\rm m})$. 
\end{theorem}

\begin{theorem}\label{t-small Q factorial model}  Let $(X,B+\bbeta)$ be a compact gklt pair (or a gdlt pair where $X$ is K\"ahler). Then there exists a small birational map $\nu :X'\to X$ such that $X'$ is strongly $\Q$-factorial.
\end{theorem}
\begin{proof} The gklt case follows from {\cite[Theorem 3.3]{HLX26}}. If $X$ is K\"ahler and $(X,B+\bbeta)$ is gdlt, then let $\nu :X'\to X$ be a log resolution which is an isomorphism over the simple normal crossings locus. 
We may assume that there is an effective and exceptional $\R$-divisor $F\geq 0$  such that $-F$ is $\nu$-ample.
We write $\nu ^*(K_X+B+\bbeta _X)=K_{X'}+B_{X'}+\bbeta _{X'}$. Let $\omega$ be a K\"ahler form on $X$. 
For $0<\delta \ll \epsilon \ll 1$, 
$\gamma:=\nu ^*\omega -\epsilon F+\delta B_{X'}^{=1}$ is K\"ahler and 
$(X',B_{X'}-\delta B_{X'}^{=1}+\epsilon F)$ is sub-klt. Let $\Delta '=B_{X'}-\delta B_{X'}^{=1}+\epsilon F$, $\Delta=\nu _*\Delta'$, and $\ddelta=\bbeta +\bar \gamma$, then $(X,\Delta +\ddelta )$ is gklt and the claim follows.
\end{proof}
\begin{corollary}\label{c-ext} Let $(X,B+\bbeta)$ be a compact K\"ahler gklt pair and $\mathcal E$ a finite set of divisors over $X$ such that $a_P(X,B+\bbeta )\leq 0$ for every $P\in \mathcal E$, then there exists a projective birational map $\mu :\hat X\to X$ such that $\hat X$ is strongly $\Q$-factorial and the set of $\nu$-exceptional divisors coincides with $\mathcal E$.
\end{corollary}
\begin{proof} Let $\nu:X'\to X$ be a projective log resolution such that $\bbeta$ descends and each divisor in $\mathcal E$ is a divisor on $X'$. We write $K_{X'}+B_{X'}+\bbeta _{X'}=\nu ^*(K_X+B+\bbeta _X)$. Let $B'$ be an effective $\R$-divisor such that \[K_{X'}+B'+\bbeta _{X'}=\nu ^*(K_X+B+\bbeta _X)+F\]
where $(X',B')$ is klt and the support of $F$ equals the set of $\nu$-exceptional divisors not in $\mathcal E$. Since $\bbeta_{X'}\equiv_X F-K_{X'}-B'$ is nef over $X$, By Theorem \ref{t-projcone}, we may run a $(K_{X'}+B'+\bbeta _{X'})$-MMP over $X$ and it terminates with a good minimal model by Proposition \ref{p-relMMP}. Let $\phi :X'\dasharrow \hat X$ be the output of this MMP, and $\mu :\hat X\to X$ the induced morphism. Then $\hat X$ is strongly $\Q$-factorial and $K_{\hat X}+\hat B+\bbeta _{\hat X}=\phi _*(K_{X'}+B'+\bbeta _{X'})$ is nef over $X$. By the negativity lemma (see Lemma \ref{l-negbir}), $\phi _*F=0$. It is easy to see that $\phi$ does not contract divisors in $\mathcal E$.
\end{proof}

\subsection{Spaces endowed with a Moishezon contraction}

Unless otherwise stated, we will assume that all the complex spaces are separated.

\begin{definition} We first give the definition of Moishezon morphisms between (not necessarily reduced or irreducible) compact complex spaces. We say a proper morphism $f: X\to Z$ is {\it Moishezon} if there exists a surjective projective morphism $g: W\to X$ such that 
$$f\circ g: W\to X\to Z$$ 
is also projective (we note that this is a non-standard definition).

When $X$ is integral (reduced and irreducible), say that $f$ is {\it strongly Moishezon} if $g$ is required to be birational/bimeromorphic. For example, any proper birational/bimeromorphic morphism is strongly Moishezon. 

We say a proper morphism $f: X\to Z$ is a {\it (resp. strongly) Moishezon contraction} if
\begin{enumerate}
\item $f$ is (resp. strongly) Moishezon,
\item $f$ is a contraction, i.e. $f_*\OO_X=\OO_Z$.
\end{enumerate}

We say a proper morphism $f: X\to Z$ between complex spaces is \textit{curve-connected} if:
\[
\text{Any two points in a fiber $f^{-1}(z)$ are connected by curves in $f^{-1}(z), z\in Z$.}
\]
\end{definition}
\begin{lemma}\label{lem: Moi contr is curve-connected}
A Moishezon contraction $f: X\to Z$ is curve-connected.
\end{lemma}

\begin{proof}
Since $f^{-1}(z)$ is connected for any $z\in Z$, it suffices to show that for any two points in the same irreducible component $X_{z,i}$ of $f^{-1}(z)$ are connected by curves.   Let $W_i$ be an irreducible component of $g^{-1}(X_{z,i})$ which dominates $X_{z,i}$. By definition $W_i$ is projective, and is connected as it is irreducible, any two points in $W_i$ are connected by curves. Therefore, any two points on $X_{z,i}$ are also connected by curves.
\end{proof}

\begin{definition}
Let $X$ be a (compact) complex space, and $[\alpha]\in H^2(X,\R)$ a geometrically nef class so that $C\cdot [\alpha]\geq 0 $ for every curve $C\subset X$ . We say that $f$ is the {\it (Moishezon) contraction defined by} $[\alpha]$ if
\begin{itemize}
\item[(i)] $f$ is a contraction, i.e. $f_*\OO_X=\OO_Z$, 
\item[(ii)] $f$ is Moishezon, and
\item[(iii)] if $C\subset X$ is a curve, then $f(C)=\mathrm{pt}$ if and only if $C\cdot[\alpha]=0$.
\end{itemize}
We will also say $[\alpha]$ is {\it endowed with a Moishezon contraction} (EMC for short).
\end{definition}

We will frequently use the following properties of Moishezon contractions in this paper. Notice that these properties arise naturally for a space endowed with a semiample line bundle.

\begin{lemma}\label{l-contraction} If $f:X\to Z$ is the Moishezon contraction defined by a nef class $[\alpha]\in H^2(X,\R )$, then the following properties hold.
\begin{enumerate}
\item $f$ is unique (up to unique isomorphism).
\item $f$ is compatible with pullback, i.e. if $\pi:W\to X$ is a Moishezon contraction, then $f\circ\pi$ is the Moishezon contraction defined by $\pi^*[\alpha]$.
\item $f$ is universal, i.e. if $\hat{f}: X\to \hat Z$ is a Moishezon contraction such that any curve contracted is $[\alpha]$-trivial, then there exists a unique Moishezon contraction $\phi: \hat Z\to Z$ such that $f=\phi\circ \hat f$. 
\item $f$ descends if $[\alpha]$ descends, i.e. if $g: X\to Y$ is a Moishezon contraction and $g^*[\beta]=[\alpha]$ for some nef class $[\beta]$ in $H^2( Y, \R)$, then the contraction defined by $[\beta]$ exists, and its composition with $g$ equals to $f$.
\end{enumerate}
\end{lemma}
\begin{proof}
(1) Let $f':X\to Z'$ be another contraction satisfying the above conditions. Since $\OO_Z\to f_*\OO_X$ and $\OO_{Z'}\to f'_*\OO_X$ are both isomorphisms, it suffices to show $f$ and $f'$ are the same as maps of sets. Since $f$ is curve-connected by Lemma \ref{lem: Moi contr is curve-connected}, we know that for any $z\in Z$, any two points in $f^{-1}(z)$ can be connected by a chain of curves $\cup_i C_i\subset f^{-1}(z)$. From the definition we have $C_i\cdot[\alpha]=0,~\forall i$, so these curves must be also contracted by $f'$. By symmetry we deduce that the fibers of $f'$ and $f$ coincide, and there is a unique isomorphism $Z'\to Z$ that identifies $f$ and $f'$.

(2) This is a direct consequence of the projection formula 
$$C\cdot\pi^*[\alpha]=\pi_*(C)\cdot [\alpha],$$
where $C\subset W$ and $\pi_*: H_2(W,\R)\to H_2(X,\R)$ is the natural map.

(3) Since $f_*\OO_X=\OO_Z$ and $\hat f_*\OO_X=\OO_{\hat Z}$, we only need to define $\phi$ set/topology theoretically. This is clear from the construction of the quotient spaces: any two points $x_1,x_2\in X$ such that $\hat f(x_1)=\hat f(x_2)$ map to the same point on $Z$ through $f$.

(4) By (3) there exists a proper contraction $\phi: Y\to Z$ such that $\phi\circ g=f$ and $\phi$ is curve-connected since $f$ is. The rest follows by the projection formula and the fact that $C\subset Y$ is always dominated by some curve $C'\subset X$.
\end{proof}

\subsection{Relatively projective morphisms}
We first specify the definition for projective morphisms of complex spaces in our paper:

\begin{definition}
Let $f: X\to Y$ be a proper holomorphic morphism between complex analytic spaces. We say the morphism $f$ is {\it projective} if there exists a line bundle $L$ on $X$ that
is relatively ample over $Y$. 

\end{definition}
\begin{remark}
We note that this is weaker than requiring that $f$ is {\it H-projective} i.e. $f$ factors through a closed embedding $X\hookrightarrow\mathbb{P}^N_Y$ (if $Y$ is quasi-projective then the two notions are equivalent). In general the composition of two projective morphisms is not necessarily projective, however if $f:X\to Y$ and $g:Y\to Z$ are projective morphisms of compact analytic spaces, then so is $g\circ f$.
To see this, let $L$ be an $f$-ample line bundle and $M$ a $g$-ample line bundle, then it suffices to show that $L\otimes f^*M^{\otimes k}$ is $g\circ f$ ample for $k\gg 0$ which follows easily by the Serre criterion for ampleness and an easy spectral sequence argument. 
\end{remark}

\begin{lemma}\label{l-rel} Let $f:X\to Y$ be a surjective morphism with connected fibers between compact K\"ahler manifolds such that $f^*: H^0(Y, \Omega^2_Y)\to H^0(X, \Omega^2_X)$ is an isomorphism. Then for any  $\alpha\in H^{1,1}(X)$ there is an $\R$-line bundle $L$ on $X$ and $\gamma \in H^{1,1}(Y)$ such that $\alpha \equiv c_1(L) +f^*\gamma $ in $H^{1,1}(X)$.
If $\alpha$ is $f$-K\"ahler  then $L$ is $f$-ample. In particular $f$ is projective.
\end{lemma}
\begin{proof} This follows easily along the lines of the proof of Step 1 of \cite[Theorem 3.1]{CH24}. For the convenience of the reader, we outline an argument here.

Pick $\eta_i\in H^2(X, \Q)$ sufficiently close to $\alpha$ and write $\alpha\equiv \sum_{i=1}^kr_i\eta_i$, where $k:=b_{2n-2}(X)+1$, $n$ is the complex dimension of $X$, and $r_i\in\R^+$ for all $i$. By the arguments in the proof of Step 1 in \cite[Theorem 3.1]{CH24}, there are $\widetilde\eta_i\in H^2(X, \Q)$ of type $(1,1)$ such that $\widetilde\eta_i-\eta_i=f^*\gamma_i$ for some $\gamma_i\in H^2(Y, \Q)$ for all $i=1,\ldots, k$. Then up to rescaling $\widetilde \eta_i$, by the Lefschetz $(1,1)$ theorem there is a line bundle $L_i\in\Pic(X)$ such that $\widetilde\eta_i=c_1(L_i)$ for all $i=1,\ldots, k$. 
Since $\eta_i$ is sufficiently close to $\alpha$, it follows from the arguments of the proof of \cite[Theorem 3.1]{CH24} that if $\omega$ is $f$-K\"ahler, then $L_i$  is an $f$-ample line bundle for all $i=1,\ldots, k$. Let $L:=\sum_{i=1}^k r_iL_i$, then we have $\omega\equiv_f \sum_{i=1}^kr_iL_i=c_1(L)$.  Note that $\alpha-c_1(L)\equiv f^*\gamma $ where $\gamma =-\sum r_i\gamma _i$ and so $\gamma \in H^{1,1}(Y)$. 
\end{proof}

The following lemma is a generalization of \cite[Lemma 3.3]{HP16}, where they assume $f$ is bimeromorphic.

\begin{lemma}\label{l-rccfibres} Let $f : X \to  Y$ be a proper morphism with connected fibers
between normal compact analytic varieties with rational singularities. Assume $X$ is in Fujiki's class $\mathcal C$, and 
that the general fiber is rationally connected or $R^if_*\OO _X=0$ for $i>0$. Then $f$ is strongly Moishezon, $H^0(Y,\Omega ^{[2]}_Y)\to H^0(X,\Omega ^{[2]}_X)$ is an isomorphism and both
\[f^*:H^{1,1}_{\rm BC}(Y)\to H^{1,1}_{\rm BC}(X)\qquad {\rm and}\qquad f^*:H^2(Y,\R)\to H^2(X,\R)\] are injective and the image of $f^*$ is given by those classes such that $\alpha \cdot C=0$ for every vertical curve $C$. 
\end{lemma}
\begin{proof}
Consider projective resolutions $\nu :X'\to X$ and $\mu :Y'\to Y$ such that $X',Y'$ are K\"ahler and $f':X'\to Y'$ is a morphism. From our assumption, either $R^if_*\OO_{X}=0,~i\ge 1$ or the general fibers $F'$ of $f'$ are rational connected, hence $h^i(\OO_{F'})=0,~i>0$. We have $R^if'_*\OO_{X'}=0,~i\ge 1$ by Leray spectral sequence (in the former case) or \cite[Theorem 7.1]{Kol86} and \cite[Thm. II and IV]{Tak95} (in the latter case). Therefore $H^2(Y',\OO_{Y'})\to H^2(X',\OO_{X'})$ is an isomorphism, hence
\[H^0(X', \Omega ^{2}_{X'})\cong\overline {H^2(X',\OO _{X'})}\cong\overline {H^2(Y',\OO _{Y'})}\cong H^0(Y', \Omega ^{2}_{Y'}).\]
By the step 1 of \cite[Theorem 3.1]{CH24} we know $f'$ is projective, so $f$ is strongly Moishezon. Notice that $H^0(X', \Omega ^{2}_{X'})\cong H^0(X, \Omega ^{[2]}_X)$ and $H^0(Y', \Omega ^{2}_{Y'})\cong H^0(Y, \Omega ^{[2]}_Y)$ (see \cite[Corollary 1.8]{KS21}),  therefore $H^0(Y, \Omega ^{[2]}_Y)\to H^0(X, \Omega ^{[2]}_X)$ is an isomorphism.

The rest of the claim about $f^*$ now follows as in \cite[Lemma 2.6]{DH24} by looking at $\nu^*,\mu^*$ and $f'^*$, which essentially arises to \cite[Theorem 12.1.3]{KM92}.
\end{proof}

\begin{lemma}\label{lem: L rel num eq to kahler => f-ample}
Let $f: X\to Y$ be a projective morphism between normal compact complex spaces. If there is an $\mathbb{R}$-Cartier divisor $L$ and a K\"ahler form $\omega$ on $X$ such that $L\equiv^{alg}_f \omega$, i.e. $L\cdot C=\omega\cdot C$ for any curve $C\subset X$ with $f(C)=pt$, then $L$ is $f$-ample.
\end{lemma}

\begin{proof}
$L$ is positive on $\overline {\rm NE}_{an}(X/Y)\setminus\{0\}$ and hence also on $\overline {\rm NE}(X/Y)\setminus\{0\}$. By Kleiman's criterion we see that $L$ is $f$-ample.
\end{proof}

The following can be viewed as a relative version of \cite[Corollary 5]{Nam02}.

\begin{lemma}\label{lem: Moishezon to strong Moishezon}
Let $f: X\to Z$ be a Moishezon morphism where $X$ is a K\"ahler variety. If there exists an $\R$-line bundle $L$ such that $L\equiv^{alg}_Z \omega$ for a K\"ahler form $\omega$ on $X$, that is
\[
L\cdot C=\omega \cdot C,\text{ for all } C\subset X \text{ such that } f(C)=\text{pt}
\]
Then $f$ is projective and $L$ is $f$-ample.
\end{lemma}
\begin{proof}
By Lemma \ref{lem: L rel num eq to kahler => f-ample} it suffices to prove $f$ is projective. Perturbing $L$ and $\omega$ a little bit we may now assume $L$ is a $\Q$-line bundle. To show $L$ is $f$-ample is equivalent to showing that $L|_{X_z}$ is ample for any $z\in Z$. Since ampleness can be checked on every irreducible component (with the reduced structure), the statement follows immediately from \cite[Corollary 5]{Nam02}. 
\end{proof}

\begin{remark}
It is worth mentioning that Namikawa showed in \cite{Nam02} that a compact Moishezon K\"ahler variety with 1-rational singularities are actually projective. One might expect an analogous result on the projectiveness of a Moishezon morphism, but it turns out to be quite subtle. 

The singularity condition in \cite{Nam02} was used to show an algebraic 1-cycle is zero in $N_1$ iff it is zero in $H_2$ by passing to its smooth projective bimeromorphic model. Hence the linear functional $\langle\cdot,\omega\rangle$ on $H_2$ defined by a K\"ahler form would descend to a linear functional $\langle\cdot,L\rangle$ on $N_1$ for some line bundle, then $L$ is ample by \cite[Corollary 5]{Nam02} as above. However, for a Moishezon morphism $f: X\to Y$, we need a vertical algebraic 1-cycle in $N_1(X/Y)$ to be zero iff it is zero in $H_2(X/Y)$. This is much more difficult to satisfy since
\begin{enumerate}
\item Some fiber $X_y$ could be very singular even if both $X$ and $Y$ are smooth.
\item Some line bundles on $X_y$ may not extend even if $f$ is smooth.
\end{enumerate}
In fact there are concrete examples such that $f: X\to Y$ is a Moishezon morphism between compact K\"ahler manifolds but $f$ is not projective.

\end{remark}

\begin{lemma}\label{l-birgklt}
Let $f:X\to Y$ be a proper birational morphism between normal varieties in Fujiki's class $\mathcal C$, $(X,B+\bbeta)$ a gklt pair such that $K_X+B+\bbeta _X\equiv^{alg} _Y 0$, then $(Y,B_Y+\bbeta )$ is gklt where $K_Y+B_Y+\bbeta _Y=f_*(K_X+B+\bbeta _X)$ and we have $f^*(K_Y+B_Y+\bbeta_Y)=K_X+B+\bbeta_X$.
\end{lemma}
\begin{proof} We begin by showing that $Y$ has rational singularities. Since the question is local over $Y$, we may assume that $Y$ is Stein and $f$ satisfies property  P. Let $\nu :X'\to X$
 be a resolution such that $X'\to Y$ is projective and write $K_{X'}+B_{X'}+\bbeta _{X'}=\nu ^*(K_X+B+\bbeta _X)$, then $-(K_{X'}+B_{X'})\equiv^{alg}_Y \bbeta _{X'}$ is nef and big over $Y$ and so there exists $\Delta '\sim_\R -(K_{X'}+B_{X'})$ such that $K_{X'}+B_{X'}+\Delta '\sim_\R 0$ is sub-klt and hence $K_X+B+\Delta =\nu _*(K_{X'}+B_{X'}+\Delta ')$ is klt and $K_X+B+\Delta\sim _{\R}0$. Let $K_Y+B_Y+\Delta _Y=f_*(K_X+B+\Delta)$, which is $\R$-Cartier. But then $K_X+B+\Delta=f^*(K_Y+B_Y+\Delta _Y)$ and it follows easily that $(Y,B_Y+\Delta _Y)$ is klt and in particular $Y$ has rational singularities. 
 
 We now drop the assumption that $Y$ is Stein. By Lemma \ref{l-rccfibres}, $K_X+B+\bbeta _X=f^*\alpha$ for some $\alpha \in H^{1,1}_{\rm BC}(Y)$ and in particular \[K_{Y}+B_Y+\bbeta _Y:=f_*(K_X+B+\bbeta _X)\in H^{1,1}_{\rm BC}(Y),\] and so $(Y,B_Y+\bbeta)$ is a gklt pair.
 \end{proof}

\subsection{The NQC property and the $\alpha$-trivial MMP}
The concept of NQC (abbreviation for nef $\Q$-Cartier combination) is first introduced in \cite{HL22} for nef line bundles when studying projective generalized pairs with $\R$ coefficients, which is very important in showing the existence and terminations for some special MMPs. 

The philosophy turns out to be particularly useful in the K\"ahler category, where we often consider closed $(n,n)$-forms which are just $\R$-classes. However, we must enlarge the spaces where we are doing the decomposition since $H^{1,1}$ is usually not defined over $\Q$.

\begin{definition}
Let $X$ be a normal compact complex space in the Fujiki's class $\mathcal{C}$ with rational singularities. Let $\alpha$ be a nef class in $H^{1,1}_{\rm BC}(X)$, then we say $\alpha$ is NQC if there exist $a_i\in \R^{>0}$ and $\alpha_i\in H^{2}(X,\Q)$ such that 
\[
\sum_i a_i\alpha_i=\alpha \text{ under the canonical inclusion } H^{1,1}_{\rm BC}(X)\to H^2(X,\R)
\]
and 
\begin{itemize}
\item (strong sense) $\alpha_i$ is non-negative on $\overline{\rm NA}(X)$.
\item (weak sense) $\alpha_i\cdot C\ge 0$ for any curve $C$ in $X$.
\end{itemize}
We may assume that these $\alpha_i$ generate the minimal $\Q$-space $V\subset H^{2}(X,\Q)$ 
such that
\[\alpha\in V\otimes_\Q \R.\]
We note that we will use the strong version in this paper, however the weak version would, also fit our needs.
\end{definition}
\begin{lemma}\label{l-NQCpullback} Let $f:X\to Y$ be a surjective morphism of normal compact  K\"ahler varieties and $f_*\OO _X=\OO _Y$. If $\alpha \in H^{1,1}_{\rm BC}(Y)$ then $\alpha$ is a nef NQC class iff $f^*\alpha$ is a nef NQC class.
\end{lemma}
\begin{proof} Recall that $\alpha$ is nef iff $f^*\alpha$ is nef (see \cite[Theorem 6.1]{HP24}). 

Suppose that $\alpha$ is NQC class so that $\alpha =\sum a_i\alpha _i$  where $a_i>0$ and $\alpha _i\in H^2(Y,\Q)$ are non-negative on $\overline{\rm NA}(Y)=f_*\overline{\rm NA}(X)$, then $f^*\alpha =\sum a_if^*\alpha _i$ where $f^*\alpha _i\in H^2(X,\Q)$ are non-negative on $\overline{\rm NA}(X)$.

Conversely, suppose $f^*\alpha =\sum a_i\alpha' _i$ where $\alpha' _i\in H^2(X,\Q)$ are non-negative on $\overline{\rm NA}(X)$. We may assume that $\alpha '_i$ belong to the minimal $\Q$-space $V\subset H^2(X,\Q)$ such that $\alpha \in V\otimes _\Q \R$. 
Since $\alpha \in f^*H^2(Y,\R)$, then $V\subset f^*H^2(Y,\Q)$ and hence we may assume that $\alpha '_i=f^*\alpha _i$ where $\alpha _i\in H^2(Y,\Q)$ are non-negative on $\overline{\rm NA}(Y)=f_*\overline{\rm NA}(X)$.
\end{proof}
\begin{lemma}\label{lem: gklt pair+Kahler is NQC}
Let $(X,B+\bbeta)$ be a compact K\"ahler gklt pair and $\omega$ a K\"ahler form on $X$. Assume $\alpha:=[K_X+B+\bbeta_X+\omega]$ is nef, then $\alpha$ is also NQC. 
\end{lemma}

\begin{proof}
By the cone theorem we have 
\[
\overline {\rm NA}(X)=\overline {\rm NA}(X)_{K_X+B+\bbeta _X+\frac{1}{2}\omega\geq 0}+\sum _{j\in J}\R^+[C_j].
\]
for a finite set $J$. Let $V$ be the minimal $\Q$-space in $H^2(X,\Q)$ such that $\alpha\in V\otimes_\Q\R$.

There exists an $\epsilon>0$ such that for any $\alpha'\in H^2(X,\Q)$, if $||\alpha'-\alpha||<\epsilon$, then
\begin{itemize}
\item[(1)] $\alpha'\cdot C_j>0$ if $\alpha\cdot C_j>0$, and
\item[(2)] $\frac{1}{2}[\omega]+\alpha'-\alpha$ is positive on $\overline{\rm NA}(X)\backslash\{0\}$.
\end{itemize}

We now claim that $\langle\alpha_i, \overline{\rm NA}(X)\rangle\ge 0$ for any $\alpha_i\in V$ such that $||\alpha_i-\alpha||<\epsilon$, which implies that $\alpha$ is NQC.

Indeed we just need to check $\alpha_i$ is non-negative on $\overline {\rm NA}(X)_{K_X+B+\bbeta _X+\frac{1}{2}\omega\ge 0}$ and every $C_j$ in the above decomposition. The minimality of $V$ implies $C_j\cdot\alpha_i=0$ for any $C_j\cdot\alpha=0$. (1) and (2) above imply that $\alpha_i$ is nef on $\overline {\rm NA}(X)_{K_X+B+\bbeta _X+\frac{1}{2}\omega\ge 0}$ and $C_j$ for any $C_j\cdot\alpha>0$. 

\end{proof}

\begin{lemma}\label{lem: NQC can run alpha-trivial MMP}
Let $(X,B+\bbeta)$ be a compact strongly $\Q$-factorial K\"ahler gklt pair, $\alpha$ a closed $(1,1)$-form/current with local potentials such that $[\alpha]$ is NQC, then there exists a real number $t_0>0$ such that for any $t>t_0$,
if $\phi: X\dasharrow X'$ is a sequence of steps of the $(K_X+B+\bbeta_X+t\alpha)$-MMP (including a Mori fiber space), then $\phi$ is $\alpha$-trivial.
\end{lemma}

\begin{proof}
Let $f:X_0:=X\to Z$ be the contraction of a $(K_X+B+\bbeta_X+t\alpha)$-negative extremal ray $R$ generated by a rational curve $C$, and $[\alpha]=\sum_i a_i\alpha_i$ be the NQC decomposition. Fix $l_i\in \mathbb N^*$ such that $l_i\alpha_i\in H^2(X,\Z)$. Define 
\[
t_0:=\max_i\{\frac{l_i(2\dim X+1)}{a_i}\}.
\]
Notice that $R$ is also a $(K_X+B+\bbeta_X)$-negative extremal ray and so we may assume that $(K_X+B+\bbeta_X)\cdot C\ge -2\dim X$. If $\alpha_i\cdot C>0$, then we have 
\[
0>(K_X+B+\bbeta_X+t\alpha)\cdot C\ge -2\dim X+t_0a_i(\alpha_i\cdot C)\ge 1,
\]
a contradiction. Therefore $C$ is $\alpha_i$-trivial for all $i$ and hence also $\alpha$-trivial. 

If $f$ is a Mori fiber space, we are already done. If $f$ is a birational contraction we let $X_1:=Z$ when $f$ is divisorial and $X_1:=X^+$ when $f$ is small. In either case we let $g: X_1\to Z$ be the induced morphism. 

Notice that $X$ and $Z$ are in the Fujiki's class $\mathcal C$ (actually K\"ahler) and have rational singularities (for example, this can be checked by the Kawamata-Viehweg vanishing theorem),  
hence by Lemma \ref{l-rccfibres} we have injections $f^*: H^2(Z,\R)\to H^2(X,\R)$ and 
\[
f^*: H^{1,1}_{\rm BC}(Z)=H^1(Z, \mathcal{H}_Z)\hookrightarrow H^1(X,\mathcal{H}_X)=H^{1,1}_{\rm BC}(X).
\]
The image is given by
\[
{\rm Im} f^*=\{\beta\in H^2
(X, \R)~|~\beta\cdot C=0\}.
\]
Since $\alpha \in V\otimes \R\cap f^*H^2(Z,\R)$, then $\alpha \in (V\cap f^*H^2(Z,\Q ))\otimes_\Q \R$, and therefore $f^*H^2(Z,\Q)\supset V$ as $V\subset H^2(X,\Q)$ is the minimal $\Q$-space such that $\alpha \in V\otimes _\Q\R$. It follows that there exist $\tilde\alpha_i\in H^2(Z,\Q)$ such that $f^*\tilde\alpha_i=\alpha_i$.

We claim that \[
l_ig^*\tilde\alpha_i\cdot C_1\in \mathbb{N}~\text{ for any curve } C_1\subset X_1.
\] To see this, consider a common resolution of $p:W\to X_0$ and $q:W\to X_1$. Since $X_1$ has gklt type singularities, then the fibers of $q$ are rationally connected by \cite{HM07}. By \cite{GHS03}, there is a section $C_W\to C_1$. Let $C=p_*C_W$, then by the projection formula we have \[l_ig^*\tilde\alpha_i\cdot C_1=l_iq^*g^*\tilde\alpha_i\cdot C_W=l_i\alpha _i\cdot C.\]

Let $K_{X_1}+B_1+\bbeta_{X_1}+t\alpha_{X_1}$ be the push-forward of $K_X+B+\bbeta_X+t\alpha$ and $\alpha_{i,X_1}:=g^*\tilde\alpha_i$. Then we have $\sum_ia_i\alpha_{i,X_1}=[\alpha_1]$ and any $(K_{X_1}+B_1+\bbeta_{X_1}+t\alpha_{X_1})$-negative extremal ray is also $\alpha_{i,X_1}$-trivial for all $i$. Repeating the same argument we obtain $\alpha_{i,X_j}\in H^2(X_j,\Q)$ such that the extremal contraction on $X_j$ is $\alpha_{i,X_j}$-trivial for all $i,j$.

\end{proof}

\subsection{Semipositivity of the relative canonical bundle}
We recall the following result from \cite{HP24}.
\begin{theorem}\label{t-gue} Let $f:X\to Z$ be a surjective projective map with connected fibers between normal compact  K\"ahler varieties, $Z$ is smooth, $(X,B+\bbeta)$ a generalized pair which is log canonical over an open subset of $Z$, $\gamma$ a real (1,1)-class on $Z$, and $L$ an $\R$-line bundle which is pseudo-effective over $Z$.
If  $K_X+B+\bbeta_X =f^*\gamma +L$, then $ K_{X/Z}+B+\bbeta_X$  is pseudo-effective.
\end{theorem}
\begin{proof} We may pick $L_i$ a sequence of $\Q$-line bundles such that $\lim L_i=L$ and $A_i:=L_i-L$ is $f$-ample. Therefore we may write $K_X+B+\bbeta _X+A_i=f^*\gamma +L_i$.  
By \cite[Theorem 2.2]{HP24}, $K_{X/Z}+B+\bbeta _X+A_i$ is pseudo-effective, and taking the limit, one sees that $K_{X/Z}+B+\bbeta _X$ is pseudo-effective. 
\end{proof}
\begin{theorem}\label{t-psef} Let $f:X\to Y$ be a morphism of compact K\"ahler varieties with rationally connected general fibers, where $Y$ smooth. Let $(X,B+\bbeta )$ be a generalized pair which is gklt over an open subset of $Y$ and $B+\bbeta_X$ is modified big. 
\begin{enumerate}
\item If $K_X+B+\bbeta _X$ is pseudo-effective over $Y$ then $K_{X/Y}+B+\bbeta _X$ is pseudo-effective,
\item if $K_X+B+\bbeta _X$ is big  over $Y$, then $K_{X/Y}+B+\bbeta _X$ is big.\end{enumerate} 
\end{theorem}
\begin{proof}
(1) Let $\nu :X'\to X$ be a resolution and write $K_{X'}+B'+\bbeta _{X'}=\nu ^*(K_X+B+\bbeta _X)+E$ where $B'=\nu ^{-1}_*B\vee{\rm Ex}(\nu)$. Note that $B'+\bbeta _{X'}$ is big (Lemma \ref{l-modbig}), $K_{X'}+B'+\bbeta _{X'}$ is pseudo-effective over $Y$, the general fibers of $X'\to Y$ are rationally connected (see \S \ref{s-rc}), and if $K_{X'/Y}+B'+\bbeta _{X'}$ is pseudo-effective then so is $K_{X/Y}+B+\bbeta _X=\nu _*(K_{X'/Y}+B'+\bbeta _{X'})$. It follows that we may replace $(X,B+\bbeta )$ by $(X',B'+\bbeta)$ and hence we may assume that $X$ is smooth and $\bbeta  $ descends to $X$ so that $\bbeta _X$ is nef. 

Let $\omega$ be a K\"ahler class on $X$.
Since the pseudo-effective condition is closed, it suffices to show that $K_{X/Y}+B+\bbeta _X+\epsilon \omega$ is big  for any $\epsilon >0$. Note that since $K_X+B+\bbeta _X$ is pseudo-effective over $Y$, then $K_{X}+B+\bbeta _X+\epsilon \omega$ is big over $Y$ for any $\epsilon >0$. By Lemma \ref{l-rccfibres} $f^*:H^0(Y,\Omega ^2_Y)\to H^0(X,\Omega ^2_X)$ is an isomorphism.
By Lemma \ref{l-rel}, we may write $K_X+B+\bbeta _X+\epsilon \omega \equiv L+f^*\gamma $ where $L$ is an $\R$-line bundle big over $Y$ and $\gamma \in H^{1,1}(Y)$.  By Theorem \ref{t-gue}, $K_{X/Y}+B+\bbeta _X+\epsilon \omega $ is pseudo-effective. 

(2) Suppose now that $K_{X}+B+\bbeta _X$ is big over $Y$. Arguing as above, we may replace $X$ by a resolution and in particular we may assume that $X$ is smooth, and $\bbeta$ descends to $X$.
By Lemma \ref{lem: g-kltperturb}, we may assume that $\bbeta _X$ is K\"ahler.
Since $K_{X}+B+\bbeta _X$ is big over $Y$, then so is $K_{X}+B+(1-\epsilon)\bbeta _X$ for any $0<\epsilon \ll 1$.
By (1), $K_{X/Y}+B+(1-\epsilon)\bbeta _X$ is pseudo-effective and so 
$K_{X/Y}+B+\bbeta _X$ is big.

\end{proof}
\subsection{The canonical bundle formula}\label{s-canbundle}
Another important result that will be used frequently throughout this paper is the following version of the canonical bundle formula below which was communicated to us by M. Paun \cite{Pau26}. 
\begin{theorem}\label{t-gcbf} Let $f:X\to Z$ be a surjective projective morphism of normal compact K\"ahler varieties with connected fibers, $(X,B+\bbeta )$  a generalized  pair which is gklt over an open subset $U$ of $Z$, such that $B+\bbeta _X$ is modified big, $K_X+B+\bbeta _X\equiv f^*\gamma $ for some  $\gamma\in H^{1,1}_{\rm BC}(Z)$. Then there exists $(Z,\DDelta _Z+\ddelta )$  a generalized pair such that 
\begin{enumerate}
\item $K_Z+\DDelta _Z+\ddelta _Z\equiv \gamma$, 
\item $\ddelta _Z$ is modified big, and 
\item $(U,\DDelta |_U+\ddelta)$ is gklt. \end{enumerate} 
\end{theorem}
Note that we have $\DDelta_Z\geq 0$,
$\ddelta$ is b-nef, that is $\ddelta $ descends to some model $Z'$ so that $\ddelta_{Z'}$ is nef,
 and the coefficients of components of $\DDelta _{Z'}$ whose images intersect $U$ are $< 1$ for any birational model $Z'\to Z$.

\subsection{The universal push-out diagram}

We will now discuss the existence of push-out diagrams (see \cite[Chapter 8]{Kol97} for a brief introduction) in the analytic category.
\begin{lemma}[{cf. \cite[Theorem 9.30]{Kol13}}]\label{lem: push-out diagram for analytic}
Let $X$ be an analytic space with $Z\subset X$ a closed subspace and $g: Z\to V$ a finite morphism. Then there is a universal push-out diagram of analytic spaces
\[\begin{tikzcd}
Z \arrow[r,hook]{}{i}\arrow[d]{}{g} & X \arrow{d}{\pi} \\
V\arrow[r,hook]{}{j} & Y
\end{tikzcd}
\]
Furthermore, $\pi$ is finite, $j:V\to Y$ is a closed embedding, $Z=\pi^{-1}(V)$ and the natural map between the ideal sheaves $\mathcal{I}_{V\subset Y}\to \pi_*\mathcal{I}_{Z\subset X}$ is an isomorphism. Recall that the push-out diagram is universal means the sequence
\[
0\to \OO_{Y}\xrightarrow{\pi^*\oplus j^*} \pi_*\OO_{X}\oplus\OO_{V}\xrightarrow{i^*-g^*} g_* \OO_{Z}\to 0
\]
is exact.
\end{lemma}

\begin{proof}
\cite[Theorem 6.1]{Ar70} shows this in the algebraic space setting, but the same proof applies since the constructions and statements are local and the morphisms are finite.
\end{proof}

\subsection{Rational connectedness}
\label{s-rc}

Let $X$ be a normal compact analytic variety then we say that $X$ is rationally connected if any two general points can be connected by a rational curve and it is rationally chain connected if any two points can be connected by a chain of rational curves. It is known that smooth varieties are rationally chain connected iff they are rationally connected. Notice that being rationally chain connected is not a birational invariant property (as shown by cones over elliptic curves), however being rationally connected is a birationally invariant property. 
We also remark that a compact K\"ahler
variety with gklt singularities is rationally chain connected if and only if it
is rationally connected (see \cite[Corollary 1.5]{HM07} and \cite[Lemma 2.10]{HLX26}). 

Let $X$ be a normal compact analytic variety then we say that $X$ is uniruled if there exists a covering family of rational curves.
By Theorem \ref{t-ou}, we know that a K\"ahler manifold is uniruled iff $K_X$ is not pseudo-effective. 
The maximal rationally connected fibration of a compact K\"ahler manifold $X$ is a quasi-holomorphic map $f:X\dasharrow Y$ such that, general fibers of $f$ are rationally connected and there are no rational curves passing through a general point $y\in Y$ (see \cite[Remark 6.10]{CH20}).

\section{The contraction theorem: the  big case}\label{s-contr-big}

In this section, assuming Theorems (\ref{t-bpf})$_{n-1}$, (\ref{t-MMPscaling})$_{n-1}$, and (\ref{t-nkltcontraction})$_{n-1}$, we will prove Theorem \ref{t-nkltcontraction}$_{n,{\rm big}}$, i.e. Theorem \ref{t-nkltcontraction} in the case that $\alpha =K_X+B+\bbeta_X$ is big.

\begin{theorem}\label{t-ind1} Theorem \ref{t-nkltcontraction}$_{n-1}$ implies Theorem \ref{t-nkltcontraction}$_{n,{\rm big}}$ :
{\it Let $(X,B+\bbeta)$ be a $n$-dimensional compact K\"ahler generalized pair such that $B+\bbeta_X$ is modified big and $\alpha=K_X+B+\bbeta _X$ is nef, big, and NQC. Assume that the contraction defined by $[\alpha]|_V$ exists, where $\mathcal I_V=\mathcal J(X,B+\bbeta)$. 
Then the contraction defined by $[\alpha]$ exists and is birational.}\end{theorem}

The proof is somewhat involved and divided into several parts. We begin by setting up some notation.
\vspace{.5em}

\begin{notation}
Since $\alpha$ is nef and big, $\mathcal N={\rm Null}(\alpha)$ is a proper analytic subvariety (see Theorem \ref{thm:null-non-kahler}), moreover there is a K\"ahler current $\eta\in \alpha$ with singularities along ${\rm Null}(\alpha)$. We can choose a log resolution $\nu :X'\to X$ such that 

\begin{enumerate}
   \item $K_{X'}+B'+\bbeta_{X'}=\nu ^*(K_X+B+\bbeta_X)$ where $\bbeta =\overline {\bbeta_{X'}}$,
   \item $\nu ^*\alpha =D'+\eta'$ where $\eta'$ is K\"ahler, $D'$ has snc support and contains all the exceptional divisors,

   \item $\lambda _i,~i\ge 0$ are the jumping numbers for $(X,B+\bbeta;D+\overline{\eta'} )$ where $D=\nu _*D'$, so that $\lfloor B'+\lambda _i D'\rfloor=\lfloor B'+\lambda  D'\rfloor$ for $\lambda _{i+1}>\lambda \geq \lambda _i\ge \lambda_0=0$, 
   \item we may assume that $S _i=\lfloor B'+\lambda _{i} D'\rfloor-\lfloor B'+\lambda _{i-1} D'\rfloor,~i\ge 1$ is prime (by perturbing $D'$ if necessary),
   \item if we write $\lfloor B'+\lambda _i D'\rfloor=Z_i-E_i$ where $E_i,Z_i\geq 0$ and $E_i\wedge Z_i=0$, then $E_i$ is $\nu$-exceptional and $Z _i$ is the union of the non-klt places of $(X,B+\lambda _iD+\bbeta+\lambda _i\overline{\eta'} )$ on $X'$.
\end{enumerate}
\end{notation}

We will show that

\begin{proposition}\label{prop: construct f_i for Z_i}
For each $i\ge 0$, $V_i:={\rm Specan}_X(\nu_*\OO _{Z _i})$ is a closed subspace of $X$ and there exists a projective contraction of analytic/complex spaces $f_i:Z_i\to T_i$ such that
\begin{itemize}
\item $-D'|_{Z_i}$ is $f_i$-ample,
\item a curve $C$ is contracted if and only if $C\cdot (\nu^*\alpha)|_{Z_i}=0$,
\item $f_{i,*}\OO_{Z_i}=\OO_{T_i}$, i.e. the Stein factorization of $f_{i}$ is itself, and
\item $f_i$ factors through $h_i:V_i\to T_i$.
\end{itemize}
In particular, the contraction defined by $[\alpha]|_{V_i}$ exists for $V_i$.
\end{proposition}

\begin{proof}
For the first part of the statement, we have
\[E_i-Z_i\equiv K_{X'}+\{B'+\lambda_i D'\}+\bbeta_{X'}+\lambda_i\eta'-(1+\lambda_i)\nu^*\alpha,\]
$(X',\{B'+\lambda_i D'\})$ is klt and $\bbeta_{X'}+\lambda_i\eta'$ is nef, hence by relative Kawamata-Viehweg
vanishing theorem (see Theorem \ref{t-kvv}) we get $R^1\nu_*\OO_{X'}(E_i-Z_i)=0$ and a short exact sequence:
\[
0\to \nu_*\OO_{X'}(E_i-Z_i)\to\nu_*\OO_{X'}(E_i)\to\nu_*\OO_{Z_i}(E_i)\to 0.
\]
Since $\nu_*\OO_{X'}(E_i-Z_i)=\nu_*\OO_{X'}(-Z_i)$ and $\nu_*\OO_{X'}(E_i)=\nu_*\OO_{X'}=\OO_{X}$,  we have the following.
\begin{claim}\label{claim: surj from O_X}
$\OO_{V_i}=\nu_*\OO_{Z_i}\to\nu_*\OO_{Z_i}(E_i)$ is an isomorphism, and $\OO_X\to \OO_{V_i}$ is a surjection.
\end{claim}
\begin{proof}
This follows directly from Lemma \ref{l-ses}: $\OO_{Z_i}\to \OO_{Z_i}(E_i)$ is injective as $E_i\wedge Z_i=0$, hence $\nu_*\OO_{Z_i}\to\nu_*\OO_{Z_i}(E_i)$ is also injective. From the short exact sequence above we know that the composition $$\nu_*\OO_{X'}(E_i)=\nu_*\OO_{X'}\to \nu_*\OO_{Z_i}\to\nu_*\OO_{Z_i}(E_i)$$ is surjective, hence $\nu_*\OO_{Z_i}\to\nu_*\OO_{Z_i}(E_i)$ is also surjective.
\end{proof}
Therefore $V_i$ is a closed subspace of $X$ and the kernel of the surjection $\OO_X\to \OO_{V_i}$ is the multiplier ideal of $(X,B+\bbeta+\lambda_i(D+\overline{\eta'}))$, given by $\mathcal J(X,B+\bbeta+\lambda_i(D+\overline{\eta'}))=\nu _*\OO _{X'}(E_i-Z_i)$.
\vspace{.5em}

We will next prove the existence of $f_i$ by induction on $i\in\N$. We first need to show that $f_0$ exists.

By assumption, $[\alpha]|_{V}$ is endowed with a Moishezon contraction, say $h_0: V=V_0\to T_0$. From what we have seen above, we know that $\nu|_{Z_0}: Z_0\to V_0$ is a projective contraction, therefore $f_0:=h_0\circ\nu|_{Z_0}: Z_0\to T_0$ is the contraction defined by $[\alpha]|_{Z_0}$. Now, $f_0$ is Moishezon and it remains to show that $-D'|_{Z_0}$ is $f_0$-ample hence $f_0$ is projective. Since ampleness can be checked on each irreducible component (with the reduced structure) of $Z_0$, this will follow from Lemma \ref{lem: Moishezon to strong Moishezon} as we have $-D'|_{Z_0}\equiv^{alg}_{f_0}\eta'|_{Z_0}$.

The above argument also shows that if $f_i$ is the contraction defined by $[\alpha]|_{Z_i}$, then $f_i$ is projective.

\vspace{.5em}

 Suppose now that $f_{i-1}$ has been constructed and the following diagram commutes 

\[\begin{tikzcd}
& Z_{i-1}\arrow[dl, swap, "f_{i-1}"]\arrow[r,hook]{}{}\arrow[d]{}{} & X' \arrow{d}{\nu} \\
T_{i-1} & V_{i-1}\arrow[l]{}{h_{i-1}}\arrow[r,hook]{}{} & X
\end{tikzcd}
\]

There are three cases to consider:

\noindent{\bf Case 1.} $Z_{i}=Z_{i-1}$, then $f_{i}=f_{i-1}$, $V_{i}=V_{i-1}$, and $T_i:=T_{i-1}$.

\vspace{.5em}

\noindent{\bf Case 2.} $Z_i\neq Z_{i-1}$ and $\nu(S_i)$ is contained in $V_{i-1}$. Notice that in this case $E_i=E_{i-1}$, $S_i=Z_{i}-Z_{i-1}$, and $V_i,V_{i-1}$ have the same support.

We want to construct $f_i$ and $T_i$. Since $\nu(S_i)\subset V_{i-1}$, we can first set-theoretically define $f_i$ to be $h_{i-1}\circ\nu|_{Z_i}$, and we see $f_i|_{S_{i}}$ can be regarded as the composition of the induced morphisms $S_i\to V_{i,red}=V_{i-1,red}$ and $V_{i,red}\to T_{i-1,red}$, where the subscript $red$ means the reduction of the corresponding space. Since $f_{i-1}$ is projective, we get $h_{i-1}$ is Moishezon, and $f_{i}|_{S_{i}}=h_{i-1}\circ\nu|_{Z_i}$ is also Moishezon as $\nu$ is projective. Since $-D'|_{S_i}\equiv^{alg}_{f_i|_{S_i}}\eta'|_{S_i}$, then $f_i|_{S_{i}}$ is projective, by Lemma \ref{lem: Moishezon to strong Moishezon} again. Considering the short exact sequence of sheaves of abelian groups
\[
0\to\OO_{S_i}(E_i-Z_{i-1})\to\OO_{Z_{i}}(E_i)\to\OO_{Z_{i-1}}(E_i)\to 0.
\]
\begin{claim}
$R^1f_{i,*}\OO_{S_i}(E_i-Z_{i-1})=0$ and $f_{i,*}\OO_{Z_i}\to\OO_{T_{i-1}}$ is surjective as sheaves of rings on $T_{i-1,red}$.
\end{claim}

\begin{proof}
Since $V_{i-1}\to T_{i-1}$ is the contraction defined by $[\alpha]|_{V_{i-1}}$, we easily see that $C\subset S_i$ is contracted by $f_i|_{S_i}$ if and only if $C\cdot\nu^*\alpha=0$. Recall that
\[
(E_i-Z_{i-1})-(K_{X'}+\{B'+\lambda_i D'\}+S_i)\equiv \bbeta_{X'}+\lambda_i\eta'-(1+\lambda_i)\nu^*\alpha,
\]
so we have 
\[
(E_i-Z_{i-1})|_{S_i}-(K_{S_i}+\{B'+\lambda_i D'\}|_{S_i})\equiv_{f_i|_{S_i}}(\bbeta_{X'}+\lambda_i\eta')|_{S_i}. 
\]
 Lemma \ref{lem: L rel num eq to kahler => f-ample} implies that the left hand side is $f_i|_{S_i}$-ample, therefore by relative Kawamata-Viehweg 
vanishing (see Theorem \ref{t-kvv}) we have $R^1f_{i,*}\OO_{S_i}(E_i-Z_{i-1})=0$, which gives us the surjection 
$f_{i,*}\OO_{Z_{i}}(E_i)\to f_{i,*}\OO_{Z_{i-1}}(E_i)$ between sheaves of abelian groups. Since $f_{i}|_{Z_{i-1}}=f_{i-1}=h_{i-1}\circ\nu|_{Z_{i-1}}$, we have that 
\[
f_{i,*}\OO_{Z_i}=h_{i-1,*}\OO_{V_i}=h_{i-1,*}\nu_*\OO_{Z_i}(E_i)\to h_{i-1,*}\nu_*\OO_{Z_{i-1}}(E_i)=h_{i-1,*}\OO_{V_{i-1}}=\OO_{T_{i-1}}
\]
is surjective, where the middle equalities are given by Claim \ref{claim: surj from O_X} and $E_i=E_{i-1}$.
\end{proof}
Now we see $f_{i,*}\OO_{S_i}(E_i-Z_{i-1})=h_{i-1,*}\mathcal{I}_{V_{i-1}\subset V_i}$ is the kernel of $f_{i,*}\OO_{Z_i}\to \OO_{T_{i-1}}$, and it can also be regarded as a coherent $\OO_{T_{i-1,red}}$ module. Since $\mathcal{I}_{V_{i-1}\subset V_i}$ is nilpotent, then $f_{i,*}\OO_{S_i}(E_i-Z_i)$ is a nilpotent ideal sheaf of the sheaf of rings $f_{i,*}\OO_{Z_i}$ on the space $T_{i-1,red}$. We can therefore define $T_i$ to be the thickening of $T_{i-1}$ (sharing the same underlying topology space $T_{i-1,red}$) with structure sheaf $\OO_{T_i}:= f_{i,*}\OO_{Z_i}$ and we get a contraction $f_i: Z_i\to T_i$. It is straightforward to verify that $f_i$ contracts a curve iff it is $[\alpha]$-trivial, and $f_i$ is Moishezon since both $f_i|_{Z_{i-1}}$ and $f_i|_{S_i}$ are. Therefore $f_i$ is the contraction defined by $[\alpha]$. As mentioned above, $-D'$ is $f_i$-ample and $f_i$ is projective.

\vspace{.5em}

\noindent{\bf Case 3.} $Z_i\neq Z_{i-1}$ and $\nu(S_i)$ is not contained in $V_{i-1}$, then by induction and adjunction we will show that there is a projective contraction $g_i: S_i\to \tilde{T_i}$ which is defined by $[\alpha]|_{S_i}$.

Let $\hat X$ be a $\Q$-factorial dlt model of $(X,B+\bbeta+\lambda_i(D+\overline{\eta'}))$ (see Theorems \ref{t-dltmodel} and \ref{t-small Q factorial model}) such that $S_i$ is a divisor on $\hat X$, denoted by $\hat S_i$ and $(\hat X,\hat B_i+\bbeta+\lambda_i\overline{\eta'})$ be the crepant pullback. Let $(X'',B''_i+\bbeta+\lambda_i\overline{\eta'})$ be a (crepant) common log resolution of $(\hat X,\hat B_i+\bbeta+\lambda_i\overline{\eta'})$ and $(X',B'+\bbeta+\lambda_i(D'+\overline{\eta'}))$, and $S''_i$ be the birational transform of $S_i$ on $X''$. 

\[\begin{tikzcd}
S_i \arrow[d,hook] & S''_i\arrow[l, swap, "q"]\arrow[r,"p"]\arrow[d,hook] & \hat S_i \arrow[d,hook] \arrow[r,"\hat g_{i}"]& \tilde T_i\\
X'\arrow[dr, swap, "\nu"] & X''\arrow[d, "\nu''"]\arrow[l, swap, "q'"]\arrow[r,"\hat p"] & \hat X\arrow[dl,"\hat\nu"] & \\
& X  & &
\end{tikzcd}
\]
By adjunction we obtain a generalized pair $(\hat S_i,\hat\Delta_i+\ggamma)$, where
\[
K_{\hat S_i}+\hat\Delta_i+\ggamma_{\hat S_i}=(K_{\hat X}+\hat B_i+\bbeta_{\hat X}+\lambda_i\overline{\eta'}_{\hat X})|_{\hat S_i}
\]
and crepant birational pairs $(S_i,\Delta_i+\ggamma)$ and $(S''_i,\Delta''_i+\ggamma)$, and we have $p:S''_i\to \hat S_i$ and $q:S''_i\to S_i$. Then, by relative Kawamata Viehweg vanishing (see Theorem \ref{t-kvv}), we have a short exact sequence
\[
0\to p_*\OO_{S''_i}(E_{S''_i}-Z_{S''_i})\to p_*\OO_{S''_i}(E_{S''_i})\to p_*\OO_{Z_{S''_i}}(E_{S''_i})\to 0,
\]
where $Z_{S''_i}=\lfloor\Delta''_i\rfloor^{\ge 0}$ and $E_{S''_i}=\lfloor \Delta''_i\rfloor^{\le 0}$ so that $E_{S''_i}-Z_{S''_i}\equiv _X K_{S_i''}+\{\Delta _i''\}+\ggamma _{S_i''}$. Since $\hat S_i$ is normal and $\hat B_i\geq 0$, it follows that $\hat \Delta _i\geq 0$ and hence $E_{S''_i}$ is $p$-exceptional so that $p_*\OO _{S_i''}(E_{S''_i})=p_*\OO _{S_i''}=\OO _{\hat S_i}$. 
It then follows, by Lemma \ref{l-ses}, that $p_*\OO_{S''_i}(E_{S''_i}-Z_{S''_i})=p_*\OO_{S''_i}(-Z_{S''_i})$ is the multiplier ideal of $(\hat S_i,\hat\Delta_i+\ggamma)$, $R^1p_*\OO_{S''_i}(E_{S''_i}-Z_{S''_i})=0$, and so $p_*\OO_{Z_{S''_i}}(E_{S''_i})=p_*\OO_{Z_{S''_i}}=\OO_{\hat V_i}$, where $\hat V_i$ is the closed subspace of $\hat S _i$ defined by this multiplier ideal. 

Let $Z_{S_i}:=Z_{i-1}\cap S_i$, then we can see that
\begin{claim}
$Z_{S''_i}\le q^* Z_{S_i}$, in particular we have a morphism $q|_{Z_{S''_i}}: Z_{S''_i}\to Z_{S_i}$.
\end{claim}
\begin{proof}
Since $\ggamma$ descends on $S_i$ and $(S_i,\{\Delta_i\})$ is snc, we have $q^*(K_{S_i}+\Delta_i)=K_{S''_i}+\Delta''_i$ and $\lfloor q^*(K_{S_i}+\{\Delta_i\})\rfloor-K_{S''_i}\le 0$. Therefore we get
\[
q^*Z_{S_i}\ge \lfloor q^*(K_{S_i}+Z_{S_i}+\{\Delta_i\})-K_{S''_i}\rfloor\ge \lfloor q^*(K_{S_i}+\Delta_i)-K_{S''_i}\rfloor =\lfloor \Delta''_i\rfloor,
\]
taking the positive parts we have $q^*Z_{S_i}\ge \lfloor\Delta''\rfloor^{\ge 0}=Z_{S''_i}$.
\end{proof}
By our inductive hypothesis we have $f_{i-1}:Z_{i-1}\to T_{i-1}$. Thus, we get a projective morphism $f_{i-1}|_{Z_{S_i}}\circ q|_{Z_{S''_i}}: Z_{S''_i}\to T_{i-1}$, and let $Z_{S''_i}\to Z_{\tilde T_i}\to T_{i-1}$ be the Stein factorization. Then one can verify, by the projection formula, that $Z_{S''_i}\to Z_{\tilde T_i}$ is the contraction defined by $[\alpha]|_{Z_{S''_i}}$. Since $p_*\OO_{Z_{S''_i}}=\OO_{\hat V_i}$ and $[\alpha]|_{Z_{S''_i}}=p^*[\alpha]|_{\hat V_i}$, the contraction factors through $\hat V_i$ and the induced $\hat V_i\to Z_{\tilde T_i}$ is the contraction defined by $[\alpha]|_{\hat V_i}$. We also note that $[\alpha]|_{\hat S_i}$ is NQC since so is $[\alpha]$.

By Theorem \ref{t-nkltcontraction}$_{n-1}$, we know there is a Moishezon contraction $\hat g_i: \hat S_i\to \tilde T_i$ with $\hat g_{i,*}\OO_{\hat S_i}=\OO_{\tilde T_i}$, which is defined by $[\alpha]|_{\hat S_i}$ . Composing with $S''_i\to \hat S_i$ we get a projective contraction $g''_i:S''_i\to \tilde T_i$ (by Lemma \ref{lem: Moishezon to strong Moishezon}). We have the following commutative diagram (where the inclusion $Z_{\tilde T_i}\hookrightarrow \tilde T_i$ follows from the claim below)
\[\begin{tikzcd}
 Z_{S''_i}\arrow[r]\arrow[d,hook] & \hat V_i \arrow[d,hook] \arrow[r]& Z_{\tilde T_i}\arrow[d,hook] \\
 S''_i\arrow[r,"p"] & \hat S_i\arrow[r,"\hat g_i"] & \tilde T_i
\end{tikzcd}
\]

\begin{claim}
$\OO_{\tilde{T_i}}= g''_{i,*}\OO_{S''_i}\to g''_{i,*}\OO_{Z_{S''_i}}$ is surjective, and the corresponding closed subspace of $\tilde{T_i}$ is exactly $Z_{\tilde T_i}$, with $\mathcal{I}_{Z_{\tilde T_i}\subset \tilde T_i}=g''_{i,*}\OO_{S''_i}(-Z_{S''_i})=g''_{i,*}\OO_{S_i}(E_{S''_i}-Z_{S''_i})$.
\end{claim}
\begin{proof}
By  relative Kawamata-Viehweg vanishing (see Theorem \ref{t-kvv}) we have the short exact sequence
\[
0\to g''_{i,*}\OO_{S''_i}(E_{S''_i}-Z_{S''_i})\to g''_{i,*}\OO_{S_i}(E_{S''_i})\to g''_{i,*}\OO_{Z_{S''_i}}(E_{S''_i})\to 0.
\]
Using $g''_{i}=\hat g_{i}\circ p$ and applying Lemma \ref{l-ses} to $p$ we can rewrite the above short exact sequence as
\[
0\to \hat g_{i,*}\mathcal{I}_{\hat V_i}\to \hat g_{i,*}\OO_{\hat S_i}\to \hat g_{i,*}\OO_{\hat V_i}\to 0,
\]
which gives us 
\[
0\to g''_{i,*}\OO_{S''_i}(-Z_{S''_i})\to g''_{i,*}\OO_{S''_i}\to g''_{i,*}\OO_{Z_{S''_i}}\to 0.
\]

\end{proof}

Let $\lfloor B''_i\rfloor =Z''_i-E''_i$ and $Z''_{i-1}=\lfloor B''_i-\delta q'^* D' \rfloor^{\ge 0}$ be the non-klt locus of the crepant pullback of $(X',B'+(\lambda_i-\delta)(D'+\bar\eta')+\bbeta)$ for some $0<\delta\ll 1$, then by our inductive assumption and Claim \ref{claim: surj from O_X} we have $\nu''_*\OO_{Z''_{i-1}}=\OO_{V_{i-1}}$, and the morphism 
$$f''_{i-1}=h_{i-1}\circ\nu''|_{Z''_{i-1}}:Z''_{i-1}\to T_{i-1}.$$ 

Let $Z^-_{i}:=Z''_{i}-S''_i$. We will show the following.

\begin{claim}
There exists $h^-_{i}: V^-_{i}\to T^-_{i}$ such that $h^-_{i,*}\OO_{V^-_i}=\OO_{T^-_i}$, where $V^-_i$ is a closed subspace on $X$ defined by $\mathcal{I}_{V^-_i}:=\nu''_*\OO_{X''}(-Z^-_i)$ and $\nu''_*\OO_{Z^-_i}(E''_i)=\nu''_*\OO_{Z^-_i}=\OO_{V^-_i}$.
\end{claim}
\begin{proof}

By perturbing $q'^*(D'+\eta')$ and rewriting it as $D''+\eta''$ so that $D''-q'^*(D')$ is a very small effective $q'$-exceptional divisor and $\eta''$ is K\"ahler, we may assume that there exist numbers 
$$0<\delta_1<\cdots<\delta_n<\delta$$
corresponding to the jumping thresholds for $\lfloor B''_i-q'^*\delta D'+t D'' \rfloor$, such that each jump corresponds to a prime divisor. It is easy to see that $q'(Z''_i-S''_i)\subset Z_{i-1}$ (as sets). Indeed, $(X',B'+\bbeta+\lambda_i(D'+\overline{\eta'})-Z_{i-1})$ is sub-plt, i.e. there are no non-klt places other than $S_i$, and so $Z''_i-q'^*Z_{i-1}\le S_i''$. 

Therefore $\nu''(Z''_i-S''_i)\subset V_{i-1}$. By running the argument from Case 2, starting from $f''_{i-1}$ and $h_{i-1}$ we obtain a morphism $h^-_{i}$ satisfying our requirements. In addition, considering the last jumping number, we obtain $$\nu''_*\OO_{Z^-_i}(E''_i)=\nu''_*\OO_{Z^-_i}=\OO_{V^-_i}$$ as in Claim \ref{claim: surj from O_X}.

\end{proof}

We display the following diagram for the reader's convenience:

\[\begin{tikzcd}
  Z''_{i-1}\arrow[r,hook]\arrow[d] & Z^-_i \arrow[d] \arrow[r,hook]& Z''_i\arrow[d]\arrow[r,hook] & X''\arrow[d,"\nu''"] \\
 V_{i-1}\arrow[r,hook]\arrow[d,"h_{i-1}"] & V^-_i\arrow[r,hook]\arrow[d,"h^-_i"] & V_i\arrow[r,hook] & X \\
 T_{i-1}\arrow[r,hook] & T^-_i & & 
\end{tikzcd}
\]

Let $f^-_{i}: Z^-_i\to T^-_i$ be the composition of $\nu''|_{Z^-_i}$ and $h^-_i$, and recall that $Z_{S''_i}=Z^-_i\cap S''_i$,

\begin{claim} The Stein factorization of $f^-_i|_{Z_{S''_i}}: Z_{S''_i}\to T^-_i$ is given by $Z_{S''_i}\to Z_{\tilde T_i}\to T^-_i$. In particular, there is a finite morphism $Z_{\tilde T_i}\to T^-_{i}$ whose composition with $Z_{S''_i}\to Z_{\tilde T_i}$ equals to $f^-_i|_{Z_{S''_i}}$.
\end{claim}
\begin{proof}
Let $g''_i|_{Z_{S''_i}}\times f^-_i|_{Z_{S''_i}}: Z_{S''_i}\to \tilde T_{i}\times T^-_i$ and $\tilde W_i\to \tilde T_{i}\times T^-_i$ be the Stein factorization, then the composition with the projections $\tilde W_i\to \tilde T_i$ and $\tilde W_i\to T^-_i$ are both finite (onto their images) by our property of the morphisms (i.e. the curves contracted by these morphisms are precisely the $\alpha$-trivial curves $\int_C \alpha=\alpha\cdot C=0$). By definition, $\tilde W_i$ is both the Stein factorization of $g''_{i}|_{Z_{S''_i}}$ and $f^-_i|_{Z_{S''_i}}$
\[
g''_{i}|_{Z_{S''_i}}: Z_{S''_i}\to \tilde W_i\to \tilde T_i~~\text{ and }~ f^-_i|_{Z_{S''_i}}: Z_{S''_i}\to \tilde W_i\to T^-_i,
\]
and so we have $Z_{\tilde T_i}=\tilde W_i$ from the construction of $Z_{\tilde T_i}$.
\end{proof}

Let $T_i$ be the push-out of the finite morphism $Z_{\tilde T_i}\to T^-_{i}$ and the closed embedding $Z_{\tilde T_i}\to\tilde{T}_i$, then by Lemma \ref{lem: push-out diagram for analytic} we get a commutative diagram

\[\begin{tikzcd}
Z_{\tilde T_i} \arrow[r,hook]\arrow[d] & \tilde T_i \arrow{d}{\pi} \\
T^-_i\arrow[r,hook,"j"] & T_i
\end{tikzcd}
\]
such that $\pi$ is finite, $j:T^-_i\to T_i$ is a closed embedding,
and the natural map between the ideal sheaves $\mathcal{I}_{T^-_i\subset T_i}\to\pi_* \mathcal{I}_{Z_{\tilde T_i}\subset \tilde T_i}$ is an isomorphism.
 
\begin{claim}
We get a morphism $f''_i: Z''_i\to T_i$ such that $f''_i$ factors through $V_i\subset X$ and $f''_*\OO_{Z''_i}=\OO_{T_i}$.
\end{claim}

\begin{proof}
We have the following push-out diagram on $X''$
\[\begin{tikzcd}
Z_{S''_i} \arrow[r,hook]\arrow[d,hook] & S''_i \arrow[d,hook] \\
Z^-_i\arrow[r,hook] & Z''_i
\end{tikzcd}
\]
Consider the morphisms $f^-_i:Z_i^-\to T_i^-$, $g''_i: S''_i\to \tilde T_i$. The previous claim says that they share the same Stein factorization $Z_{S''_i}\to Z_{\tilde T_i}$ when restricting to $Z_{S''_i}$, therefore by the universal property of the push-out we get an induced morphism $f''_i:Z''_i\to T_i$. Similar arguments show that $f''_i$ factors through $V_i$ as $h_i:V_i\to T_i$. Arguing as in Claim \ref{claim: surj from O_X} we have $$\nu''_*\OO_{Z''_i}(E''_i)=\nu''_*\OO_{Z''_i}=\OO_{V_i}.$$
Consider the short exact sequence
\[
0\to \OO_{S''_i}(E''_i-Z^-_i)\to\OO_{Z''_{i}}(E''_i)\to\OO_{Z^-_i}(E''_i)\to 0.
\]
Pushing forward by $f''_i$, since $\pi$ is finite,  by the vanishing theorem \ref{t-kvv}, we get $R^1f''_{i,*}\OO_{S''_i}(E''_i-Z^-_i)=\pi_*R^1g''_{i,*}\OO_{S''_i}(E''_i-Z^-_i)=0$ and hence we have a short exact sequence
\[
0\to \pi_*g''_{i,*}\OO_{S''_i}(E''_i-Z^-_i)\to f''_{i,*}\OO_{Z''_{i}}(E''_i)\to f^-_{i,*}\OO_{Z^-_i}(E''_i)\to 0.
\]
Recall that we have 
\begin{itemize}
\item $f^-_{i}=h^-_{i}\circ\nu''|_{Z^-_i}$, $f''_i=h_i\circ \nu''|_{Z''_i}$,
\item $\OO_{S''_i}(E''_i-Z^-_i)=\OO_{X''}(E''_i-Z^-_i)|_{S''_i}=\OO_{S''_i}(E_{S''_i}-Z_{S''_i})$,
\item $g''_{i,*}\OO_{S''_i}(E''_i-Z^-_i)=\mathcal{I}_{Z_{\tilde T_i}\subset \tilde T_i}$, $\nu''_*\OO_{Z^-_i}(E''_i)=\OO_{V^-_i}$,
\end{itemize}
hence we get 
\[
0\to \pi_*\mathcal{I}_{Z_{\tilde T_i}\subset \tilde T_i}=\mathcal{I}_{T^-_i\subset T_i}\to h_{i,*}\OO_{V_i}\to h^-_{i,*}\OO_{V^-_i}=\OO_{T^-_i}\to 0
\]
and we deduce that $\OO_{T_i}\to h_{i,*}\OO_{V_i}=f''_{i,*}\OO_{Z''_i}$ is an isomorphism.
\end{proof}

Let $f_i:Z_i\to T_i$ be the composition of $\nu|_{Z_i}:Z_i\to V_i$ and $h_i: V_i\to T_i$, then it is easy to see that $f_{i,*}\OO_{Z_i}=\OO_{T_i}$ and $f_i$ contracts any curve $C$ iff $\int_C\alpha=0$. As above, $f_i$ is the contraction defined by $[\alpha]|_{Z_i}$, $f_i$ is projective and $-D'$ is $f_i$-ample.
\end{proof}
\vspace{.5em}

Now we can prove Theorem \ref{t-ind1} by extending $f_i$ to $X'$:

\begin{proof}[Proof of Theorem \ref{t-ind1}]
By Proposition \ref{prop: construct f_i for Z_i} we have defined $f_i:Z_i \to T_i$. We claim that we can extend $f_i$ to a proper morphism $f:X'\to Y$ such that $f|_{Z_i}=f_i$.
To see this, note that for any subvariety $Z\subset X'$ whose support is contained in ${\rm Supp}(D')$, then 
there is an inclusion $Z\hookrightarrow Z _i$ for any $i\gg 0$ and hence an induced holomorphic map $f_Z:=f_i|_Z$ (which is independent of $i$). In particular, we let $f^r:=f_i|_{\Supp(D')}$ for $i\gg 0$.

From Proposition \ref{prop: construct f_i for Z_i} we know $-D'|_{Z_i}$ is $f_i$-ample for any $i\ge 0$. We can approximate $D'$ with an effective $\Q$-divisor $\hat D$ such that $-\hat D|_{{\rm Supp}(D')}$ is $f^r$-ample. We then let $G=k\hat D$ for some sufficiently divisible $k>0$ so that $G$ is an integral divisor supported on $X'$ such that
\begin{itemize}
\item ${\rm Supp}(G)={\rm Supp}(D')$ and,
\item $-G|_{{\rm Supp}(D')}$ is $f^r$-ample. 
\end{itemize}

Notice that $\OO _G(-G)$ is $f_G$-ample since ampleness can be checked on the reduced space. By Serre vanishing and an argument on short exact sequences (see \cite[Proof of Claim 5.10]{DH23}), after replacing $G$ by a multiple, we may assume 
\[
R^1f_{G,*}\OO_G(-kG)=0, \text{ for all } k>0.
\]
Then by \cite[Theorem 2]{Fuj75} (see also \cite[Theorem 4.2]{DH23}) it follows that $f_G$ extends to a proper birational morphism $f:X'\to Y$ such that 
\[
f|_G=f_G \text{ and } f|_{X'\setminus {\rm Supp}(G)}={\rm id}_{X'\setminus {\rm Supp}(G)} .
\]

\end{proof}

\section{The transcendental base-point-free theorem}\label{s-semiample}

In this section we will prove the following base-point-free theorem, and we will show that  we can run the MMP for gklt K\"ahler varieties.

\begin{theorem}\label{t-ind2}Theorems \ref{t-bpf}$_{n-1}$, \ref{t-MMPscaling}$_{n-1}$, and \ref{t-nkltcontraction}$_{n,{\rm big}}$ imply Theorem \ref{t-bpf}$_{n}$: {\it 
Let $(X,B+\bbeta)$ be a compact K\"ahler, $n$-dimensional gklt pair such that  $\alpha=K_X+B+\bbeta _X$ is nef and $B+\bbeta_X$ is modified big. Then $\alpha$ is semiample.}
\end{theorem}

\begin{proof} By taking a small modification (cf. Theorem \ref{t-small Q factorial model} and Lemma \ref{l-modbig}) we may assume $X$ is strongly $\Q$-factorial. 

\vspace{.5em}
\noindent\textbf{Case 1:} Theorem \ref{t-ind2} holds if $\alpha$ is big.
By Lemma \ref{lem: g-kltperturb}, it follows easily that we may assume there is a gklt pair $(X,G+\ggamma )$ and a K\"ahler form $\omega$ such that $K_X+B+\bbeta _X\equiv K_X+G+\ggamma_X+\omega$.  By Lemma \ref{lem: gklt pair+Kahler is NQC} $\alpha$ is NQC. By Theorem \ref{t-nkltcontraction}$_{n,{\rm big}}$ (see Theorem \ref{t-ind1}), the contraction $f:X\to Z$ defined by $[\alpha]$ exists and is birational, hence automatically strongly Moishezon. By Lemma \ref{l-birgklt} we have $[\alpha]=f^*[K_Z+B_Z+\bbeta_Z]\in H^{1,1}_{\rm BC}(Z)$, where $K_Z+B_Z+\bbeta_Z=f_*(K_X+B+\bbeta_X)$ and $(Z,B_Z+\bbeta)$ is a gklt pair. Since $f: X\to Z$ is defined by $[\alpha]$, we know there are no $[K_Z+B_Z+\bbeta_Z]$-trivial curves on $Z$. By Theorem \ref{t-bigKahler},  $[K_Z+B_Z+\bbeta_Z]$ is a K\"ahler class and so $\alpha$ is semiample. 
\vspace{.5em}

\noindent\textbf{Case 2:} Theorem \ref{t-ind2} holds if $\alpha$ is not big.

Since $B+\bbeta _X$ is big, then $K_X$ is not pseudo-effective. By Theorem \ref{t-ou}, $X$ is uniruled. Let $f:X\dasharrow Z $ be the MRC. 
We may assume that $Z$ is smooth and K\"ahler.  
Let $\nu :X'\to X$ be a resolution such that
\begin{enumerate}
\item $f'=f\circ \nu:X'\to Z$ is a projective morphism \cite{CH24},
\item $\alpha ':=\nu ^*\alpha\equiv \nu ^*(K_X+B+\bbeta _X)=K_{X'}+B'+\bbeta _{X'}$.
\end{enumerate}
By Lemma \ref{lem: g-kltperturb}, we may assume that $\bbeta _{X'}$ is K\"ahler.
Write $B'=B'^{\ge  0}-B'^{\le 0} $ where $B'^{\ge 0},B'^{\le 0}\geq 0$ and $B'^{\ge 0}\wedge B'^{\le 0}=0$. Let $\Delta' :=B'^{\ge 0}$, and $F':=B'^{\le 0}$. 
 
\begin{claim}  We may run the $\alpha'$-trivial $(K_{X'}+\Delta '+\bbeta _{X'})$-MMP over $Z$ and its outcome $ X'\dasharrow X''\to Z''$
 is a good minimal model for $K_{X'}+\Delta '+\bbeta _{X'}+a\alpha '$ over $Z$ for any $a\gg 0$, such that  $K_{X''}+\Delta ''+\bbeta _{X''}+a\alpha ''$ is semiample over $Z$ and the corresponding morphism $g:X''\to Z''$ is $\alpha''$-trivial.
 In particular $F''$ is $g$-vertical.
\end{claim}
\begin{proof} 
As observed above, Lemma \ref{lem: gklt pair+Kahler is NQC} implies that $\alpha$ is NQC.  By Lemma \ref{lem: NQC can run alpha-trivial MMP}, since $a\gg 0$, if $\eta=a\alpha ' $, then
any sequence of steps of the $(K_{X'}+\Delta '+\bbeta _{X'}+\eta )$-MMP is $\eta$-trivial. Note that $(K_{X'}+\Delta '+\bbeta _{X'}+\eta )\equiv (a+1)\alpha '+F'$ is pseudo-effective and so this MMP can not end with a Mori-fiber space.
By Lemmas \ref{l-rel} and \ref{l-rccfibres}, $\bbeta _{X'}+\eta\equiv _Z c_1(N)$ where $N$ is an $\R$-Cartier divisor. By Proposition \ref{p-relMMP}, we can run the $(K_{X'}+\Delta '+\bbeta _{X'}+\eta )$-MMP with scaling over $Z$ which ends with a good log terminal model.
Let \[X'\dasharrow X_1 \dasharrow \ldots \dasharrow X_n=X''\] be the corresponding sequence of $\alpha'$-trivial $K_{X'}+\Delta '+\bbeta _{X'}$ flips and divisorial contractions over $Z$.

If $\phi:X'\dasharrow X''$ is a log terminal model (over $Z$), then $K_{X''}+\Delta ''+\bbeta _{X''}+\eta''$ is semiample (over $Z$) and hence we have a morphism $g :X''\to Z''$ over $Z$ so that $K_{X''}+\Delta ''+\bbeta _{X''}+\eta''=g ^*\eta _{Z''}$ where $\eta_{Z''}$ is K\"ahler over $Z$. Note that for $\epsilon >0$, $K_{X''}+\Delta ''+\bbeta _{X''}+(a+\epsilon)\alpha ''\equiv _{Z''}\epsilon \alpha ''$ is also nef over $Z''$ and hence semiample over $Z''$ and if $h:X''\to Z'''$ is the K\"ahler model over $Z''$, then $\alpha _{Z'''}$ is K\"ahler over $Z''$ and so $h$ is the K\"ahler model for $K_{X''}+\Delta '''+\bbeta _{X''}+(a+\epsilon)\alpha ''$  over $Z$ for $0<\epsilon \ll 1$ and
$\alpha ''\equiv _{Z''} h^*\alpha _{Z'''}$. But it then follows easily that $h$ is the K\"ahler model for  $K_{X''}+\Delta ''+\bbeta _{X''}+(a+t)\alpha ''$ over $Z$ for any $t>0$.
Since $a\gg 0$, we may assume that $Z''=Z'''$ and $\alpha'' =g^*\alpha _{Z''}$.

Finally note that as $f$ is quasi-holomorphic, 
$F'|_{X'_z}$ is $\nu _z$-exceptional where $\nu _z:X'_z\to X_z$ is the corresponding morphism on general fibers. It follows that $F'_z=N(F'_z+a\alpha '_z)$ and hence $F'_z$ is contracted by $X'_z\dasharrow X''_z$ (here $F'_z=F'|_{X'_z}$ and $\alpha '_z=\alpha '|_{X'_z}$).
\end{proof}

Therefore, from now on we will assume that 
$K_{X''}+\Delta ''+\bbeta _{X''}+a\alpha ''$ is semiample over $Z$ for any $a\gg 0$.

 \begin{claim} $K_{X''}+\Delta ''+\bbeta _{X''}+\eta''$ is not big over $Z$. 
 \end{claim}
 \begin{proof} 
 By Theorem \ref{t-psef}, if $K_{X''}+\Delta ''+\bbeta _{X''}+\eta''$ is big over $Z$ then $K_{X''/Z}+\Delta ''+\bbeta _{X''}+\eta ''$ is big. 
 Since $K_{Z}$ is pseudo-effective, it follows that  $K_{X''}+\Delta ''+\bbeta _{X''}+\eta''$ is big.
 Since $X'\dasharrow X''$ is part of a $(K_{X'}+\Delta '+\bbeta _{X'}+\eta')$-MMP, then
 \[ (a+1)\nu ^*\alpha +F'\equiv K_{X'}+\Delta '+\bbeta _{X'}+\eta'\] is big.
 Since $F'$ is $\nu$-exceptional, then $\alpha$ is big, contradicting our assumptions.
 \end{proof}
 Since $\alpha'$ is not big over $Z$, then $\dim (X''/Z'')>0$. By the canonical bundle formula (Theorem \ref{t-gcbf}), we have \[K_{X''}+\Delta ''+\bbeta _{X''}\equiv g ^*(K_{Z''}+\Delta _{Z''}+\ggamma_{Z''})\] where $(Z'',\Delta _{Z''}+\ggamma )$ is gklt, $\ggamma _{Z''}$ is big, and $K_{Z''}+\Delta _{Z''}+\ggamma_{Z''}+a\alpha ''$ is K\"ahler over $Z$ for any $a\gg 0$.  Since 
 $F''$ is vertical over $Z''$, we can write
 \[\alpha _{Z''}=K_{Z''}+\Delta _{Z''}+\ggamma_{Z''}-F_{Z''}\]
 where $g^*(F_{Z''})=F''$. 
 \begin{claim} $K_{Z''}+\Delta _{Z''}+\ggamma_{Z''}+a\alpha _{Z''}$ is pseudo-effective.
 \end{claim}
 \begin{proof} This follows from Theorem \ref{t-psef}, since $K_Z$ is pseudo-effective and $K_{Z''}+\Delta _{Z''}+\ggamma_{Z''}+a\alpha_{Z''}$ is K\"ahler over $Z$.
 \end{proof}

Next let $(\hat Z'',\hat\Delta''+\ggamma)\to Z''$ be a small $\Q$-factorial modification (cf. Theorem \ref{t-small Q factorial model}) and $\hat\nu:W\to X''$ be a log resolution such that $\hat h: W\to \hat Z''$ is a projective morphism, then let 
\[
K_W+\Delta_W+\bbeta_W+a\alpha_W=\hat \nu^*(K_{X''}+\Delta''+\bbeta_{X''}+a\alpha'').
\] Since \[K_W+\Delta^{\ge 0}_W+\bbeta_W+a\alpha_W+\epsilon{\rm Ex}(\nu') \equiv_{\hat Z''} \Delta^{\leq 0}_W+\epsilon{\rm Ex}(\nu') \] then by Proposition \ref{p-relMMP}
we can run  $(K_W+\Delta^{\ge 0}_W+\bbeta_W+a\alpha_W+\epsilon{\rm Ex}(\hat \nu))$-MMP over $\hat{Z}''$,  which terminates with a log terminal model $\hat g: \hat{X''}\to \hat Z''$. Then
\[
K_{\hat X''}+(\hat\Delta '')^{\geq 0}+\bbeta _{\hat X''}+a\hat\alpha''\equiv \hat g ^*(K_{\hat Z''}+\Delta _{\hat Z''}+\ggamma_{\hat Z''}+a\alpha_{Z''})
+(\hat\Delta '')^{\leq 0}+\epsilon {\rm Ex}(\hat \nu)\] is nef over $\hat Z''$. 
\begin{claim} $(\hat\Delta '')^{\leq 0}+\epsilon {\rm Ex}(\hat \nu)=0$.\end{claim}
\begin{proof}
Since $\varphi:\hat Z''\to Z''$ is a small birational map, then there is a big open subset $U\subset Z''$ such that $\hat U:=\varphi ^{-1}(U)\to U$ is an isomorphism.
Since both $X''_U\to U$ and $\hat X''_{\hat U}\to \hat U$ are minimal models for $K_{X''}+\Delta ''+\bbeta _{X''}$ over $U$ it follows that  $\psi:   \hat X''_{\hat U}\dasharrow X''_U$ is a small birational map.
Since $\psi_*((\hat\Delta '')^{\leq 0}+\epsilon {\rm Ex}(\hat \nu))=0$, then $\hat g({\rm Supp}((\hat\Delta '')^{\leq 0}+\epsilon {\rm Ex}(\hat \nu)))\subset \hat Z''\setminus \hat U$. Since $\hat U$ is a big open subset of $\hat Z''$, then 
$(\hat\Delta '')^{\leq 0}+\epsilon {\rm Ex}(\hat \nu)$ is degenerate over $\hat Z''$. But $(\hat\Delta '')^{\leq 0}+\epsilon {\rm Ex}(\hat \nu)$ is nef over $\hat Z''$ and so the claim follows from Lemma \ref{l-neg}. 
\end{proof}

Note that $X''\dasharrow \hat X''$ is small and $\alpha$-trivial (i.e. $\alpha''$ and $\hat\alpha''$ agree up to pullbacks on a common resolution). Replacing $X''\to Z''$ by $\hat X''\to \hat Z''$, we may assume $Z''$ is strongly $\Q$-factorial from now on (we no longer assume that $K_{Z''}+\Delta _{ Z''}+\ggamma_{ Z''}+a\alpha_{Z''}$ is K\"ahler over $Z$).

 By Theorem \ref{t-MMPscaling}$_{n-1}$, there is a $(K_{Z''}+\Delta _{ Z''}+\ggamma_{ Z''}+a\alpha_{Z''})$-MMP
 \[Z''\dasharrow Z_1 \dasharrow \ldots \dasharrow Z_n\]
 such that $K_{Z_n}+\Delta _{Z_n}+\ggamma_{Z_n}+a\alpha _{Z_n}$ is nef, and hence semiample (globally) by Theorem \ref{t-bpf}$_{n-1}$. By Lemma \ref{l-NQCpullback}, $\alpha _{Z''}$ is NQC, and by Lemma \ref{lem: NQC can run alpha-trivial MMP}, we may assume that this MMP is $\alpha _{Z''}$-trivial. 
 \begin{claim} We can lift the above MMP to $X''$ so that we get a sequence 
 \[X''\dasharrow X''_1 \dasharrow \ldots \dasharrow X''_n\]
 where $X''_i\dasharrow X''_{i+1}$ is a sequence of $K_{X''}+\Delta ''+\bbeta _{X''}+a\alpha ''$ flips and divisorial contractions, and $g_i:X''_i\to Z_i$ is a morphism such that \[ K_{X''_i}+\Delta _{X''_i}+\bbeta _{X''_i}+a\alpha _{X''_i}\equiv g_i ^*(K_{Z_i}+\Delta _{Z_i}+\ggamma_{Z_i}+a\alpha _{Z_i}).\]\end{claim}
 \begin{proof} We proceed by induction.
 Suppose that $X''\dasharrow X''_1 \dasharrow \ldots \dasharrow X''_{i-1}$ has been constructed with the above properties. Let $Z_{i-1}\to W_{i-1}$ be the corresponding flipping or divisorial contraction, then $Z_i$ is the relative K\"ahler model for $K_{Z_{i-1}}+\Delta _{Z_{i-1}}+\ggamma_{Z_{i-1}}+a\alpha _{Z_{i-1}}$ over $W_{i-1}$. Since $\alpha _{X''_{i-1}}\equiv _{W_{i-1}}0$, then $K_{X''_{i-1}}+\Delta _{X''_{i-1}}+\bbeta _{X''_{i-1}}+a\alpha _{X''_{i-1}}\equiv_{W_{i-1}} F''_{i-1}$ where $F''_{i-1}$ is the pushforward of $F''$, and so by Proposition \ref{p-relMMP} we can run the relative MMP to obtain $X''_{i-1}\dasharrow X''_{i}$ a sequence of flips and divisorial contractions over $W_{i-1}$ such that $K_{X''_{i}}+\Delta _{X''_{i}}+\bbeta _{X''_{i}}+a\alpha _{X''_{i}}$ is semiample over $W_{i-1}$. Since $Z_i$ is the corresponding K\"ahler model, then $g_i:X_i''\to Z_i$ is a morphism and $K_{X''_{i}}+\Delta _{X''_{i}}+\bbeta _{X''_{i}}+a\alpha _{X''_{i}}\equiv g_i ^*(K_{Z_i}+\Delta _{Z_i}+\ggamma_{Z_i}+a\alpha _{Z_i})$.
 \end{proof}
 It follows that \[K_{X''_n}+\Delta _{X''_n}+\bbeta _{X''_n}+a\alpha _{X''_n}\equiv (a+1)\alpha _{X''_n} +F''_{n} \] is semiample where $F''_n$ is the pushforward of $F'$.
 Since $F'$ is the negative part of the Boucksom-Zariski decomposition of $(a+1)\alpha ' +F'\equiv (a+1)\nu ^*\alpha + F'$ (see \cite[Lemma A.5]{DHY23}), it follows that 
 $F''_n$ is contained in the negative part of the Boucksom-Zariski decomposition of $(a+1)\alpha _{X''_n} +F''_n$ and therefore $F''_n=0$ (see \cite[Lemma A.11]{DHY23}). Since $F''_n=g_n^*F_{Z_n}$, then $F_{Z_n}=0$. 
 Thus $(a+1)\alpha _{Z_n}\equiv K_{Z_n}+\Delta _{Z_n}+\ggamma _{Z_n}+a\alpha _{Z_n}$ is semiample, and hence  $\alpha_{X''_n}=g_n^*\alpha _{Z_n}$ is semiample.
 
Finally, since $X\dasharrow X''\dasharrow X''_n$ is $\alpha$-trivial, then $\alpha$ is semiample. Since $X''_n\to Z_n$ is projective and the contraction defined by $[\alpha_{Z_n}]$ is strongly Moishezon by Theorem \ref{t-bpf}$_{n-1}$, the contraction defined by $[\alpha_{X''_n}]$ (which is the composition of these two contractions) is also strongly Moishezon. By taking a common resolution of $X$ and $X''_n$, we know the contraction defined by $[\alpha]$ is strongly Moishezon
\end{proof}

\begin{lemma}\label{lem: contr+big get R^if=0}
Let $f: X\to Y$ be a strongly Moishezon contraction, $(X,B+\bbeta)$ be a compact K\"ahler gklt pair such that $B+\bbeta_X$ is modified big, and 
\[K_X+B+\bbeta_X\equiv^{alg}_Y0.\]
Then we have $R^if_*\OO_X=0,~i>0$. In particular, $f$ is locally projective and $Y$ has rational singularities.
\end{lemma}

\begin{proof}
Let $\nu: X'\to X$ be a projective log resolution such that $h:=f\circ\nu: X'\to Y$ is projective, and write
\[
\nu^*(K_X+B+\bbeta_X)=K_{X'}+B'+\bbeta_{X'}.
\] By Lemma \ref{lem: g-kltperturb}, we may assume that $\bbeta _{X'}$ is K\"ahler. Then $-\lfloor B'\rfloor-(K_{X'}+\{B'\})$ is $h$-ample by Lemma \ref{lem: L rel num eq to kahler => f-ample}, and by \cite[Theorem 3.7]{Nak87} we have
\[
R^i\nu_*\OO_{X'}(-\lfloor B'\rfloor)=0\qquad {\rm and} \qquad R^ih_*\OO_{X'}(-\lfloor B'\rfloor)=0,~i>0.
\]
Notice that $-\lfloor B'\rfloor\ge 0$ is exceptional over $X$, therefore by the Leray spectral sequence we get 
\[
R^if_*\OO_X=R^if_*(\nu_*\OO_{X'}(-\lfloor B'\rfloor))=R^ih_*\OO_{X'}(-\lfloor B'\rfloor)=0, ~i>0.
\]
Now $\OO_Y\to Rf_*\OO_X$ is a quasi-isomorphism, and $X$ has rational singularities by Lemma \ref{L-gklt-rational}, therefore $Y$ has rational singularities as well.

For any Stein open subset $U$ of $Y$, we have $H^i(X_U,\OO_{X_U})=H^i(U,\OO_U)=0,~i>0$ by the Leray spectral sequence. Therefore
\[
0\to\R\to \OO_X\to \mathcal{H}_X\to 0,\qquad 0\to\ZZ\to \OO_X\to \OO^*_X\to 1
\]
give us the isomorphisms $H^{1,1}_{\rm BC}(X_U)\to H^2(X_U,\R)$ and $\Pic(X_U)\to H^2(X_U,\ZZ)$. Since $X_U$ is of finite CW complex type ($U$ is relatively compact in $Y$ and $f$ is proper over $Y$), we have $H^2(X_U,\R)=H^2(X_U,\Z)\otimes_\Z\R$. Hence and we get an ample $\R$-line bundle over $U$ through $\Pic(X_U)_\R\simeq H^{1,1}_{\rm BC}(X_U)$ by any K\"ahler form.
\end{proof}
 
Theorem \ref{t-bpf} implies that we can run the  MMP for K\"ahler gklt pairs.

\begin{proposition}\label{prop: MMP can be run}
Assuming Theorem \ref{t-bpf}$_n$. Let $(X,B+\bbeta)$ be a compact gklt pair and $X$ is an $n$-dimensional strongly $\Q$-factorial K\"ahler variety. Let $R=\R^+[\Sigma ]\subset\overline{\rm NA}(X)$ be a $(K_X+B+\bbeta_X)$-negative extremal ray. Then the contraction $f: X\to Y$ associated to $R$ exists and 
\begin{enumerate}
\item $f$ is locally projective and $Y$ is K\"ahler with rational singularities.
\item If $f$ is a divisorial contraction, then $f$ is projective and $Y$ is strongly $\Q$-factorial. There is a canonical pushforward map: 
\[
f_*: H^{1,1}_{\rm BC}(X)\to H^{1,1}_{\rm BC}(Y) \text{ such that } f_*\circ f^*={\rm Id}
\]
\item If $f$ is a flipping contraction, then the flip $f^+:X^+\to Y$ exists, $X^+$ is strongly $\Q$-factorial, and there is a canonical injection
\[
\varphi_*: H^{1,1}_{\rm BC}(X)\to H^{1,1}_{\rm BC}(X^+) \text{ induced by } \varphi: X\dasharrow X^+
\]
\end{enumerate}
\end{proposition}

\begin{proof}
(1) By the cone theorem (see Theorem \ref{t-cone}) we can find a K\"ahler form $\omega$ on $X$ such that $\alpha:=[K_X+B+\bbeta_X+\omega]$ is the nef supporting class of $R$. By Theorem \ref{t-bpf}$_n$, $f$ is the strongly Moishezon contraction defined by $\alpha$ and $Y$ is K\"ahler. Then the rest follows from Lemma \ref{lem: contr+big get R^if=0}. Let $V=R^\perp \subset H^{1,1}_{\rm BC}(X)$ be the codimension 1 subspace perpendicular to the ray $R$, then for every $f$-exceptional curve $C$ and any $\alpha\in V$
we have $\alpha \cdot C=0$.
By Lemma \ref{l-rccfibres}, $\alpha =f^*\alpha _Y$ for some $\alpha _Y\in H^{1,1}_{\rm BC}(Y)$. It follows that $\rho _{\rm BC}(X/Y)=1$.

(2) There is a unique irreducible $f$-exceptional divisor $E$ and $E\cdot R<0$ (see Remark \ref{r-div}). Then one easily sees $-E$ is $f$-ample by Lemma \ref{lem: Moishezon to strong Moishezon} so $f$ is projective. Let $\mathcal{L}$ be any rank one reflexive sheaf on $Y$, then $\mathcal{L}_X:=(f^*\mathcal{L})^{\vee\vee}$ is a rank one reflexive sheaf on $X$. Since $X$ is strongly $\Q$-factorial, there exists integers $m\in\N^*,~l\in\Z$ such that 
\begin{itemize}
\item $\mathcal{L}_X^{\otimes m}$ is a line bundle,
\item $\OO_X(lE)$ is a line bundle, and 
\item $c_1(\mathcal{L}_X^{\otimes m}(lE))\cdot R=0$.
\end{itemize}
Since $X$ and $Y$ have rational singularities, then $R^if_*\OO_X=0$ for $i>0$. By \cite[Proposition 12.1.4]{KM92}, $\mathcal{L}_X^{\otimes mk}(lkE)$ descends to a line bundle $\mathcal{L}_Y$ on $Y$ for some $k\in\N^*$ (this can be checked locally over $Y$). It follows that $\mathcal{L}^{[mk]}=\mathcal{L}_Y$ since both sheaves are reflexive and the equality holds on a big open set on $Y$. Thus $\mathcal{L}$ is a $\Q$-line bundle. By Lemma \ref{l-rccfibres} we can regard $H^{1,1}_{\rm BC}(Y)$ as a subspace of $H^{1,1}_{\rm BC}(X)$ via the injection $f^*$, and we can define
\[
f_*: H^{1,1}_{\rm BC}(X)\to H^{1,1}_{\rm BC}(Y)~,~\gamma\mapsto \gamma _Y \qquad {\rm where}\qquad \gamma-\frac{\gamma\cdot \Sigma }{E\cdot \Sigma } [E]=f^*\gamma _Y
\] (the last equality follows from Lemma \ref{l-rccfibres}).
In particular, $H^{1,1}_{\rm BC}(X)\simeq H^{1,1}_{\rm BC}(Y)\oplus\langle E\rangle$.

(3) The flip $f^+:X^+\to Y$ exists by Theorem  \ref{t-flip}. Note that $f^+$ is a small birational morphism.
Let $\mathcal{L}_{X}$ be a line bundle on $X$, then we claim that the reflexive sheaf $\mathcal{L}_{X^+}:=\varphi_*\mathcal{L}_{X}$ is a $\Q$-line bundle on $X^+$, where $\varphi:X\dasharrow X^+$ is the induced map and we can naturally pushforward a reflexive sheaf. There are two cases to consider. If $c_1(\mathcal{L}_{X})\cdot R=0$, then by \cite[Proposition 12.1.4]{KM92} again, possibly after passing to some multiple, we may assume that $\mathcal {L}_{X}=f^*\mathcal {L}_{Y}$ for some line bundle $\mathcal {L}_{Y}$ on $Y$. Then $\mathcal {L}_{X^+}=(f^+)^*\mathcal {L}_{Y}$ is also a line bundle.
Suppose now that $c_1(\mathcal{L}_{X})\cdot R\ne 0$. Possibly replacing $\mathcal{L}_{X}$ by its dual, we may assume that $(\mathcal{L}_{X})^{\vee}$ is $f$-ample. Then the flip is given by \[
f^+: X^+=\Proj_Y\oplus_{m\ge 0}~f_*(\mathcal{L}_X^{\otimes m})\to Y 
\]
(this can be checked locally over $Y$ using \cite{Fuj22}). Replacing $\mathcal{L}_X$ by a multiple we may assume $f_*\mathcal{L}_X$ generates the $\OO_Y$-algebra $\oplus_{m\ge 0}~f_*(\mathcal{L}_X^{\otimes m})$, then $\mathcal{L}_{X^+}:=\OO_{X^+}(1)$ is an $f^+$-(very) ample line bundle, which agrees with $\varphi _*\mathcal{L}_X$ on a big open set and hence $\mathcal{L}_{X+}=\varphi _* \mathcal{L}_X$. 

Suppose now that $\mathcal{L}_{X^+}$ is a divisorial sheaf on $X^+$. Since $\varphi$ is a small birational map, $\mathcal{L}_{X}:=(\varphi ^{-1}_* \mathcal{L}_{X^+})^{\vee\vee}$ is a divisorial sheaf on $X$ and hence a $\Q$-line bundle as $X$ is strongly $\Q$-factorial. By what we have seen above, $\varphi _*\mathcal{L}_{X}$ is a $\Q$-line bundle. Thus $\varphi _*\mathcal{L}_{X}^{[m]}$ is a line bundle for some $m>0$. Since $\varphi _*\mathcal{L}_{X}^{[m]}$ is isomorphic to the reflexive sheaf $\mathcal{L}_{X^+}^{[m]}$ on a big open set, then the two sheaves are isomorphic and hence $\mathcal{L}_{X^+}$ is a $\Q$-line bundle. 
Thus $X^+$ is strongly $\Q$-factorial.

We now focus on $H^{1,1}_{\rm BC}$. Let $\alpha_X:=[K_X+B+\bbeta_X]\in H^{1,1}_{\rm BC}(X)$ and $\alpha_{X^+}:=[K_{X^+}+B^++\bbeta_{X^+}]\in H^{1,1}_{\rm BC}(X^+)$, notice that $\alpha_{X} \cdot \Sigma <0$.
Let $\gamma\in H^{1,1}_{\rm BC}(X)$, then $\gamma-\frac{\gamma\cdot \Sigma }{\alpha_X \cdot \Sigma}\cdot\alpha_X=f^*(\gamma_Y)$ for some $\gamma_Y\in H^{1,1}_{\rm BC}(Y)$ by Lemma \ref{l-rccfibres}. Now we define
\[
\varphi_*: H^{1,1}_{\rm BC}(X)\to H^{1,1}_{\rm BC}(X^+)~,~ \gamma\mapsto f^{+,*}(\gamma_Y)+\frac{\gamma\cdot \Sigma }{\alpha_X \cdot \Sigma }\alpha_{X^+}.
\]
It is easy to check that $\varphi_*$ is injective by noticing that $\alpha_{X^+}$ is $f^+$-K\"ahler. Let $U$ be the big open set that $\varphi$ is an isomorphism, and $\theta$ be a locally exact $(1,1)$-current representing $\gamma\in H^{1,1}_{\rm BC}(X)$, then we see $\theta|_U$ extends uniquely to a locally exactly $(1,1)$-current $\theta^+$ on $X^+$, and the class coincides with $\varphi_*\gamma$. Therefore the map $\varphi_*$ is canonical.  

\end{proof}

\section{Existence of log terminal models}\label{s-MMPS}

The goal of this section is to prove Theorem \ref{t-MMPscaling}. More precisely, we will show the following.
\begin{theorem} Assuming Theorems \ref{t-bpf}$_{n}$ and \ref{t-MMPscaling}$_{n-1}$, then Theorem \ref{t-MMPscaling}$_{n}$ holds: 

Let $(X,B+\bbeta )$ be a compact, $n$-dimensional, strongly $\Q$-factorial, K\"ahler, gklt pair such that $B+\bbeta _X$ is big and $\omega$ a modified K\"ahler form such that $K_X+B+\bbeta_X +\omega$ is nef. Then we can run the $K_X+B+\bbeta _X$  minimal model program with scaling which terminates with either a good log terminal model or a Mori fiber space.\end{theorem}

Note that Theorem \ref{t-MMPscaling}$_{n-1}$ immediately implies Theorem \ref{t-KahlerBCHM}$_{n-1}$ by taking a small $\Q$-factorial dlt modification see the beginning of \S \ref{s-application}. 
We will first show Theorem \ref{t-KahlerBCHM}$_n$ (the existence of log terminal models). There are two cases to consider depending on whether $K_X+B+\bbeta _X$ is big, or $K_X+B+\bbeta _X$ is pseudo-effective and $B+\bbeta _X$ is big.
\subsection{The big case} 
Let $X$ be a compact normal K\"ahler variety.
Recall that if $\alpha\in H^{1,1}_{\rm BC}(X)$ is pseudo-effective, then one defines   $\alpha=N(\alpha)+P(\alpha)$  the Boucksom-Zariski decomposition (BZ-decomposition for short) as in \cite[Definition 3.7]{Bou04} and \cite[Appendix A]{DHY23}.
Here \[N(\alpha)=\sum _{Q\subset X}\nu (\alpha ,Q)Q\geq 0\] is the negative part, where $Q$ are prime divisors and $\nu(\alpha ,Q)$ are the corresponding Lelong numbers.
The positive part $P(\alpha):=\alpha -N(\alpha)$ is modified nef (hence nef in codimension 1), $\alpha=P(\alpha)$ iff $\alpha$  is modified nef, and $P(\alpha)$ is big iff so is $\alpha$ (\cite[Proposition 3.8]{Bou04}).
The above decomposition is continuous on the big cone and if $\beta$ is modified nef, then $N(\alpha+\beta)\leq N(\alpha)$.
In particular, if $\omega $ is K\"ahler, then $\lim _{\epsilon \to 0}N(\alpha +\epsilon\omega)=N(\alpha)$.
We recall the following properties.
\begin{lemma}\label{l-BZ-prop} Let $X$ be a compact normal K\"ahler variety, and $\alpha,\beta \in H^{1,1}_{\rm BC}(X)$  pseudo-effective classes. Then
\begin{enumerate}
\item $N(\alpha)$ is an exceptional divisor (in the sense of \cite[Section 3]{Bou04}), and $N(N(\alpha))=N(\alpha)$,
\item $N(t\alpha)=tN(\alpha)$ and $P(t\alpha)=tP(\alpha)$ for all $t\geq 0$,
\item $N(\alpha+\beta)\leq N(\alpha)+N(\beta)$, and
    \item $N(P(\alpha)+C)=C$ for all $0\leq C$ such that ${\rm Supp}(C)\subset {\rm Supp}N(\alpha)$.
\end{enumerate}
\end{lemma}
\begin{proof} This lemma is proved in \cite{DH26}. We include the proof for the convenience of the reader.

(1) By \cite[Theorem 3.12]{Bou04}  $N(\alpha)$ is exceptional and so by \cite[Proposition 3.11]{Bou04},  $N(N(\alpha))=N(\alpha)$.

(2-3)  See \cite[Proposition 3.9]{Bou04}.

(4) Suppose first that $0\leq C\leq N(\alpha)$. Since  $N(\alpha)$ is exceptional, so is $N(\alpha) -C$. By (1) we have $N(N(\alpha)-C)=N(\alpha)-C$ and $N(C)=C$.
Since $N(P(\alpha))=0$, by property (3) we have\[N(P(\alpha)+C)+N(\alpha) -C\leq N(P(\alpha))+N(C)+N(\alpha) -C=N(\alpha).\]
Properties (3) and (1) imply that \[N(P(\alpha)+N(\alpha))\leq N(P(\alpha)+C)+N(N(\alpha)-C)=N(P(\alpha)+C)+(N(\alpha)-C).\] Putting this together we then have a sequence of inequalities \[N(\alpha)=N(P(\alpha)+N(\alpha))\leq N(P(\alpha)+C)+(N(\alpha)-C) \leq N(\alpha).\]
It follows that the above inequalities are equalities and hence (4) holds if $0\leq C\leq N(\alpha)$. 

Suppose now that $C=tN(\alpha)$.
By what we have shown above, we may assume that $t>1$. By (2) we have $tN(\alpha)=N(t\alpha)$. (1-3) imply that
\[N(t\alpha)\leq N((t-1)P(\alpha))+N(P(\alpha)+tN(\alpha))\leq N(P(\alpha))+N(N(t\alpha))= N(t\alpha).\] Therefore all inequalities are equalities and in particular $N(P(\alpha)+tN(\alpha))=N(t\alpha)$.

Fix $t>0$ such that $C\leq tN(\alpha)$, then $C\leq N(P(\alpha)+tN(\alpha))$ and $P(P(\alpha)+tN(\alpha))=P(\alpha)$. By what we have seen above, it follows that 
\[N(P(\alpha)+C)=N(P(P(\alpha)+tN(\alpha))+C)=C.\]
\end{proof}
\begin{lemma}\label{l-N+} Let $X$ be a compact normal K\"ahler variety, $\alpha\in H^{1,1}_{\rm BC}(X)$ be a big class, and  $N_+(\alpha ):=\cap_{\epsilon >0} {\rm Supp}N(\alpha -\epsilon\omega)$. Then $N_+(\alpha )\geq {\rm Supp}(N(\alpha))$ is the reduced divisor supported on the divisorial components of $E_{\rm nK}(\alpha)$ and there exists $\epsilon _0>0$ such that $N_+(\alpha )={\rm Supp}N(\alpha -\epsilon\omega)$ for all $0<\epsilon<\epsilon _0$.
\end{lemma}
\begin{proof} Since $N(\alpha -\epsilon\omega)\geq N(\alpha)$ for any $\epsilon>0$, then $N_+(\alpha )\geq {\rm Supp}(N(\alpha))$.   By definition (see \cite[Definition 3.1]{DH23}) $E_{\rm nK}(\alpha)=\cap_{T\in \alpha} E_+(T)$ where $T$ runs over all K\"ahler currents contained in the class $\alpha$.
Thus $E_{\rm nK}(\alpha)=\cap_{T\in \alpha} E_+(T)$ where $[T]\in \alpha $ runs over all currents $T\geq \epsilon\omega$ for some $\epsilon>0$, and $E_+(T)\subset X$ is the (analytic) subset of $X$ for which the Lelong numbers of $T$ are strictly positive.
Therefore, a divisor $P$ is contained in $E_{\rm nK}(\alpha)$ iff there exists an $\epsilon _0>0$ such that $P$ is contained in $E_+(T)$ for every $T\in \alpha-\epsilon \omega$ for all $0<\epsilon\leq \epsilon _0$ iff $P$ is contained in ${\rm Supp}E_{\rm nK}(\alpha -\epsilon \omega)$  for all $0<\epsilon<\epsilon _0$ iff $P$ is contained in $N(\alpha -\epsilon \omega)$  for all $0<\epsilon<\epsilon _0$ iff
$P$ is contained in $N_+(\alpha)$. In particular $N_+(\alpha )={\rm Supp}N(\alpha -\epsilon\omega)$ for all $0<\epsilon<\epsilon _0$. 

\end{proof}
\begin{lemma}\label{l-N+C} Let $X$ be a compact normal K\"ahler variety, $\alpha\in H^{1,1}_{\rm BC}(X)$ a big class. Suppose that $C\geq 0$ is an $\R$-divisor supported on $N_+(\alpha)$, then \begin{enumerate} \item $N_+(\alpha+C)=N_+(\alpha)$, 
\item $N(\alpha+C)=N(\alpha)+C$.\end{enumerate}
\end{lemma}
\begin{proof} (1) By Lemma \ref{l-N+}, $N_+(\alpha )={\rm Supp} (N(\alpha -\epsilon \omega))$ for any $0<\epsilon \ll 1$ and in particular ${\rm Supp}(C)\subset {\rm Supp}(N(\alpha -\epsilon \omega))$. By Lemma \ref{l-BZ-prop}, $N(\alpha -\epsilon \omega+C)=N(\alpha -\epsilon \omega)+C$ and the claim follows.

(2) By Lemma \ref{l-N+} there exists $\epsilon _0>0$ such that ${\rm Supp}(C)\subset {\rm Supp}(N(\alpha -\epsilon \omega))$ for any $0<\epsilon < \epsilon _0$. By Lemma \ref{l-BZ-prop}, $N(\alpha -\epsilon \omega+C)=N(\alpha -\epsilon \omega)+C$. By continuity \[N(\alpha +C)=\lim _{\epsilon \to 0^+}N(\alpha -\epsilon \omega+C)=\lim _{\epsilon \to 0^+}N(\alpha -\epsilon \omega)+C=N(\alpha)+C.\]
\end{proof}

\begin{lemma}\label{l-N+birational} Let $f:X'\to X$ be a birational morphism of compact normal $\Q$-factorial K\"ahler varieties, $F\geq 0$ an $f$-exceptional divisor such that $-F$ is $f$-ample, $E\geq 0$ an $f$-exceptional divisor, and $\alpha\in H^{1,1}_{\rm BC}(X)$. Then
\begin{enumerate}
\item $N(f^*\alpha+E)=f^*(N(\alpha))+E$,
\item if $\alpha$ is big, then $N_+(f^*\alpha+E)=f^{-1}_*N_+(\alpha)+{\rm Ex}(f)$.
\end{enumerate}
\end{lemma}
\begin{proof}
(1) See \cite[Lemma A.5]{DHY23}. 

(2) Since $\alpha$ is big, then so is $f^*\alpha$. By Lemma \ref{l-N+C}, we may assume that $E=0$. Rescaling $F$ we may assume that $\omega ':=f^*\omega -F$ is K\"ahler. By Lemma \ref{l-N+}, \[N_+(f^*\alpha)={\rm Supp}N(f^*\alpha-\epsilon \omega ')={\rm Supp}N(f^*(\alpha-\epsilon \omega)+ \epsilon F )\] for any $0<\epsilon \ll 1$. By (1), \[{\rm Supp}N(f^*(\alpha-\epsilon \omega)+ \epsilon F )={\rm Supp}(f^*(N(\alpha-\epsilon \omega))+\epsilon F) =f^{-1}_*N_+(\alpha)+ {\rm Ex}(f).\]
\end{proof}
\begin{lemma}\label{l-contdiv} If $(X,B+\bbeta)$ is a compact gdlt pair  and $\phi :X\dasharrow X'$ a log terminal model (resp. $K_X+B+\bbeta _X$ is moreover big and $\nu:X'\to X^c$ is the log canonical model). Then ${\rm Supp}N (K_X+B+\bbeta _X)$ (resp.  $N_+ (K_X+B+\bbeta _X)$) is the set of divisors contracted by $\phi$ (resp. contracted by $\psi=\nu\circ \phi$).
\end{lemma}
\begin{proof} Let $p:W\to X$ and $q:W\to X'$ be a common resolution. By assumption $K_{X'}+B'+\bbeta _{X'}$ is nef (respectively $K_{X'}+B'+\bbeta _{X'}\equiv \nu ^*(K_{X^c}+B^c+\bbeta _{X^c})$ where $K_{X^c}+B^c+\bbeta _{X^c}$ is K\"ahler) and if we write $p^*(K_X+B+\bbeta _X)=q^*(K_{X'}+B'+\bbeta _{X'})+E$ then $E\geq 0$ and its support contains all $\phi$-exceptional divisors. 

Assume that $\phi$ is a log terminal model, then by Lemma \ref{l-N+birational}, \[N(p^*(K_X+B+\bbeta _X))=N(q^*(K_{X'}+B'+\bbeta _{X'})+E)=E\] as $K_{X'}+B'+\bbeta _{X'}$ is nef and $E\geq 0$ is exceptional. 
It follows that \[N (K_X+B+\bbeta _X)=p_*(N(p^*(K_X+B+\bbeta _X)))=p_*(E).\]
This proves the first claim, since the set of divisors contracted by $\phi$ is the support of $p_*E$.

Assume now that $K_X+B+\bbeta _X$ is moreover big. 
By Lemma \ref{l-N+birational} \[N_+(p^*(K_X+B+\bbeta _X))=N_+((\nu\circ q)^*(K_{X'}+B'+\bbeta _{X'})+E)=\]\[(\nu\circ q)^{-1}_*N_+(K_{X^c}+B^c+\bbeta _{X^c})+{\rm Ex}(\nu\circ q)={\rm Ex}(\nu\circ q).\] 
This implies the second claim since 
\[N_+(K_X+B+\bbeta _X)=p_*{\rm Ex}(\nu\circ q)={\rm Ex}(\psi).\]

\end{proof}

\begin{definition}\label{d-M}
We let $\mathcal M_d$ be the set of all $d$-dimensional compact generalized pairs $(X,B+\bbeta)$ such that
\begin{enumerate}
\item $(X,B+\bbeta )$ is gdlt, $\bbeta=\overline{\bbeta_X}$, and $[\bbeta _X]\in H^{1,1}_{\rm BC}(X)$ is K\"ahler, 
\item $K_X+B+\bbeta _X\equiv \omega +D$ where $D\geq 0$, $\omega \in H^{1,1}_{\rm BC}(X)$ is K\"ahler, and  ${\rm Supp}(D)=N_+ (K_X+B+\bbeta _X)$, 
\item $X$ is smooth and $B+D$ has simple normal crossings support, and
\item $(X,B+\bbeta )$ does not have a log terminal model.
\end{enumerate}
We define $\theta (X,B+\bbeta , \omega +D )$ to be the number of components of $D$ not contained in $\lfloor B\rfloor$.
\end{definition} We aim to show that $\mathcal M_d=\emptyset$. 

\begin{lemma}\label{l-t=0}Assume Theorems \ref{t-bpf}$_n$ and  \ref{t-KahlerBCHM}$_{n-1}$ hold. Then $\theta (X,B+\bbeta,\omega +D)>0$ for all $(X,B+\bbeta)\in \mathcal M_d$.
    \end{lemma} 
  \begin{proof} Suppose by contradiction that $\theta (X,B+\bbeta,\omega +D)=0$ for some $(X,B+\bbeta)\in \mathcal M_d$. We will proceed by running a minimal model program with scaling of $\omega$. Recall that $(X,B+\bbeta)$ is gdlt and $K_X+B+\bbeta _X\equiv D+\omega$ where $D\geq 0$, ${\rm Supp}(D)\subset \lfloor B\rfloor$, and $\omega$ is K\"ahler. Clearly $K_X+B+\bbeta _X+t\omega$ is K\"ahler for $t\gg 0$. We will denote by \[X=X_0\dasharrow X_1\dasharrow X_2\dasharrow \ldots \] the corresponding sequence of flips and
divisorial contractions so that $X_i$ is strongly $\Q$-factorial, $K_{X_i}+B_i+\bbeta_{X_i}+t\omega _i$ is nef for $t_{i-1}\geq t\geq t_i$, where $\phi_i:X\dasharrow X_i$ and $K_{X_i}+B_i+\bbeta_{X_i}+t\omega _i=(\phi_i)_*(K_{X}+B+\bbeta_{X}+t\omega )$. We let
$R_i$ be the corresponding $(K_{X_i}+B_i+\bbeta_{X_i}+t_i\omega _i)$-trivial, $(K_{X_i}+B_i+\bbeta_{X_i})$-negative extremal ray.
We must show that this minimal model program terminates and in particular we may assume that $t_i>0$.

Since $R_i\cdot \omega _i>0$, it follows that $R_i\cdot D_i<0$ where $D_i=(\phi_i)_*D$, so that the flipping locus is contained in the support of $D_i$, and hence in the support of $\lfloor B_i\rfloor$. We will show, via special termination, that after finitely many steps of this minimal model program, each flip is disjoint from the support of $D_i$. Since this is impossible, the minimal model program will terminate after finitely many steps.

Turning to the details of special termination, we first note that after finitely many steps, we may assume that all $\psi_i:X_i\dasharrow X_{i+1}$ are flips and the flipping loci do not contain any stratum of $\lfloor B_i\rfloor$. Let $S$ be a component of $\lfloor B\rfloor $ which is not contracted by $\phi _i:X\dasharrow X_i$ for $i\gg 0$, then we must show that $\psi _i$ is an isomorphism on a neighborhood of $S_i:=(\phi_i)_*S$ for all $i\gg 0$.

Note that \[K_X+B+\bbeta_X\equiv K_X+B'+\bbeta'_X\] where 
$B'=B-\epsilon(\lfloor B\rfloor -S)$, $\bbeta '=\bbeta +\epsilon(\overline{\lfloor B\rfloor -S})$ for $0<\epsilon \ll 1$ so that $\bbeta '_X$ is K\"ahler.
Replacing $(X,B+\bbeta)$ by $(X, B'+\bbeta')$, we may assume that 
$(X,B+\bbeta)$ is gplt and $\lfloor B\rfloor =S$ (the condition $\theta (X,B+\bbeta,\omega +D)=0$ no longer holds and will not be used in the sequel).

We let $\nu :X'\to X$ be a log resolution such that if $K_{X'}+B_{X'}+\bbeta _{X'}=\nu ^*(K_X+B+\bbeta_X)$, $S'=\nu ^{-1}_*S$, and 
\[(K_{X'}+(B_{X'})^{\geq 0}+\bbeta _{X'} )|_{S'}=K_{S'}+B_{S'}+\ggamma_{S'},\]
then $(S', B_{S'}+\ggamma_{S'})$ is generalized terminal (in fact $(S',B _{S'})$ is terminal and $\ggamma _{S'}=\bbeta _{X'}|_{S'}$ is nef and descends to $S'$). Let $F\geq 0$ be a $\nu$-exceptional divisor such that $-F$ is relatively ample. We then have that, for $0<\epsilon \ll 1$,  $[\bbeta _{X'}-\epsilon F]\in H^{1,1}_{\rm BC}(X')$ is K\"ahler and $(X',B':=(B_{X'})^{\geq 0}+\epsilon F)$ is plt, $(S',B _{S'}+\epsilon F|_{S'})$ is terminal, and $\beta '_{X'}:=\bbeta _{X'}-\epsilon F$  and $\gamma'_{S'}:=\ggamma _{S'}-\epsilon F|_{S'}$ are K\"ahler. 
Let $\omega '=\nu ^*\omega$, $\bbeta '=\overline{\beta '_{X'}}$, and  $\ggamma '=\overline{\gamma '_{S'}}$, then  $K_{S'}+B'_{S'}+\ggamma' :=(K_{X'}+B'+\bbeta ')|_{S'}$ is gklt and $\ggamma' _{S'}$ is K\"ahler.
Since \[K_{X'}+B'+\bbeta '_{X'}\equiv \nu ^*(K_X+B+\bbeta _X)+(B_{X'})^{\leq 0}\]
where $(B_{X'})^{\leq 0}$ is effective and $\nu$-exceptional, then the induced rational maps $\phi '_i:X'\dasharrow X_i$ are weak log canonical models for 
$K_{X'}+B'+\bbeta '_{X'}+t\omega '$ for any $t_{i-1}\leq t\leq t_i$.  

By Lemma \ref{l-special}, there are divisors $0\leq C'_i\leq B'_{S'_i}$ such that $\varphi '_i:=\phi '_i|_{S'}:S'\dasharrow S_i$ is a weak log canonical model for $(S',C'_i+\ggamma '_{S'}+t\omega '|_{S'})$ where $t_{i-1}\leq t\leq t_i$, \[(\varphi'_i)_*(K_{S'}+C'_i+\ggamma '_{S'})=K_{S_i}+C_i+\ggamma '_{S_i}=(K_{X_i}+B_i+\bbeta '_{X_i})|_{S_i},\]
$B_i=(\phi _i)_* B=(\phi ' _i)_* B'$, and $\bbeta '_{X_i}=\bbeta _{X_i}$. By Theorem \ref{t-MMPgeography}, we have an infinite subsequence $\varphi '_{i_k}:S'\dasharrow S_{i_k}$ 
such that the induced maps $\sigma _k:S_{i_1}\to S_{i_k}$ are isomorphisms for all $k>0$. 
We also write $(K_{X_i}+S_i)|_{S_i}=K_{S_i}+\Phi _i$ where $0\leq \Phi _i\leq C_i$ and the coefficients of $\Phi _i$ are of the form $1-\frac 1 r$ and bounded from above so that they belong to a finite set. Passing to a subsequence, we may assume that $\sigma ^* _k(\Phi _{i_k})$ is independent of $k$ and hence $X_{i_1}\dasharrow X_{i_k}$ is an isomorphism in a neighborhood of codimension 1 points of $S_{i_1}$. In particular the Cartier index at codimension 1 points of $S_{i_k}$ is bounded (see \cite[Definition-Lemma 3.4.1]{BCHM10}). 
 Therefore, there exists an integer $r>0$ such that if $Q$ is a Weil divisor on $X_{i_k}$, then the coefficients of $Q|_S$ belong to $\frac 1r {\mathbb N}$. 
 If $P$ is a component of $B'-S'$, then the coefficients of $((\phi'_i)_*P)|_{S_{i_k}}$ belong to the finite set $\{0,\frac 1r,\frac 2r,\ldots ,\frac {r-1}r\}$ and therefore there are only finitely many possibilities for the coefficients of $((\phi'_{i_k})_*(B'-S'))|_{S_{i_k}}$. Passing to a subsequence of $\{i_k\}$, we may assume that $\sigma _k^*(((\phi'_{i_k})_*(B'-S'))|_{S_{i_k}})$ is independent of $k$.
By \cite[Lemma 4.3]{BCHM10}, it follows that $X_{i_1}\dasharrow X_{i_k}$ is an isomorphism in a neighborhood of codimension 1 points of $S_{i_1}$. By Lemma \ref{l-codim1iso}, 
$X_{i_1}\dasharrow X_{i_k}$ is an isomorphism in a neighborhood of $S_{i_1}$, which is impossible. The lemma is proved.

  \end{proof}
  \begin{lemma}\label{l-special} Let $(X,S+B+\bbeta)$ be a compact generalized pair such that $X$ is smooth, $S+B$ has simple normal crossings support, $S$ is irreducible, $\lfloor B|_S\rfloor=0$, $\bbeta=\overline{\bbeta_X}$, and $\bbeta _X$ is nef. Let $\phi:X\dasharrow X'$ be a weak log canonical model for $(X,S+B+\bbeta_X)$ such that $\phi _*S\ne 0$, and denote
  $(X,S+B+\bbeta )|_S=(S,B_S+\ggamma)$ and $(X',S'+B'+\bbeta )|_{S'}=(S',C_{S'} +\ggamma)$ where $K_{X'}+S'+B'+\bbeta _{X'}=\phi _*(K_{X}+S+B+\bbeta _{X})$.
  
  If $(S,B_S)$ is terminal, then there is a divisor $0\leq C_S \leq B|_S$ such that $\psi :S\dasharrow S'$ is a weak log canonical model for $(S,C_S +\ggamma )$ and $\psi_* C_S =C_{S'}$ where $\psi=\phi|_S$.
  \end{lemma}
  \begin{proof} Let $p:W\to X$ and $q:W\to X'$ be a common resolution and write \[K_W+S_W+B_W+\bbeta _W=p^*(K_X+S+B+\bbeta _X),\]\[ K_W+S_W+C_W+\bbeta _W=q^*(K_{X'}+S'+B'+\bbeta _{X'})\]
  where $S_W=p^{-1}_*S$. Since $\phi$ is $(K_X+S+B+\bbeta _X)$-non-positive, then $(X',S'+B'+\bbeta _{X'})$ is gplt, and  $B_W-C_W=E\geq 0$ is $q$-exceptional. We now write \[K_{S_W}+B_{S_W}+\ggamma _{S_W}=(K_W+S_W+B_W+\bbeta _W)|_{S_W},\]
\[K_{S_W}+C_{S_W}+\ggamma _{S_W}=(K_W+S_W+C_W+\bbeta _W)|_{S_W}.\] Then $K_S+B_S+\ggamma _S=(p|_{S_W})_*(K_{S_W}+B_{S_W}+\ggamma _{S_W})$ and  $K_{S'}+C_{S'}+\ggamma _{S'}=(q|_{S_W})_*(K_{S_W}+C_{S_W}+\ggamma _{S_W})$. Since the support of $B_W-C_W$ does not contain $S$, then $B_{S_W}-C_{S_W}=(B_W-C_W)|_{S_W}\geq 0$. 

Let $P$ be a divisor on $S_W$ which is $p|_{S_W}$-exceptional, but not $q|_{S_W}$-exceptional, then ${\rm mult}_P(C_{S_W})\geq 0$ and so \[{\rm mult}_P(B_{S_W})\geq {\rm mult}_P(C_{S_W})\geq 0\] contradicting the fact that $(S,B_{S}+\ggamma)$ is a generalized terminal pair. Thus $\psi :S\dasharrow S'$ is a birational contraction. We also observe that as $B_{S_W}-C_{S_W}\geq 0$, then $\psi _* B_{S}\geq C_{S'}$. We now define $C_S =(\psi ^{-1}_*C_{S'}+{\rm Ex}(\psi))\wedge B_S$. 
We then have \[(p|_{S_W})^*(K_S+C_S+\ggamma _S)=(q|_{S_W})^*(K_{S'}+C_{S'}+\ggamma _{S'})+F\]
where $F=B_{S_W}-C_{S_W}-(p|_{S_W})^*(B_S-C_S)$. Note that if $P$ is a component of ${\rm Ex}(\psi)$, then ${\rm mult }_P(B_S-C_S)=0$ so that ${\rm mult }_P(F)\geq 0$ and if $P$ is contained in the support of $C_{S'}$, then ${\rm mult }_P(C_{S_W}-\psi ^{-1}_*C_S)=0$ so that if $P$ is also not $p|_{S_W}$-exceptional, then ${\rm mult }_P(F)= 0$. But then  $(p|_{S_W})_*F\geq 0$, and by the negativity lemma (see Lemma \ref{l-negbir}) $F\geq 0$ so that $\psi $ is a weak log canonical model for $(S,C_S+\ggamma _S)$.
  \end{proof}
  \begin{lemma}\label{l-codim1iso} Let $(X,S+B+\bbeta)$ be a compact gplt pair and $\phi :X\dasharrow X'$ be a proper small birational map given by a sequence of the $K_X+S+B+\bbeta_X$ minimal model program.
  If $\phi|_S:S\dasharrow S'$ is an isomorphism and $\phi$ is an isomorphism on a neighborhood of every codimension 1 point of $S$, then $\phi$ is an isomorphism on a neighborhood of $S$.
  \end{lemma}
  \begin{proof}
  Let $p:W\to X$ and $q:W\to X'$ be a common resolution and $E:=p^*(K_X+S+B+\bbeta _X)-q^*(K_{X'}+S'+B'+\bbeta _{X'})$ where $E\geq 0$ and $p(E)$ is the set where $\phi$ is not an isomorphism. We write $(X,S+B+\bbeta )|_S=(S,B_S+\ggamma)$ and $(X',S'+B'+\bbeta )|_{S'}=(S',B_{S'}+\ggamma)$. Since $\phi$ is an isomorphism around codimension 1 points of $S$ and $\phi |_S$ is an isomorphism, then $\phi |_S$ gives an equivalence $(S,B_S+\ggamma)\cong (S',B_{S'}+\ggamma)$. 
  Let $S_W=p_*^{-1}S$, it then follows that \[0=(p|_{S_W})^*(K_S+B_S+\ggamma _S)-(q|_{S_W})^*(K_{S'}+B_{S'}+\ggamma _{S'})=E|_{S_W}.\]
Thus $E\cap S_W=\emptyset$, which implies that $p(E)\cap S=\emptyset$ and hence $\phi $ is an isomorphism on a neighborhood of $S$.
  \end{proof}

\begin{proposition}\label{p-Md} Assume Theorems \ref{t-bpf}$_n$ and  \ref{t-KahlerBCHM}$_{n-1}$ hold. Then  $ \mathcal M_n=\emptyset$.
\end{proposition}\begin{proof}Proceeding by induction, we will assume that $(X,B+\bbeta ,D+\omega )\in \mathcal M_n$ and
$\theta (X,B+\bbeta ,D+\omega )$ is minimal. In particular $(X,B+\bbeta )$ does not have a log terminal model.  By Lemma \ref{l-t=0}, we may assume that $\theta (X,B+\bbeta  ,D+\omega )>0$. We follow the approach of \cite{Bir12}. Define
\[\lambda={\rm min}\{t\geq 0|\lfloor (B+tD )^{\leq 1}\rfloor\ne \lfloor B\rfloor\}\] so that we can write \[(B+\lambda D)^{\leq 1}=B+C,\qquad C\geq 0,\qquad {\rm Supp}(C)\subset {\rm Supp}(D),\qquad {\rm and}\] 
\[\lambda D=C+A,\qquad A\geq 0,\qquad {\rm Supp}(A)\subset {\rm Supp}(\lfloor B\rfloor ),\qquad A\wedge C=0.\]
In other words, the components of $A$ are the common components of $\lfloor B\rfloor$ and $D$, and the components of $C$ are the components of $D$ which are not components of $\lfloor B\rfloor$.
The pair $(X,B+C+\bbeta )$ is log smooth, gdlt, and we may write
\[K_X+B+C+\bbeta _X \equiv D+C+\omega .\]
 Note that ${\rm Supp}(C)\subset {\rm Supp}(D)$ where ${\rm Supp}(D)=N_{+} (D+\omega)$ and so
\[N_{+}(D+\omega )=N_{+} (D+C+\omega)={\rm Supp}(D+C),\]  and $N(D+C+\omega)=N(D+\omega)+C$ (see Lemma \ref{l-N+C}). It follows easily  that $(X,B+C+\bbeta )$ satisfies conditions (1-3) of Definition \ref{d-M}. 
Since  \[\theta (X,B+C+\bbeta  , D+C+\omega )<\theta (X,B+\bbeta  ,D+\omega ),\]  $(X,B+C+\bbeta )$ has a log terminal model, say $\phi :X\dasharrow X'$.
Let $X'\to X^c$ be the corresponding log canonical model (which exists by Theorem \ref{t-bpf}$_n$) and $\psi :X\dasharrow X^c$ the induced birational map. By Lemma \ref{l-contdiv}, $\phi$ contracts \[N(K_X+B+C+\bbeta _X)=N(D+C+\omega )=N(D+\omega )+C\] and hence $\phi _*C=0$, i.e. $\phi _*(K_X+B+\bbeta _X)=\phi _*(K_X+B+C+\bbeta _X)$ is nef.
Let $p:W\to X$ and $q:W\to X'$ be a common resolution, then we may write \[p^*(K_X+B+C+\bbeta _X)=q^*(K_{X'}+B'+\bbeta _{X'})+E\] where $K_{X'}+B'+\bbeta _{X'}$ is nef and $E$ is $q$-exceptional. Thus $p_*E=N(K_X+B+C+\bbeta _X)=N(K_X+B+\bbeta _X)+C$ and in particular $p_*E-C$ is effective.
But then \[p^*(K_X+B+\bbeta _X)=q^*(K_{X'}+B'+\bbeta _{X'})+E-p^*C.\] Since $p^*C-E\equiv _X q^*(K_{X'}+B'+\bbeta _{X'})$ is $p$-nef and $p_*(E-p^*C)\geq 0$, then by the negativity lemma (see Lemma \ref{l-negbir}) we have $E-p^*C\geq 0$.
It follows that $\phi$ is a $\Q$-factorial weak log canonical model for $(X,B+\bbeta )$. Let $X^\sharp\to X'$ be the birational morphism of strongly $\Q$-factorial varieties obtained by extracting the divisors $Q\subset X$ such that ${\rm mult}_Q(p_*E-C)=0$ (see Corollary \ref{c-ext}), then $\phi ^\sharp:X\dasharrow X^\sharp$ is a log terminal model for $(X,B+\bbeta )$ (see Lemma \ref{l-wmm}). This is the required contradiction. 
\end{proof}
\begin{proposition}\label{p-bigmm} Assume Theorems \ref{t-bpf}$_n$ and  \ref{t-KahlerBCHM}$_{n-1}$ hold.
Let $(X,B+\bbeta )$ be a compact, K\"ahler, $n$-dimensional gklt pair such that $K_X+B+\bbeta _X$ is big, then $(X,B+\bbeta )$ has a good log terminal model.
\end{proposition}
\begin{proof} Let $(X,B+\bbeta )$ be a $n$-dimensional gklt pair where $K_X+B+\bbeta _X$ is big. By standard reductions, we may assume that $(X,B)$ is log smooth and $\bbeta =\bar \bbeta _X$ where $\bbeta _X$ is K\"ahler (cf. Lemma \ref{lem: g-kltperturb}). 
Pick $\nu :X'\to X$ a log resolution such that $\nu ^*(K_X+B+\bbeta _X)=G'+\omega '$ where $G'\geq 0$, $\nu (G')=E_{\rm nK}(K_X+B+\bbeta _X)+\nu ({\rm Ex}(\nu))$, and $\omega '$ is K\"ahler. Since $-G'\equiv _X\omega '$, then ${\rm Supp}G'=\nu ^{-1}(\nu (G'))$ and ${\rm Ex}(\nu)\subset {\rm Supp}G'$. Note that by \cite[Lemma 3.8]{DH23}, \[E_{\rm nK}(\nu ^*(K_X+B+\bbeta _X))=\nu ^{-1}(E_{\rm nK}(K_X+B+\bbeta _X))\cup {\rm Ex}(\nu)=\nu ^{-1}(\nu (G'))\cup {\rm Ex}(\nu).\] Since $\nu ^{-1}(\nu (G'))\cup {\rm Ex}(\nu)={\rm Supp}(G')$, it follows that $N_+(G'+\omega ')={\rm Supp}(G')$. 
We may assume that $G'+{\rm Ex}(\nu)$ has simple normal crossings support.
We write \[K_{X'}+B_{X'}+\bbeta _{X'}=\nu ^*(K_X+B+\bbeta _X).\]
If $B'=B_{X'}^{\geq 0}+\epsilon B_{X'}^{\leq 0} +\epsilon (G'+{\rm Ex}(\nu))$ where $0<\epsilon \ll 1$, then $(X',B'+\bbeta +\epsilon \bar \omega')$ is gklt and
\begin{equation}\label{eq-1}K_{X'}+B'+\bbeta _{X'}+\epsilon \omega '\equiv (1+\epsilon)\nu ^*(K_X+B+\bbeta _{X})+(1+\epsilon)B_{X'}^{\leq 0}+\epsilon {\rm Ex}(\nu) .\end{equation}
Since $K_{X'}+B_{X'}^{\geq 0}+\frac \epsilon {1+\epsilon} {\rm Ex}(\nu)+\bbeta _{X'}=\nu ^*(K_X+B+\bbeta _X)+B_{X'}^{\leq 0}+\frac \epsilon {1+\epsilon} {\rm Ex}(\nu)$ where $B_{X'}^{\leq 0}+\frac \epsilon {1+\epsilon} {\rm Ex}(\nu)$ is $\nu$-exceptional, then
by Lemma \ref{l-mms}, $(X,B+\bbeta )$ has a log terminal model iff so does $(X',B_{X'}^{\geq 0}+\frac \epsilon {1+\epsilon} {\rm Ex}(\nu)+\bbeta )$.
Since \[K_{X'}+B'+\bbeta _{X'}+\epsilon \omega '\equiv (1+\epsilon)(K_{X'}+B_{X'}^{\geq 0}+\frac \epsilon {1+\epsilon} {\rm Ex}(\nu)+\bbeta _{X'})\]
then $(X',B_{X'}^{\geq 0}+\frac \epsilon {1+\epsilon} {\rm Ex}(\nu)+\bbeta )$ has a log terminal model iff so does $(X',B'+\bbeta +\epsilon\bar \omega ')$ (see Lemma \ref{l-dltmultiples}).
Note that $(X',B'+\bbeta +\epsilon\bar \omega ')$ is gdlt, $\bbeta _{X'}+\epsilon \omega '$ is K\"ahler and \[K_{X'}+B'+\bbeta _{X'}+\epsilon \omega '\equiv (1+\epsilon)G'+(1+\epsilon)B_{X'}^{\leq 0}+\epsilon {\rm Ex}(\nu) +(1+\epsilon)\omega '\] where 
$D':=(1+\epsilon)G'+(1+\epsilon)B_{X'}^{\leq 0}+\epsilon {\rm Ex}(\nu)$ has simple normal crossings and $(1+\epsilon)\omega '$ is K\"ahler.

We now claim that $N_+(K_{X'}+B'+\bbeta _{X'} + \epsilon \omega ')={\rm Supp}(D')={\rm Supp}(G')$. To see this note that by
\eqref{eq-1} \[K_{X'}+B'+\bbeta _{X'} + \epsilon \omega '\equiv (1+\epsilon)(G'+\omega ')+(1+\epsilon)B_{X'}^{\leq 0}+\epsilon {\rm Ex}(\nu).\]
By what we have seen above $N_+((1+\epsilon)(G'+\omega '))=N_+(G'+\omega ')={\rm Supp}(G')$. Since \[{\rm Supp}((1+\epsilon)B_{X'}^{\leq 0}+\epsilon {\rm Ex}(\nu))\subset {\rm Supp}(G')\]
then by Lemma \ref{l-N+C}
\[N_+((1+\epsilon)(G'+\omega ')+(1+\epsilon)B_{X'}^{\leq 0}+\epsilon {\rm Ex}(\nu))={\rm Supp}(G')\] and the claim follows.

Replacing $(X,B+\bbeta, D+\omega )$ by \[(X',B'+\bbeta +\epsilon \bar \omega ', D' +(1+\epsilon) \omega '),\] we may assume that $(X,B+\bbeta,D+\omega)$ satisfies properties (1-3) of Definition \ref{d-M}, and hence $(X,B+\bbeta, D+\omega )$ belongs to $\mathcal M _n$ iff $(X,B+\bbeta_X)$ does not have a log terminal model. By Proposition \ref{p-Md},  $\mathcal M _n$ is empty and so
$(X,B+\bbeta_X)$ has a log terminal model. By Theorem \ref{t-bpf} this is a good log terminal model.
\end{proof}

\subsection{The non big case} 
\begin{proposition}\label{p-notbigmm} Assume Theorems \ref{t-bpf}$_n$ and  \ref{t-KahlerBCHM}$_{n-1}$ hold.
Let $(X,B+\bbeta )$ be a compact, K\"ahler, $n$-dimensional gklt pair such that $B+\bbeta _X$ is modified big and $K_X+B+\bbeta _X$ is pseudo-effective but not big, then $(X,B+\bbeta )$ has a good log terminal model.
\end{proposition}
\begin{proof}
By standard arguments we may replace $X$ by a resolution and so we may assume that $X$ is smooth (see Lemma \ref{l-mms}). 
Since $B+\bbeta _X$ is big, but $K_X+B+\bbeta _X$ is not big, then $K_X$ is not pseudo-effective and hence the MRC $f:X\dasharrow Z$ is a pseudo-holomorphic map such that $\dim (X/Z)>0$ and $K_Z$ is pseudo-effective (Theorem \ref{t-ou}). Replacing $Z$ by a higher model, we may assume that $Z$ is a K\"ahler manifold. Let $\nu :X'\to X$ be a resolution such that $f':X'\to Z$ is a morphism and write
\[K_{X'}+B_{X'}+\bbeta _{X'}=\nu ^*(K_X+B+\bbeta _X).\]
Let $B'=B_{X'}^{\geq 0}+\epsilon {\rm Ex}(\nu)$ for some $0<\epsilon \ll 1$, then $(X',B'+\bbeta)$ is also gklt and,
by Lemma \ref{l-mms}, $(X,B+\bbeta )$ has a log terminal model iff so does $(X',B'+\bbeta )$. Note that since $K_{X'}+B'+\bbeta _{X'}$ is not big and $K_{Z}$ is pseudo-effective, then $K_{X'/Z}+B'+\bbeta _{X'}$ is not big. By Theorem \ref{t-psef}, $K_{X'}+B'+\bbeta _{X'}$ is not big over $Z$.
Replacing  $(X,B+\bbeta _X)$ by $(X',B'+\bbeta_{X'})$, we may therefore assume that $f:X\to Z$ is a projective morphism of K\"ahler manifolds and $K_{X}+B+\bbeta _{X}$ is not big over $Z$.
We may also assume that $\bbeta _{X}$ is K\"ahler (cf. Lemma \ref{lem: g-kltperturb}) and in particular $K_X+B$ is not pseudo-effective over $Z$ and $K_X+B+t\bbeta _X$ is K\"ahler for $t\gg 0$. Since general fibers of $f$ are rationally connected, then by Lemma \ref{l-rccfibres}, $f^*:H^0(\Omega ^2_Z)\to H^0(\Omega _X^2)$ is an isomorphism, and so by Lemma \ref{l-rel} $\bbeta _X\equiv c_1(L)+f^*\gamma$ where $L$ is an $\R$-line bundle and $\gamma \in H^{1,1}(Z)$. 
We now run the $K_X+B$ minimal model program with scaling of $\bbeta _{X}$ over $Z$. By Proposition \ref{p-relMMP}, this MMP \[X\dasharrow X_1\dasharrow \ldots  \dasharrow X_n=X'\]terminates with a $(K_X+B)$-Mori fiber space $g:X'\to Z'$ over $Z$ such that $K_{X'}+B'+ \bbeta _{X'}\equiv _{Z'}0$. 
By Theorem \ref{t-gcbf}, \[K_{X'}+B'+\bbeta _{X'}\equiv g^*(K_{Z'}+B_{Z'}+\ggamma _{Z'})\]
where $(Z',B_{Z'}+\ggamma)$ is a gklt pair and $B_{Z'}+\ggamma_{Z'}$ is modified big.

Since $K_X+B+\bbeta _X$ is pseudo-effective, $K_{Z'}+B_{Z'}+\ggamma _{Z'}$ is also pseudo-effective. As in the proof of Theorem \ref{t-ind2}, after replacing $X'\to Z'$ with a birational model $\hat{X}'\to \hat{Z}'$ such that $\hat Z'\to Z'$ is a small strongly $\Q$-factorial modification and $\hat X'\dasharrow X'$ is small, we may assume that $\hat X'$ and $\hat Z'$ are strongly $\Q$-factorial.

By Theorem \ref{t-MMPscaling}$_{n-1}$ there is a $(K_{Z'}+B_{Z'}+\ggamma _{Z'})$-MMP 
\[Z'\dasharrow Z_1\dasharrow \ldots  \dasharrow Z_n\] so that 
$K_{Z_n}+B_{Z_n}+\ggamma _{Z_n}$ is nef.
Following the proof of Theorem \ref{t-ind2}, we can lift this minimal model program to a sequence \[\hat X'\dasharrow X'_1\dasharrow \ldots  \dasharrow X'_n\] such
that each $X'_i\dasharrow X'_{i+1}$ is a sequence of steps of the $K_{\hat X'}+B_{\hat X'}+\bbeta _{\hat X'}$ minimal model program,  $g_i:X'_i\to Z_i$ is a morphism and $K_{X'_i}+B_{X'_i}+\bbeta _{X'_i}=g_i^*(K_{Z_i}+B_{Z_i}+\ggamma _{Z_i})$. In particular $K_{X'_n}+B_{X'_n}+\bbeta _{X'_n}$ is nef
and $(X'_n,B_{X'_n}+\bbeta _{X'_n})$ is the sought log terminal model. By Theorem \ref{t-bpf} this is a good log terminal model.\end{proof}

The above proof also gives relative log terminal models.
\begin{theorem}\label{t-relMMP} Assume Theorems \ref{t-bpf}$_n$ and  \ref{t-KahlerBCHM}$_{n-1}$ hold. Let $f:X\to S$ be a morphism of normal compact K\"ahler varieties where $n=\dim X$. \begin{enumerate}
\item If $(X,B+\bbeta )$ is a gklt pair such that $B+\bbeta _X$ is modified big over $S$ and $K_X+B+\bbeta _X$ is pseudo-effective over $S$, then there exists a $K_X+B+\bbeta _X$ log terminal model over $S$.
\item If $(X,B+\bbeta )$ is a gklt pair such that $B+\bbeta _X$ is modified big over $S$ and $K_X+B+\bbeta _X$ is nef over $S$, then $K_X+B+\bbeta _X$ is semiample over $S$.
\end{enumerate}
\end{theorem}
\begin{proof} (1) This follows by the same proof as above, simply by replacing $(X,B+\bbeta )$ by $(X,B+\bbeta +\bar \eta)$ where $\eta =f^*\omega _S$ and $\omega _S$ is a K\"ahler form on $S$ such that $\int _C\omega _S> 2\dim X$ for any curve $C\subset S$. Then the arguments of Propositions \ref{p-bigmm} and \ref{p-notbigmm} apply.

(2) By Theorem \ref{t-small Q factorial model}, we may assume that $X$ is strongly $\Q$-factorial. Let $R=\R^+[C]$ be a $K_X+B+\bbeta _X+\eta$ negative extremal ray, then we may assume that $0>(K_X+B+\bbeta _X)\cdot C\geq -2\dim X$. By our choice of $\eta$, we must have $\eta \cdot R=0$. Since $K_X+B+\bbeta _X$ is nef over $S$, this is a contradiction and so there are no $K_X+B+\bbeta _X+\eta$ negative extremal rays. By Theorem \ref{t-cone}, $K_X+B+\bbeta _X+\eta$ is nef.  Replacing $\eta$ by a multiple, we may assume that $B+\bbeta _X+\eta$ is big. By Theorem \ref{t-bpf}$_n$ we know $K_X+B+\bbeta _X+\eta$ is semiample, and hence there is a morphism $g:X\to Z$ such that $K_X+B+\bbeta _X+\eta=g^*\omega _Z$ where $\omega _Z$ is K\"ahler.
Since $-K_X\equiv _Z B+\bbeta _X+\eta$ is big over $Z$, then $g:X\to Z$ is Moishezon, and in particular the fibers are covered by curves. It suffices to show that there is an induced morphism $Z\to S$. Suppose that this is not the case, then there is a curve $C\subset X$ such that $\int _C g^*\omega _Z=0$, but $\int _C f^*\omega _S\ne 0$ and hence \[-\int _C(K_X+B+\bbeta _X)=\int _C f^*\omega _S>2\dim X.\] By Theorem \ref{t-cone} there exists a $(K_X+B+\bbeta _X)$-negative, $(K_X+B+\bbeta _X+\eta )$-trivial extremal ray $R=\R^+[\Sigma ]$ such that $\int _\Sigma (K_X+B+\bbeta _X)\geq -2\dim X$. But then $\int _\Sigma (K_X+B+\bbeta _X+\eta )>0 $ which is a contradiction and the claim is proven.\end{proof}
Following \cite[Corollary 1.1.5]{BCHM10} and \cite[Theorem 3.18]{DHY23}, we prove the following result on the geography of log canonical models.

\begin{theorem}\label{t-MMPgeography} Assume Theorems \ref{t-bpf}$_n$ and  \ref{t-KahlerBCHM}$_{n-1}$ hold. Let $f:X\to S$ be a morphism of normal compact K\"ahler varieties where $n=\dim X$, and $\nu :X'\to X$ a resolution.
Let $\mathcal P\subset {\rm Div}_\R (X)\times H^{1,1}_{\rm BC}(X')$ be a compact polyhedral set such that for any $(B, \beta ')\in \mathcal P$ and $\bbeta =\bar \beta'$
\begin{enumerate}
\item $\beta '$ is nef over $S$,
\item $(X,B+\bbeta)$ is gklt, 
\item $B+\bbeta _X$ is modified big over $S$.
\end{enumerate}
Let \[\mathcal E=\{ (B,\bbeta _{X'})\in \mathcal P\ {\rm s.t.}\ K_X+B+\bbeta _X\ \text{is pseudo-effective over} \ S\},\]
then $\mathcal E$ is a compact polyhedral set and there exists a finite polyhedral decomposition $\mathcal E=\cup _{i\in I} (\mathcal E_i)^\circ$ and finitely many
bimeromorphic maps $\{\psi _i:X\dasharrow X_i\}_{i\in I}$, and $\{\varphi _{j}:X\dasharrow X_j\}_{j\in J}$ over $S$, such that \begin{enumerate}
\item if $\psi :X\dasharrow Y$ is a log canonical model for
$(X,B+\bbeta _X)$ over $S$ where $(B,\bbeta _{X'})\in (\mathcal E_i)^\circ $ then $\psi =\psi _i$, and
\item if $\varphi:X\dasharrow Y$ is a weak log canonical model for
$(X,B+\bbeta _X)$ over $S$ where $(B,\bbeta _{X'})\in \mathcal E$ then $\varphi =\varphi _{j}$ for some $j\in J$.\end{enumerate}
Conclusion (2) also holds if instead of condition (3) we assume the following

(3')  $K_X+B+\bbeta _X$ is big over $S$.

\end{theorem}
\begin{proof} The proof follows by the arguments of \cite[Theorem 3.18]{DHY23} which in turn are inspired by \cite[Corollary 1.1.5]{BCHM10}. 
Several of the arguments in this proof are by now fairly standard and we refer the reader to the proof of \cite[Theorem 3.18]{DHY23} for further details.
We also note that case (3') follows from case (3) by routine arguments.

We proceed by induction on the dimension of $\mathcal P$. 
The case when $\dim \mathcal P=0$ follows by Theorem \ref{t-relMMP}.
Since $\mathcal E\subset \mathcal P$ is closed, by compactness, it suffices to prove the statement locally around any given $(B_0,(\bbeta_0) _{X'})\in \mathcal E$. By Theorem \ref{t-relMMP}, we have a good log terminal model $\phi :X\dasharrow X^{\rm m}$ and a log canonical model $\psi : X^{\rm m}\to X^{\rm c}$ for $(X,B_0+\bbeta _0)$ over $S$. By standard arguments (see e.g. loc. cit.), shrinking $\mathcal P$, we may assume that $a_E(X,B+\bbeta )<a_E(X^{\rm m},B^{\rm m}+\bbeta )$ for every $\phi $-exceptional divisor $E$ on $X$ and every $(B,\bbeta _{X'})\in \mathcal P$. In particular if $X^{\rm m}\dasharrow Y$ is a $(K_{X^{\rm m}}+B^{\rm m}+\bbeta_{X^{\rm m}})$-log canonical model over $S$ (resp. log terminal model over $S$), then $X\dasharrow Y$ is a $(K_{X}+B+\bbeta_{X})$-log canonical model over $S$ (resp. log terminal model over $S$). 

Let $\mathcal P ^{\rm m}=\phi _*\mathcal P\subset {\rm Div}_\R(X^{\rm m})\times H^{1,1}_{\rm BC}(X')$. Let \[\mathcal F=\{ (B^{\rm m},\bbeta _{X'})\in  \mathcal P^{\rm m}\ |\ K_{X^{\rm m}}+B^m+\bbeta _{X^{\rm m}}\ {\rm is\ psef\ over}\   X^{\rm c}\}.\]
By induction, if  $\partial \mathcal P^{\rm m}$ is the boundary of $\mathcal P^{\rm m}$, then $\partial \mathcal P^{\rm m}\cap \mathcal F$ is a compact polyhedral set and there is a finite polyhedral decomposition $\partial \mathcal P^{\rm m}\cap \mathcal F=\cup _{i=1}^k(\mathcal Q_i)^\circ$ and finitely many meromorphic maps 
$g_i:X^{\rm m}\dasharrow X_i$ over $X^{\rm c}$ such that if $(B^{\rm m},\bbeta_{X'})\in (\mathcal Q_i)^\circ$, then $g_i:X^{\rm m}\dasharrow X_i$ is the log canonical model for $({X^{\rm m}},B^{\rm m}+\bbeta)$ over $X^{\rm c}$ (resp. a log terminal model  for $({X^{\rm m}},B^{\rm m}+\bbeta)$ over $X^{\rm c}$). 
Consider now the pairs $({X^{\rm m}},B^{\rm m}_\lambda +\bbeta_\lambda )$ where $(B^{\rm m}_\lambda +\bbeta_\lambda ):= (1-\lambda )(B^{\rm m}_0+\bbeta_0 )+\lambda (B^{\rm m}+\bbeta)$ for some $(B^m,\bbeta_{X'})\in \partial \mathcal P ^{\rm m}$. Note that
\[K_{X^{\rm m}}+B^{\rm m}_0+(\bbeta_0)_{X^{\rm m}}\equiv \psi ^* \omega \equiv _{X^{\rm c}}0\]
where $\omega $ is a K\"ahler form on $X^{\rm c}$.
It follows that $K_{X^{\rm m}} +B^{\rm m}_\lambda +(\bbeta_\lambda )_{X^{\rm m}}\equiv _{X^{\rm c}}\lambda (K_{X^{\rm m}} +B^{\rm m}+(\bbeta)_{X^{\rm m}})$, and so $K_{X^{\rm m}} +B^{\rm m}_\lambda +(\bbeta_\lambda )_{X^{\rm m}}$ is pseudo-effective over ${X^{\rm c}}$ iff so is $ (K_{X^{\rm m}} +B^{\rm m}+\bbeta_{X^{\rm m}})$.  It follows that $\mathcal F$ is the cone over $\partial \mathcal P^{\rm m}\cap \mathcal F$ with vertex $(B_0,(\bbeta _0)_{X'})$, and hence $\mathcal F$ is also a compact polyhedral set.

\begin{claim}\label{c-51} There exists a constant $\bar \lambda >0$ such that if $(B^{\rm m},\bbeta_{X'})\in (\mathcal Q_i)^\circ$, then $g_i$ is the log canonical model (resp. a log terminal model) for $K_{X^{\rm m}}+B^{\rm m}_\lambda +\bbeta_\lambda$ over $S$ for any $0<\lambda \leq \bar \lambda $. 
\end{claim}
\begin{proof} By definition of log canonical model (resp. of log terminal model), $a_E({X^{\rm m}},B^{\rm m}+\bbeta)\leq a_E(X_i,B_i+\bbeta)$ for every divisor $E$ over $X$ (resp. $a_E({X^{\rm m}},B^{\rm m}+\bbeta)< a_E(X_i,B_i+\bbeta)$ for all $g_i$-exceptional divisors $E$), and since $g_i$ is $(K_{X^{\rm m}}+B^{\rm m}_0+\bbeta_0)$-trivial, $a_E({X^{\rm m}},B^{\rm m}_0+\bbeta_0)=a_E(X_i,(B_0)_i+\bbeta_0)$ for every divisor $E$ over $X$. 
Therefore \[ a_E({X^{\rm m}},B^{\rm m}_\lambda +\bbeta_\lambda )\leq a_E(X_i,(B_\lambda)_i+\bbeta_\lambda)\] for $\lambda > 0$ and every divisor $E$ over $X$ (resp. $a_E({X^{\rm m}},B^{\rm m}_\lambda +\bbeta_\lambda )< a_E(X_i,(B_\lambda)_i+\bbeta_\lambda)$ for all $g_i$-exceptional divisors $E$). In particular $(X_i,(B_\lambda)_i+\bbeta_\lambda)$ is gklt.

Since $\omega $ is K\"ahler, then there exists $\delta >0$ such that $\omega \cdot C>\delta$ for any curve $C$ on $X^{\rm c}$.
Let $\bar \lambda :=\frac \delta{\delta +2\dim X}$.
Suppose that  $K_{X_i}+(B_\lambda)_i+(\bbeta_\lambda)_{X_i}$ is not K\"ahler (resp. nef) over $S$, then there exists an extremal ray $R=\R ^+[C]$ such that $(K_{X_i}+(B_\lambda)_i+(\bbeta_\lambda)_{X_i})\cdot C\leq 0$ (resp. $(K_{X_i}+(B_\lambda)_i+(\bbeta_\lambda)_{X_i})\cdot C< 0$) see \cite[Corollary 4.2]{HLX26}.
Since $K_{X_i}+B_i+\bbeta_{X_i}$ is K\"ahler over $X^{\rm c}$ (resp. nef over $X^{\rm c}$), then $(\psi _i)_*C\ne 0$ where $\psi _i:X_i\to X^{\rm c}$. By the cone theorem (see Theorem \ref{t-cone}), we may assume that $(K_{X_i}+B_i+\bbeta_{X_i})\cdot C\geq -2 \dim X$. Since $\lambda \leq \bar \lambda$, then 
\[ (K_{X_i}+(B_\lambda)_i+(\bbeta_\lambda)_{X_i})\cdot C= (1-\lambda)g _i^*\omega \cdot C+\lambda (K_{X_i}+B_i+\bbeta_{X_i}) \cdot C\]
\[> (1-\lambda )\delta -2\lambda \dim X \geq \delta - \lambda (\delta +2\dim X)\geq 0\] (resp. $(K_{X_i}+(B_\lambda)_i+(\bbeta_\lambda)_{X_i})\cdot C\geq 0$). This is impossible and the claim follows.
\end{proof}

\begin{claim} If $(B,\bbeta _{X'})\in \partial \mathcal P$ then for every $0\leq \lambda \leq \bar\lambda$, we have that $K_{X^{\rm m}}+B^{\rm m}_\lambda +(\bbeta_\lambda)_{X^{\rm m}}$ is pseudo-effective over $S$ iff it is pseudo-effective over $X^{\rm c}$. In particular $\mathcal E$ is a compact polyhedral set.
\end{claim}
\begin{proof} If $K_{X^{\rm m}}+B^{\rm m}_\lambda +(\bbeta_\lambda)_{X^{\rm m}}$ is pseudo-effective over $S$, then clearly it is pseudo-effective over $X^{\rm c}$. 
Conversely, suppose that  $K_{X^{\rm m}}+B^{\rm m}_\lambda +(\bbeta_\lambda)_{X^{\rm m}} $ is pseudo-effective over $X^{\rm c}$, then so is $K_{X^{\rm m}}+B^{\rm m}+\bbeta_{X^{\rm m}}$. Thus $(B^{\rm m},\bbeta _{X'})\in \partial  \mathcal P^{\rm m}\cap \mathcal F$ and hence  $(B^{\rm m},\bbeta _{X'})\in  \mathcal Q_i^\circ$ for some $i$. By Claim \ref{c-51}, $K_{X^{\rm m}}+B^{\rm m}_\lambda +(\bbeta_\lambda)_{X^{\rm m}}$ is pseudo-effective over $S$, for every $0\leq \lambda \leq \bar\lambda$.

By what we have seen above, in a neighborhood of $(B_0,(\bbeta _0)_{X'})$ we have that $\mathcal E\cong \mathcal F$. Since $\mathcal F$ is  polyhedral, it follows that $\mathcal E$ is a compact polyhedral set on a neighborhood of $(B_0,(\bbeta _0)_{X'})$.  Since $\mathcal E$ is a compact convex locally polyhedral set, it is in fact a compact polyhedral set.
\end{proof}

Claim (1) is now seen to hold by considering the decomposition $\mathcal P=\{(B,\bbeta_{X'})\}\cup_{i=1}^k \mathcal P_i^\circ$ where $\mathcal P_i$ is the polytope spanned by $\mathcal Q_i$ and 
$(B_0,(\bbeta _0)_{X'})$ and the maps $\psi _0=\psi\circ \phi , \{\phi _i=g_i\circ \phi\}$. 

The above arguments also prove the following.
\begin{claim}\label{c-ltms} After refining the polyhedral decomposition $\mathcal P=\cup \mathcal P_i$, we may assume that there are finitely many birational contractions $\phi _{i}:X\dasharrow Y_i$ such that if $(B,\bbeta _{X'})\in \mathcal P _i$ then $\phi_i:X\dasharrow Y_i$ is a log terminal model for
$(X,B+\bbeta _X)$ over $S$.
\end{claim}

We will now prove (2). 
Replacing $\bbeta$ by $\bbeta +\overline{f^*\omega}$ for an appropriate K\"ahler form $\omega$ on $S$, we may assume that $\bbeta _X$ is big. 
By standard reductions, we may assume that $X$ is smooth, $\bbeta =\overline{\bbeta _X}$ and $\bbeta _X-\omega$ is nef for some K\"ahler form $\omega$ (see Lemma \ref{lem: g-kltperturb}). Let $\{\gamma _1,\ldots, \gamma _\rho\}$ be K\"ahler forms that span $H^{1,1}_{\rm BC}(X)$ such that $\omega -\sum _{i=1}^\rho \gamma _i$ is K\"ahler. Let \[\mathcal Q=\{ (B,\bbeta _X+\sum t_i\gamma _i)| -1\leq t_i\leq 1\}.\]
By Claim \ref{c-ltms} there are finitely many birational contractions $\phi _i:X\dasharrow X_i$ and a decomposition $\mathcal Q=\cup \mathcal Q_i$ such that if $(X,B+\bbeta _X+\sum t_i\gamma _i)\in \mathcal Q_i$, then $\phi _i$ is a log terminal model for $(X,B+\bbeta _X+\sum t_i\bar \gamma _i)$. Let $g_i:X_i\to Y_i$ be the corresponding log canonical model, and $\psi _i=g_i\circ \phi _i$.

Suppose that $\psi :X\dasharrow Y$ is a weak log canonical model over $S$ for $(X,B+\bbeta _X)\in \mathcal P$. Let $\eta$ be a K\"ahler form on $Y$. If $p:W\to X$, $q:W\to Y$ is a resolution of the indeterminacies of $\psi$, then let $\eta _X=p_*(q^*\eta)$. Note that $q^*\eta \in H^{1,1}_{\rm BC}(W)=H^{1,1}(W)$ as $W$ is smooth. Since $X$ is smooth, $H^{1,1}_{\rm BC}(X)=H^{1,1}(X)$, $h^{0,2}(X)=h^{2,0}(X)=h^{2,0}(W)=h^{0,2}(W)$ whilst $h^2(W,\C)=h^2(X,\C)+r$ where $r$ is the number of exceptional divisors. But then $H^{1,1}(W)\cong H^{1,1}(X)\oplus\langle E_1,\ldots , E_r\rangle$ where $E_i$ are the exceptional divisors. Therefore $\eta_X:=p_*q^*\eta\in H^{1,1}(X)$. Since the $\gamma _i$ generate $H^{1,1}(X)$, we may write $\eta _X=\sum t_i\gamma _i$. Rescaling $\eta$, we may assume that $|t_i|<1$.
Since $\psi$ is the log canonical model over $S$ for $K_X+B+\bbeta _X+\eta _X$, then it is also the log canonical model over $S$ for $(X,B+\bbeta +\sum t_i\bar \gamma _i)$ and hence $\psi=\psi _i$ where $(B,\bbeta _X+\sum t_i \gamma _i)\in \mathcal Q_i$.

\end{proof}

\subsection{Running the minimal model program with scaling}
Let $(X,B+\bbeta )$ be a strongly $\Q$-factorial gklt pair where $B+\bbeta _X$ is big. We have seen above that $(X,B+\bbeta )$ admits a log terminal model or a Mori fiber space. We will now show that the minimal model program with scaling terminates with a log terminal model or a Mori fiber space.
\begin{theorem}\label{t-runMMP} Assume Theorems \ref{t-bpf}$_n$ and  \ref{t-KahlerBCHM}$_{n-1}$ hold. Let $(X,B+\bbeta )$ be a strongly $\Q$-factorial compact K\"ahler gklt pair and $\omega \in H^{1,1}_{\rm BC}(X)$ a modified K\"ahler form such that $K_X+B+\bbeta _X+\omega $ is nef. Let $f: X\to S$ be a proper morphism between compact normal K\"ahler varieties.

Then we may run the $K_X+B+\bbeta _X$ minimal model program with scaling of $\omega$ over $S$, which ends with \begin{enumerate} \item a log terminal model if $K_X+B+\bbeta _X$ is pseudo-effective over $S$ and $B+\bbeta _X$ is big over $S$, or
\item a Mori fiber space if $K_X+B+\bbeta _X$ is not pseudo-effective over $S$.\end{enumerate}
\end{theorem}
\begin{proof} 
We may replace $(X,B+\bbeta)$ by $(X,B+\bbeta+\overline{f^*\omega_S})$ for a  K\"ahler form $\omega_S$ such that $\omega _S\cdot C>2\dim X$ for any curve $C\subset S$. Then any $(K_X+B+\bbeta_X)$-MMP is automatically a MMP over $S$. Since the cone theorem holds and we can run the MMP by Proposition \ref{prop: MMP can be run}, 
it suffices to show that the given minimal model program terminates. To this end we follow \cite{DHY23}. Let \[X\dasharrow X_1\dasharrow \ldots \dasharrow X_n\dasharrow \ldots \]
be the corresponding minimal model program with scaling and let $t_0:=1\geq t_1\geq t_2\geq \ldots\geq 0$ be the corresponding constants such that $K_{X_i}+B_i+\bbeta _{X_i}+t \omega _i$ is nef for $t_i\geq t\geq t_{i+1}$.
Let $\nu:X'\to X$ be a birational map such that $\omega =\nu _*\omega '$ where $\omega '$ is nef and big and $\bbeta =\overline{\bbeta _{X'}}$. We let $\Omega '=[0,\omega ']$ in case (1) and $\Omega '=[\epsilon \omega ',\omega ']$ in case (2) where $0<\epsilon \ll 1$ so that $K_X+B+\bbeta _X+\epsilon \omega $ is not pseudo-effective over $S$. Let $\Omega '=\cup_{i\in I}(\Omega _i')^\circ$ be the polyhedral decomposition of Theorem \ref{t-MMPgeography} and $\phi _{i,j}$ the rational maps corresponding to weak log canonical models over $S$ of $(X,B+\bbeta _X+t\omega)$ where $t\omega '\in \Omega '$.
Since each $f _k:X\dasharrow X_k$ is a  weak log canonical model over $S$ for $K_X+B+\bbeta _X+t\omega$ where $t_k\geq t\geq t_{k+1}$, it follows that $f_k=\phi _{i,j}$. Note that $f_k\ne f_l$ for $k< l$ since in this case $a_{E}(X_k,B_k+\bbeta)<a_{E}(X_l,B_l+\bbeta)$ for any divisor $E$ whose center is contained in the indeterminacy locus of $X_k\dasharrow X_l$. Since $\{\phi _{i,j}\}$ is finite, so is $f_k:X\dasharrow X_k$. It follows that this minimal model program with scaling over $S$  ends with a log terminal model over $S$ for $K_X+B+\bbeta _X+t_n\omega$. If $t_n=0$, then we have a log terminal model over $S$ for $K_X+B+\bbeta _X$, and otherwise we have $g:X_n\to Z$ a $(K_X+B+\bbeta _X+t_n\omega)$-trivial, $(K_X+B+\bbeta _X)$-Mori fiber space over $S$.
\end{proof}

\section{Non-big contraction theorem}\label{s-contr-nonbig}

\begin{prop}
Assume that Theorem \ref{t-bpf}$_n$ and Theorem \ref{t-MMPscaling}$_n$ hold, then Theorem \ref{t-nkltcontraction}$_n$ holds.
\end{prop}

\begin{proof}
By Theorem \ref{t-ind1}, Theorem \ref{t-nkltcontraction}$_{n,{\rm big}}$ holds, and therefore, we may assume $\alpha=K_X+B+\bbeta_X$ is not big. By passing to a strongly $\Q$-factorial dlt modification (see Theorems \ref{t-dltmodel}, \ref{t-small Q factorial model}) we may assume $X$ is strongly $\Q$-factorial and klt. 

Since $\alpha=K_X+B+\bbeta _X$ is nef and NQC, then by Lemma \ref{lem: NQC can run alpha-trivial MMP}, there is a $t_0>0$ such that the $(K_X+t_0\alpha)$-MMP is $\alpha$-trivial. By Theorem \ref{t-runMMP} we may run the $(K_X+t_0\alpha)$-MMP with scaling.
Since $(t_0+1)\alpha$ is not big but $B+\bbeta_X$ is big, then $K_X+t_0\alpha$ is not pseudo-effective and so this MMP ends with a Mori-fiber space $X\dasharrow X'\xrightarrow{f'} Z$. By Theorem \ref{t-bpf}$_n$ we know the last step $f':X'\to Z$ is a projective contraction such that $Z$ is K\"ahler, and by the canonical bundle formula we know $Z$ has gklt and hence rational singularities (see Theorem \ref{t-gcbf} and Lemma \ref{L-gklt-rational}).

Let $X''$ be a common resolution of $X$ and $X'$ with morphisms $p: X''\to X$ and $q: X''\to X'$, then let $K_{X''}+B''+\bbeta_{X''}=p^*(K_X+B+\bbeta_{X})$ and $K_{X'}+B'+\bbeta_{X'}=q_*(K_{X''}+B''+\bbeta_{X''})$. Notice that 
$$K_{X''}+B''+\bbeta_{X''}=q^*(K_{X'}+B'+\bbeta_{X'})$$ since the $K_X$-MMP $X\dasharrow X'\to Z$ is $\alpha$-trivial. We see that $[K_X+B+\bbeta_X]$ is endowed with a Moishezon contraction if and only if $[K_{X'}+B'+\bbeta_{X'}]$ is. By the definition of multiplier ideals and the vanishing theorem (as in Case 2 of Proposition \ref{prop: construct f_i for Z_i}) we get
\[
p_*\OO_{\lfloor B''\rfloor^{\ge 0}}=\OO_{{\rm NKLT}(X,B+\bbeta)}, \text{ and } q_*\OO_{\lfloor B''\rfloor^{\ge 0}}=\OO_{{\rm NKLT}(X',B'+\bbeta)}.
\]
Therefore, $[\alpha']|_{\text{NKLT}(X',B'+\bbeta)}$ is also endowed with a contraction. After replacing $(X,B+\bbeta)$ by $(X',B'+\bbeta)$ we may assume there is a projective $\alpha$-trivial Mori fiber space $f: X\to Z$, with $Z$ K\"ahler. Then by the vanishing theorem we know $R^if_*\OO_X=0$ for $i>0$, which implies that $Z$ has rational singularities as well. By Lemma \ref{l-rccfibres}, we know $\alpha\equiv f^*\gamma$ for some $\gamma\in H^{1,1}_{\rm BC}(Z)$ and $[\gamma]$ is nef by \cite[Lemma 2.38]{DHP24}. Then $[\gamma]$ is NQC by Lemma \ref{l-NQCpullback} and the fact that $f^*: H^2(Z,\R)\to H^2(X,\R)$ is defined over $\Q$.

By perturbing $B+\bbeta_X$ (see Lemma \ref{lem: g-kltperturb}) we may rewrite 
$$
K_X+B+\bbeta_X\equiv K_X+\tilde B+\tilde\bbeta_X+\omega,
$$ where $\omega$ is a K\"ahler form, and a closed embedding   \[i:{\rm NKLT}(X,\tilde B+\tilde\bbeta)\hookrightarrow {\rm NKLT}(X, B+\bbeta).\] Then by taking the Stein factorization of the composition $\pi\circ i$, where $\pi$ is the contraction defined by $\alpha|_{{\rm NKLT}(X, B+\bbeta)}$, 
it follows that $[\alpha]|_{{\rm NKLT}(X,\tilde B+\tilde\bbeta)}$ is endowed with a contraction. Thus, replacing $(X, B+\bbeta)$ by $(X,\tilde B+\tilde\bbeta)$, we may assume that $\bbeta=\bbeta'+\overline{\omega}$ for some K\"ahler form $\omega$ on $X$.

\vspace{.5em}

\noindent\textbf{Case 1}: Suppose that $\text{NKLT}(X,B+\bbeta)$ dominates $Z$. The canonical injection
\[
\OO_Z\to f_*\OO_{\text{NKLT}(X,B+\bbeta)}
\]
is surjective by Nadel vanishing, and hence an isomorphism. Let 
$$g:\text{NKLT}(X,B+\bbeta)\to T$$ be the contraction defined by $[\alpha]|_{\text{NKLT}(X,B+\bbeta)}$, then we see that $g$ factors through $f|_{\text{NKLT}(X,B+\bbeta)}$ and we get a contraction $h: Z\to T$, which is the contraction defined by $[\gamma]$. Then  $h\circ f: X\to T$ is the contraction defined by $[\alpha]=f^*[\gamma]$ and we are done.

\vspace{.5em}

\noindent\textbf{Case 2}: Suppose that $\text{NKLT}(X,B+\bbeta)$ does not dominate $Z$, i.e. $(X,B+\bbeta)$ is  gklt over an open subset of $Z$. By the canonical bundle formula (see Theorem \ref{t-gcbf}) there is an induced generalized pair structure $(Z,B_Z+\ddelta)$ on $Z$ with $\ddelta$ big such that
\[
K_{X}+B+\bbeta_{X}\equiv  f^*(K_Z+B_Z+\ddelta_Z) .
\]
Therefore it suffices to show $\alpha_Z:=K_Z+B_Z+\ddelta_Z$ is endowed with a contraction. In order to apply the induction hypothesis we need to show that $[\alpha_Z]|_{\text{NKLT}(Z,B_Z+\ddelta)}$ is endowed with a contraction. Since $[\alpha]|_{\text{NKLT}(X,B+\bbeta)}$ is endowed with a contraction, it suffices to show that 
$$f_*\OO_{\text{NKLT}(X,B+\bbeta)}=\OO_{\text{NKLT}(Z,B_Z+\ddelta)}.$$

Replacing $X''$ by a higher resolution we may assume $\bbeta$ descends on $X''$ and the following diagram commutes:

\[\begin{tikzcd}
  X'' \arrow[d,"f''"] \arrow[r,"q"]& X\arrow[d,"f"] \\
  Z'' \arrow[r,"\phi"] & Z
\end{tikzcd}
\]
where
\begin{enumerate}
\item $\phi:Z''\to Z$ is a log resolution of $(Z,B_Z)$ and $\ddelta$ descends to $Z''$.
\item There is a Zariski open subset $U\subset Z$ such that $(X'',\Supp B''+\mathrm{Exc}(q))$ is log smooth over $U$.
\item $\Delta= Z''\backslash U$ is a simple normal crossing divisor and $\Supp f''^{-1}(\Delta)$ is also a simple normal crossing divisor.
\end{enumerate}

\begin{claim} It suffices to show that $f''_*\OO_{X''}(-\lfloor B''\rfloor)=\OO_{Z''}(-\lfloor B_{Z''}\rfloor)$.
\end{claim}
 \begin{proof} Suppose that $f''_*\OO_{X''}(-\lfloor B''\rfloor)=\OO_{Z''}(-\lfloor B_{Z''}\rfloor)$.
We have \[
q_*\OO_{X''}(-\lfloor B''\rfloor^{\ge 0})=q_*\OO_{X''}(-\lfloor B''\rfloor^{\ge 0}+\lfloor B''\rfloor^{\le 0})=\mathcal{I}_{\text{NKLT}(X,B+\bbeta)}
\]
and 
\[
\phi_* \OO_{Z''}(-\lfloor B_{Z''}\rfloor)=\phi_*\OO_{Z''}(-\lfloor B_{Z''}\rfloor^{\ge 0}+\lfloor B_{Z''}\rfloor^{\le 0})=\mathcal{I}_{\text{NKLT}(Z,B_Z+\ddelta)}.
\]
By Nadel vanishing, we also have a short exact sequence
\[
0\to f_*\mathcal{I}_{\text{NKLT}(X,B+\bbeta)}\to f_*\OO_{X}\to f_*\OO_{\text{NKLT}(X,B+\bbeta)}\to 0.
\]
Since \[f_*\mathcal{I}_{\text{NKLT}(X,B+\bbeta)}=\phi_* f''_*\OO_{X''}(-\lfloor B''\rfloor)=\phi_* \OO_{Z''}(-\lfloor B_{Z''}\rfloor),\]
we obtain the short exact sequence
\[0\to \mathcal J(Z,B_Z+\ddelta)\to \OO _Z\to f_*\OO_{\text{NKLT}(X,B+\bbeta)}\to 0\] 
and hence $\OO_{\text{NKLT}(Z,B_Z+\ddelta)}=f_*\OO_{\text{NKLT}(X,B+\bbeta)}$. \end{proof} 

First, we check $\OO_{Z''}(-\lfloor B_{Z''}\rfloor)\subseteq f''_*\OO_{X''}(-\lfloor B''\rfloor)$. We will need the following inequalities.
Let $K_{Z''}+B_{Z''}+\ddelta_{Z''}=\phi^*(K_Z+B_Z+\ddelta_Z)$, then by the assumption above the support of $B_{Z''}:=\sum_{i}c_i D_i$ (the prime decomposition) is simple normal crossing and
\[
1-c_i= \min_{f''(E_{i,j})=D_i}\{c\in\R~|~ \mult_{E_{i,j}}(B''+cf''^*D_i)= 1 \}.
\]
\begin{claim}\label{c-cbf-coeff} The above equality holds for all divisors over $Z''$, and in particular we have
\begin{align}
1-c_i= \min_{f''(E_{i,j})\subseteq D_i}\{c\in\R~|~ \mult_{E_{i,j}}(B''+cf''^*D_i)= 1 \}.
\end{align}\end{claim}
\begin{proof} Replacing $B''$ by $B''+\sum (1-c_i){f''}^*D_i$, we can assume that $B_{Z''}$ is reduced. If the equality below does not hold for all divisors on $X''$, then $(X'',B'')$ is not sub-lc over some point $Q\in Z''$. Let $\mu:Z'''\to Z''$ be a resolution such that $Q$ is a divisor on $Z'''$, then ${\rm mult}_Q(B_{Z'''})>1$. This contradicts the fact that $\ddelta$ descends to $Z''$ and hence $K_{Z'''}+B_{Z'''}=\mu ^*(K_{Z''}+B_{Z''})$ is sub-lc.
\end{proof}
The inclusion  $\OO_{Z''}(-\lfloor B_{Z''}\rfloor)\subseteq f''_*\OO_{X''}(-\lfloor B''\rfloor)$ can be checked locally on  open subsets of $Z''$, and so it follows from the fact that by Claim \ref{c-cbf-coeff} 
\[
\mult_{E_{i,j}}(-B''+\lfloor c_i\rfloor f''^*D_i)>\mult_{E_{i,j}}(-B''+ (c_i-1) f''^*D_i)\ge -1.
\]

For the other direction, note that since the horizontal part $(-\lfloor B''\rfloor)^{h}$ is an effective $q$-exceptional divisor, and hence over a Zariski open set the torsion free sheaf $f''_*\OO_{X''}(-\lfloor B''\rfloor)$ is contained in $\OO_{Z''}$. Thus there is a canonical inclusion $f''_*\OO_{X''}(-\lfloor B''\rfloor)\hookrightarrow\mathfrak{M}(Z'')$, where $\mathfrak{M}(Z'')$ is the sheaf of meromorphic functions on $Z''$. A local section $\sigma$ of $f''_*\OO_{X''}(-\lfloor B''\rfloor)$ corresponds to a meromorphic function on $Z''$ such that 
\[
f''^*\Div(\sigma)-\lfloor B''\rfloor\ge 0.
\]
It is enough to check over a Zariski big open set $U''$ of $Z''$, since then we will have
\[
f''_*\OO_{X''}(-\lfloor B''\rfloor)\subseteq (f''_*\OO_{X''}(-\lfloor B''\rfloor))^{\vee\vee}=\OO_{Z''}(-\lfloor B_{Z''}\rfloor).
\]
We may assume $f''$ is flat over $U''\supset U$ so that $f''(E_{i,j})=D_i$ for any $f''(E_{i,j})\subseteq D_i$ (over $U''$). Over $U$ this is clear, assume the prime components of $\Delta$ are $D_i, i \in I$, let $f''^*D_i=\sum_j m_{i,j} E_{i,j}$. Next we claim that (over $U''$)
\begin{claim}
$\OO_{Z''}(-\sum_i e_i D_i)= f''_*\OO_{X''}(-\lfloor B''\rfloor)$, where
\[
e_i:=\max_{f''(E_{i,j})=D_i}\{\lceil\frac{\mult_{E_{i,j}}(\lfloor B''\rfloor)}{m_{i,j}}\rceil\}
\]
\end{claim}
\begin{proof}
One can see $\OO_{Z''}(-\sum_i e_i D_i)\subseteq f''_*\OO_{X''}(-\lfloor B''\rfloor)$ easily. For any local section $\sigma\in f''_*\OO_{X''}(-\lfloor B''\rfloor)\subset\mathfrak{M}(Z'')$, we have $\Div(\sigma)^{\le 0}$ must be supported on $\Delta$ and
\[
m_{i,j}\cdot\mult_{D_i}(\Div(\sigma))\ge \mult_{E_{i,j}}(\lfloor B''\rfloor),~\forall i,j
\]
in particular $\Div(\sigma)\ge \sum_i e_i D_i$, i.e. $\sigma\in\OO_{Z''}(-\sum_ie_iD_i)$.
\end{proof}

Therefore we just need to show $\lfloor \sum_i c_i D_i\rfloor=\lfloor B_{Z''}\rfloor=\sum_i e_iD_i$. Recall again from the canonical bundle formula that
\[
\lfloor c_i\rfloor=\max_{f''(E_{i,j})=D_i}\{\lfloor\frac{\mult_{E_{i,j}}(B'')-1}{m_{i,j}} \rfloor+1\}.
\]
For any $a\in \R,m\in \mathbb{Z}^+$, the following equality holds
\[
\lfloor\frac{a-1}{m} \rfloor+1=\lceil \frac{\lfloor a\rfloor}{m}\rceil.
\]
Hence we get 
\[
\lfloor c_i\rfloor=\max_{f''(E_{i,j})=D_i}\{\lceil\frac{\lfloor\mult_{E_{i,j}}(B'')\rfloor}{m_{i,j}}\rceil\}=e_i,~\forall i.
\]
To summarize, we just showed that $f''_*\OO_{X''}(-\lfloor B''\rfloor)=\OO_{Z''}(-\lfloor B_{Z''}\rfloor)$, which implies 
\[
f_*\OO_{\text{NKLT}(X,B+\bbeta)}=\OO_{\text{NKLT}(Z,B_Z+\ddelta)}.
\]
Now by Lemma \ref{l-contraction} we know $[\alpha_Z]|_{{\rm NKLT}(Z,B_Z+\ddelta)}$ is endowed with a contraction, and then by Theorem \ref{t-nkltcontraction}$_{n-1}$ we know $\alpha_Z$ is endowed with a contraction. Hence $\alpha$ is also endowed with a contraction since $f$ is a projective contraction.
\end{proof}

\section{Applications}\label{s-application}
In this section, we prove Theorem \ref{t-KahlerBCHM} and several of its consequences.
\begin{proof}[Proof of Theorem \ref{t-KahlerBCHM}] If $K_X+B+\bbeta _X$ is not pseudo-effective, then there exists a K\"ahler form on $X$ such that $K_X+B+\bbeta _X+\omega$ is not pseudo-effective. Replacing $\bbeta$ by $\bbeta +\omega$, we may assume that (in both cases (1) and (2)) $B+\bbeta _X$ is modified big. By Theorem \ref{t-small Q factorial model} and Lemma \ref{l-modbig}, we may assume that $X$ is strongly $\Q$-factorial.
The claim now follows easily from Theorem \ref{t-MMPscaling}.
\end{proof}
\begin{theorem}[K\"ahler criterion]\label{t-Kahler} Let $(X,B+\bbeta)$ be a $n$-dimensional compact generalized pair such that $\alpha =K_X+B+\bbeta_X$ is NQC, $B+\bbeta _X$ is modified big, and $[\alpha]|_V$ is K\"ahler, where $V={\rm NKLT}(X,B+\bbeta)$ with the reduced structure. 
If $\alpha \cdot C>0$ for every curve $C\subset X$, then $\alpha$ is K\"ahler.
\end{theorem}

\begin{proof}
By \cite{HLX26} it suffices to show that $\alpha$ is big. Suppose that $\alpha$ is not big, and consider a log resolution $f: X'\to X$ such that $X'$ is a compact K\"ahler manifold. If we write
\[
K_{X'}+B'+\bbeta_{X'}=f^*(K_X+B+\bbeta_X)
\]
 and  $B'=B'^{\ge 0}-B'^{\le 0}$, then $B'^{\le 0}$ is an $f$-exceptional divisor. Let $\alpha '=f^*\alpha$, then $\alpha'+E$ is not big for any $f$-exceptional divisor $E$. Since $B+\bbeta_X$ is modified big, we see that
\begin{itemize}
\item $K_{X'}+t\alpha'$ is not pseudo-effective for any $t\ge 0$, and
\item $[\alpha']$ is also NQC as $\alpha$ is NQC by assumption.
\end{itemize}
Then by Lemma \ref{lem: NQC can run alpha-trivial MMP} there exists a $t_0>0$ such that any $(K_{X'}+t_0\alpha')$-MMP is $\alpha'$-trivial. By Theorem \ref{t-MMPscaling} we know there is a $(K_{X'}+t_0\alpha')$-MMP with scaling which terminates with a Mori fiber space $X'\to Z$, then $X'$ is covered by $\alpha'$-trivial rational curves, and so $X$ is covered by $\alpha$-trivial rational curves. But this contradicts our assumption.
\end{proof}

\begin{remark}

Theorem \ref{t-Kahler} is known to be true even without the NQC assumption when $\alpha$ is big (cf.\cite{HLX26}). In general it will be implied by the termination of flips. In fact, if $\alpha$ is not big, we can consider a smooth K\"ahler resolution $f:X'\to X$ and then we run an $f^*\alpha$-trivial $K_{X'}$-MMP, which must terminate with a Mori fiber space $X''\to Z$, hence there is a dominant family of $\alpha$-trivial rational curves on $X''$. The strict transforms of these curves give a dominant family of $\alpha$-trivial rational curves on $X$. 

\end{remark}

\begin{theorem}\label{t-Fujiki bpf}
Let $(X,B+\bbeta)$ be a compact gklt pair such that $B+\bbeta _X$ is modified big and $\alpha=[K_X+B+\bbeta_X]$ is nef, then there exists a birational map $\phi: X\dasharrow Y$ such that
\begin{enumerate}
\item $Y$ is K\"ahler and strongly $\Q$-factorial,
\item $\phi$ is small, so $\alpha_Y:=\phi_*\alpha\in H^{1,1}_{\rm BC}(Y)$,
\item $\phi$ is crepant, i.e. $\alpha$ and $\alpha'$ are equal up to pullbacks on a common resolution.
\end{enumerate}
In particular, there is a Moishezon contraction $g: X\to Z$ such that $\alpha=g^*\omega_Z$ for a K\"ahler class $\omega_Z$ on $Z$.
\end{theorem}
\begin{proof}
From the assumption we know $X$ is in the Fujiki's class $\mathcal{C}$, there exists a log resolution $f: X'\to X$ such that $X'$ is K\"ahler and $\bbeta$ descends on $X'$. Let $\Delta:=f^{-1}_*B+(1-\epsilon){\rm Ex}(f)$ for some $0<\epsilon\ll 1$, then we can write 
\[K_{X'}+\Delta+\bbeta_{X'}=f^*(K_X+B+\bbeta_X)+F,\]
where $F$ is an effective $f$-exceptional divisor such that $\Supp (F)={\rm Ex}(f)$. Then by Lemma \ref{l-modbig} we know $\Delta+\bbeta_{X'}$ is big, hence $(X',\Delta+\bbeta)$ has a good log terminal model $\phi': (X',\Delta+\bbeta)\dasharrow (Y,\Delta_Y+\bbeta)$ by Theorem \ref{t-KahlerBCHM}. We may assume $Y$ is K\"ahler and strongly $\Q$-factorial by Theorem \ref{t-MMPscaling} and Proposition \ref{prop: MMP can be run}.

Since $K_X+B+\bbeta_X$ is nef, we have $N(K_{X'}+\Delta+\bbeta_{X'})=F$ by \cite[Lemma A.5]{DHY23}, then by Lemma \ref{l-contdiv} we know $\Supp F={\rm Ex}(f)$ is exactly the set of divisors contracted by $\phi'$, which induce a small birational map $\phi: X\dasharrow Y$ and we have $\phi_*B=\Delta_Y$.

Now $\phi_*[K_X+B+\bbeta_X]=[K_Y+\Delta_Y+\bbeta_Y]\in H^{1,1}_{\rm BC}(Y)$ is also nef, therefore by the negativity lemma we can easily see that $\phi:(X,B+\bbeta)\dasharrow (Y,\Delta_Y+\bbeta)$ is crepant. By Theorem \ref{t-bpf} $\alpha_Y=\phi_*\alpha$ is endowed with a Moishezon contraction $g_Y:Y\to Z$ and $g_Y^*\omega_Z=\alpha_Y$ for some K\"ahler class $\omega_Z$, so $\alpha$ is endowed with a Moishezon contraction $g: X\to Z$ as well by Lemma \ref{l-contraction}, and we have $g^*\omega_Z=\alpha$ since $\alpha,\alpha_Y$ are crepant.
\end{proof}

\begin{corollary}
Let $(X,B+\bbeta)$ be a compact gklt pair such that $B+\bbeta_X$ is modified big and $\alpha=[K_X+B+\bbeta_X]$ is nef, then $\alpha$ is semiample.
\end{corollary}

\begin{corollary}\label{c-cyrc}  Let $(X,B+\bbeta )$ be a compact gklt pair such that $B+\bbeta _X$ is modified big and $K_X+B+\bbeta _X\equiv 0$, then $X$ is rationally connected.
\end{corollary}
\begin{proof}
Since $(X,B+\bbeta)$ is gklt, $X$ is rationally connected iff a resolution of $X$ is rationally connected. By Theorem \ref{t-Fujiki bpf} we may assume $X$ is K\"ahler and strongly $\Q$-factorial. By assumption $-K_X\equiv B+\bbeta_X$ admits a K\"ahler current, hence $X$ is Moishezon. By Lemma \ref{L-gklt-rational} we know $X$ has only rational singularities, then $X$ is actually projective by \cite{Nam02}. 

Arguing as in Lemma \ref{lem: g-kltperturb} we can show that there exists a gklt pair $(X,B'+\bbeta')$ and a K\"ahler form $\omega$ on $X$ such that 
\[
K_X+B'+\bbeta'_X+\omega\equiv 0. 
\]
By the vanishing theorem \ref{t-kvv} we have $H^i(X,\OO_X)=0,i>0$, which also holds on any resolution $X'\to X$ via the Leray spectral sequence. Hence $\bbeta'_{X'}\equiv N$ for a divisor $N$ and we may assume $(X,B'+\bbeta')$ is a usual log Fano generalized pair, where $\bbeta'$ is a b-nef b-divisor. One can easily show that $X$ is of log fano type and then the result is standard (see eg. \cite{HM07}).
\end{proof}

\begin{theorem} Let $(X,B+\bbeta)$ be a compact K\"ahler gklt pair such that either $K_X+B+\bbeta _X\in H^2(X,\Q)$ or $B+\bbeta _X$ is modified big. If $\phi _i:X\dasharrow X_i$ are log terminal models of $(X,B+\bbeta _X)$ for $i=1,2$, then there exists a sequence of $K_{X_1}+B_1+\bbeta _{X_1}$ flops $X_1\dasharrow X_2$ connecting $X_1$ and $X_2$. 
\end{theorem}
\begin{proof}  Note that in both cases $K_X+B+\bbeta _X$ is NQC, see Lemma \ref{lem: gklt pair+Kahler is NQC}. Then the rest follows from the proof of \cite[Theorem 3.28]{DHY23}.

\end{proof}
The next result shows that the MMP is well behaved in families over a curve (see \cite[Proposition 4.16]{HLR26} for the case when the central fiber is projective).
\begin{theorem}\label{t-PHLR} Let $g: X \rightarrow S$ be a flat proper morphism with connected fibers from a generalized pair $(X,B+ \bbeta)$ to a smooth, connected, and relatively compact curve. Fix a point $0\in S$ and assume that the support of the boundary divisor $B$ does not contain the fiber $X_0$. Suppose that $(X_0,B_0 + { \boldsymbol{\beta}}_0)$ is K\"ahler, with canonical singularities, and $\lfloor B_0 \rfloor = 0$, where $$(K_X+X_0+B+\bbeta _X)|_{X_0}=K_{X_0}+B_0+\bbeta _0,$$ and the negative part of Boucksom--Zariski decomposition satisfies the relation $$N(K_{X_0}+ B_0 + \bbeta_0) \wedge B_0 = 0.$$ Then every sequence of transcendental MMP-steps \[(X_0, B_0  +\bbeta_0) \dashrightarrow (X_0^{(1)}, B_0^{(1) }+  \bbeta_0^{(1)}) \dashrightarrow (X_0^{(2)},B_0^{(2)}+ \bbeta_0^{(2)}) \dashrightarrow \cdots\]  extends to a sequence of $(K_X+B+\bbeta _X)$-negative proper meromorphic maps 
$$ (X, B+ \bbeta)/U \dashrightarrow (X^{(1)},B^{(1)}+ \bbeta^{(1)})/U \dashrightarrow (X^{(2)},B^{(2)}+ \bbeta^{(2)}) /U\dashrightarrow \cdots,$$over some open neighborhood $U\subset S$ of $0$. 
\end{theorem}
\begin{proof} The proof follows by the arguments of \cite[Proposition 4.16]{HLR26} where instead of \cite{DH24}, we use Theorem \ref{t-KahlerBCHM}.
\end{proof}
We also have the following immediate consequence (\cite[Conjecture 1.4]{HLR26}).
\begin{corollary} Let $f : X \to S$ be a family from a K\"ahler variety $X$ onto a smooth,
connected, and relatively compact curve $S$. Assume that the fibers $X_t$ have canonical singularities
for all $t \in S$. Then $K_{X_0}$ is pseudo-effective if and only if $K_{X_t}$ are pseudo-effective for all
$t \in S \setminus \{ 0\}$, and $X_0$ is uniruled if and only if $X_t$ are uniruled for all $t \in S \setminus \{ 0\}$.
\end{corollary}
\begin{proof} The proof follows by the same arguments of the proof of \cite[Corollary 1.3]{HLR26}, where instead of \cite[Proposition 4.16]{HLR26} we rely on Theorem \ref{t-PHLR} and instead of \cite{DH24}, we rely on Theorem \ref{t-KahlerBCHM}.
\end{proof}

\end{document}